\documentclass[12pt,a4paper, reqno]{amsart}
\usepackage[utf8]{inputenc}
\usepackage[english]{babel}
\usepackage{amsmath,mathrsfs}
\usepackage{amsfonts}
\usepackage{amssymb}
\usepackage{hyperref}
\usepackage{amsxtra,latexsym,amscd,amsthm,marvosym,multirow,mathtools}
\usepackage[a4paper,inner=2.3cm, outer=2.3cm, top=2.3cm, bottom=2.3cm]{geometry}
\usepackage{setspace}
\usepackage[pdftex]{graphicx}
\usepackage{framed, fancybox}
\usepackage{epstopdf}
\usepackage{subcaption} 
 
\usepackage{enumerate}
\usepackage{csquotes}
\usepackage[shortlabels]{enumitem}
\usepackage{cases}

\usepackage{cleveref}

\newtheorem{theorem}{Theorem}[section]
\newtheorem{lemma}[theorem]{Lemma}
\newtheorem{notation}[theorem]{Notation}
\newtheorem{remark}[theorem]{Remark}
\newtheorem{proposition}[theorem]{Proposition}

\newtheorem{example}{Example\/}

\renewcommand{\Re}{\mathrm{Re}\,}
\renewcommand{\Im}{\mathrm{Im}\,}
\newtheorem{assumpA}{Assumption}

\def\beq{\begin{eqnarray*}}\def\eeq{\end{eqnarray*}}
\def\bq{\begin{equation}}\def\eq{\end{equation}}

\newcommand{\N}{\mathbb{N}}
\newcommand{\Z}{\mathbb{Z}}
\newcommand{\R}{\mathbb{R}}
\newcommand{\C}{\mathbb{C}}

\newcommand{\eps}{\varepsilon}
\newcommand{\dd}{\mathrm{d}}
\newcommand{\Dom}{\mathcal{D}}

\newcommand{\lV}{\left\Vert}
\newcommand{\rV}{\right\Vert}

\makeatletter

\@addtoreset{equation}{section}  
\makeatother

\begin{document}
\title{Pseudomodes for biharmonic operators with complex potentials}
\author{Tho Nguyen Duc}
\address[Tho Nguyen Duc]{Department of Mathematics, Faculty of Nuclear Sciences and Physical Engineering,
Czech Technical University in Prague, Trojanova 13, 12000 Prague 2, Czech Republic}
\email{nguyed16@fjfi.cvut.cz}
\begin{abstract}
This article is devoted to the construction of pseudomodes of one-dimensional 
biharmonic operators with the complex-valued potentials via the WKB method. As a by-product, the shape of pseudospectrum near infinity can be described. This is a newly discovered systematic method that goes beyond the standard semi-classical setting which is a direct consequence. This approach can cover a wide class of previously inaccessible potentials, from logarithmic to superexponential ones. 
\end{abstract}
\maketitle
\section{Introduction}
\subsection{Context and motivations}
In the well-known self-adjoint theory, the norm of the resolvent associated with a self-adjoint operator is very large if and only if the spectral parameter comes close to the spectrum. The picture is very different for the non-self-adjoint operator: the resolvent may blow up  when the spectral parameter is far from the spectrum. It leads to the instability of the spectrum under a small perturbation and reveals that, in this case, the numerical methods will fail to compute the eigenvalues. In order to describe this pathological property of the non-self-adjoint operators, the notion 
of pseudospectra was regarded
\cite{LM05,Davies_2007,Helffer}. That is, given a positive number~$\eps$,
the $\eps$-\emph{pseudospectrum} $\sigma_\eps(H)$ 
of an operator~$H$ in a complex Hilbert space
is defined as its spectrum $\sigma(H)$ along with 
those resolvent points whose norm of the resolvent larger than $\varepsilon^{-1}$. This definition is equivalent to the description of the set $\sigma_\eps(H)$ as the spectrum $\sigma(H)$ enlarged by the complex points $\lambda$ (called \emph{pseudoeigenvalues}) for which 
there exists a vector $\Psi \in \Dom(H)$ such that
\begin{equation}\label{Pseudo Equation}
\Vert (H-\lambda)\Psi \Vert 
< \varepsilon \, \Vert \Psi\Vert\,.
\end{equation}
Any $\Psi$ satisfies \eqref{Pseudo Equation} is called a \emph{pseudomode} (or its other names: \emph{pseudoeigenfunction}, \emph{pseudoeigenvector}, \emph{quasimode}). 

The analysis of pseudospectra of the non-self-adjoint Schr\"{o}dinger operators has given rise to many investigations in twenty years \cite{Da99,Hitrik-Sjostrand-Viola_2013,
Henry,Henry_2014,Novak_2015,KSTV15,HK17,KS19,CK20}. By change of scale to transform the original Schr\"{o}dinger operator to its semi-classical form, which has the small parameter $h^2$ in front of the second derivative, Davies' pioneering work \cite{Da99} has constructed the pseudomode for the semi-classical Schr\"{o}dinger operator as $h\to 0$, then the pseudomode of the original operator with large $\lambda$ is achieved (\emph{i.e.}\ those corresponding in~\eqref{Pseudo Equation} to $|\lambda|\to+\infty$ with $\eps_\lambda \to 0$). However, this method seems merely effective for a certain class of polynomial potentials in which the scaling is able to be performed, while it is inapplicable, for example, for logarithmic or exponential potentials. Moreover, the semi-classical approach is even more inaccessible to   the class of discontinuous potentials since the semi-classical construction of pseudomode requires that the potential is smooth (or at least continuous). In \cite{HK17}, Rapha\"{e}l and Krej\v{c}i\v{r}\'{i}k had studied the imaginary sign potential by constructing the resolvent kernel of the Schr\"{o}dinger operator. Up to twenty years after the earlier Davies' work~\cite{Da99},  Krej\v{c}i\v{r}\'{i}k and Siegl in~\cite{KS19} developed 
a \emph{direct construction of large-energy pseudomodes} for Schr\"{o}dinger operator, which does not require the passage through semi-classical setting and can cover all mentioned above potentials. Recently, this technique is applicable to other models such as the damped wave equation~\cite{AS20} and  Dirac operator~\cite{KN20}. The purpose of the present paper is to extend the method developed in~\cite{KS19,KN20} to higher differential operator by considering biharmomic instead of Schr\"odinger operators and to discover the more universal shape of the potentials such that this method works. 

More precisely, this document is devoted to construct a $\lambda$-dependent family of pseudomodes $f_{\lambda}$ such that
\begin{equation}\label{Eq Intro Main Prob}
\Vert \left(\mathscr{L}_{V}-\lambda\right) f_{\lambda} \Vert  = o(1) \Vert  f_{\lambda} \Vert,  \qquad \text{ as } \lambda \to \infty \text{ in } \Omega \subset \C.
\end{equation}
Here $\mathscr{L}_{V}$ is the one-dimensional biharmonic operator which is defined as a forth derivative perturbed by a complex-valued potential $V$: 
\[ \mathscr{L}_{V}= \frac{\dd^4}{\dd x^4}+V(x). \]
In \cite{Boulton02}, when the author work with the non-self-adjoint harmonic oscillator, it has shown us that we do not always obtain the decay in \eqref{Eq Intro Main Prob} by just letting $\lambda \to \infty$ in $\C$. Indeed, he proves in \cite{Boulton02} that the norm of the resolvent is bounded above when $\lambda \to \infty$ along some half-lines parallel to the positive semi-axis $\R^{+}$ and blows up when $\lambda \to \infty$ in a region bounded by two certain curves. The set $\Omega$ is a region in the neighborhood of infinity in $\C$ which contains large $\lambda$ allowing the decay in~\eqref{Eq Intro Main Prob} to happen and thus the norm of the resolvent will go up in this region. 

There are three main theorems in this paper. The first theorem will provide the answer to the question: What is the behaviour of $V$ such that we obtain the decay in \eqref{Eq Intro Main Prob} when $\lambda \to \infty$ in the region parallel to the positive semi-axis? We will see that the behaviour at infinity of the imaginary part of $V$, denoted by $\Im V$, plays the decisive role in this mission, that is
\[ \left(\lim_{x \to -\infty} \Im V(x)\right) \left(\lim_{x \to +\infty} \Im V(x)\right)<0.\]
This condition is essential to ensure the \enquote{significantly non-normality} of $H_{V}$. Theorem \ref{Theorem 1} also generalizes the results in \cite[Thm. 3.7]{KS19} and \cite[Thm. 3.10,\ Thm. 3.11]{KN20} in which only $\lambda \to +\infty$ on the positive semi-axis is considered. Furthermore, the class of the potential $V$ is also widened in our paper by controlling the derivative of $V$ by general functions $\tau_{\pm}$ (see Cond. \eqref{Assump G Der}) instead of some polynomial functions in \cite[Cond. (3.2)]{KS19}. This allows us to cover the super-exponential function (see Example \ref{Example Superexp}) that could not be covered in \cite{KS19}.

Theorem \ref{Theorem 2} addresses the question: What is the shape of $\Omega$ for each type of the potential~$V$? Once again, our assumption can cover a larger class of potentials than~\cite[Asm. III]{KS19} and \cite[Asm. II]{KN20}. By applying it for the functions from growing slowly at $+\infty$ such as logarithmic one $V(x)=i \ln(x)$ or root-type $V(x)=ix^{\gamma}$ with $\gamma\in(0,1)$ to growing faster at $+\infty$ such as polynomial $V(x)=ix^{\gamma}$ with $\gamma\geq 1$ and super-exponential functions $V(x)=ie^{e^{x}}$, the region $\Omega$ for each type of them is described differently (see Subsection \ref{Subsection App Theorem 2} in which many picture for illustrations are added). Finally, Theorem \ref{Theorem 3} shows us that the semi-classical setting is actually a special consequence of our pseudomode construction. In all of the  theorems mentioned above, the regularity of the potentials are assumed as mild as possible, which is a small plus point compared with the previous semi-classical setting that the smoothness of $V$ always assumed highly. 

The technique that we employed to find the pseudomodes is the (J)WKB method (also known as the Liouville--Green approximation). The WKB method is not only seen as a tool to approximate the eigenfunction for some differential operators, but it also reveals some information of the eigenvalues. For instance, in \cite{YN20,BNRVN20}, the asymptotic expansions for the eigenvalues of the self-adjoint magnetic Laplacian were found in the process of doing the WKB analysis. Now we could ask a similar question for the non-self-adjoint operator
\begin{center}
\enquote{Can we describe $\Omega$ (pseudo-spectrum near the infinity) by the WKB method?}
\end{center}

Continuing the work of \cite{KS19,KN20}, we would like to provide a positive answer by considering the higher differential operator with more general potentials. We refer~\cite[Remark~2.3]{KN20} to explain why our approach goes beyond the standard semi-classical settings. It is our belief that the study in this article is a necessary step allowing us to approach in the future with more general differential operators such as $\frac{\dd^{n}}{\dd x^n}+V(x)$ for arbitrary $n\in \N_{1}$ and a complex electric potential $V$. The extension of this analysis to the magnetic Laplacian with a complex magnetic field (recently motivated in \cite{David19}) also constitutes a challenging open problem.

Biharmonic operator has its own application in continuum mechanics and in linear elasticity theory. In recent years, the biharmonic operator attracted considerable attention in particular in the context of spectral theory. For instance, in \cite{Enblom16}, Enblom study the bound of the eigenvalues for the non-self-adjoint polyharmonic operator in which the biharmonic operator is a special case. The sharp confinement of the eigenvalues in a closed disk of the biharmonic operators has been produced in \cite{IKL21} for low dimension. We also list here some recent studies related to the spectral properties of the biharmonic operators \cite{AFL17,Hulko18,BL19,FL21}.
\subsection{Handy notations and conventions}\label{Handy notation}
Here we summarise some special notations and conventions which we use regularly in the paper: 
\begin{enumerate}[label=\arabic*)]
\item $\N_{k}$, with a non-negative integer $k$,  
is the set of integers starting from $k$;
\item For the semi real axes, we denote $\R_{0}^{-}\coloneqq (-\infty,0]$ and $\R_{0}^{+}\coloneqq [0,\infty)$, and for strictly positive or negative axes, we denote $\R^{-}\coloneqq(-\infty,0)$ and $\R^{+}\coloneqq(0,+\infty)$;
\item For the list of integer numbers from $m$ to $n$, where $m,n\in \Z$ and $m<n$, we denote~$[[m,n]]$, \emph{i.e.} $[[m,n]]:=\left\{ k\in \Z: m\leq k\leq n\right\}$ ;
\item $f^n$ and $f^{(n)}$ denotes respectively
the power $n$ and the $n$-th derivative of a function $f:\R \to \C$ with $n\in \N_{0}$;
\item 
We use the same symbol $\Vert \cdot \Vert:=\lV \cdot \rV_{L^{2}(\R)}$
for $L^2$-norms of complex-valued functions defined on $\R$;
\item We often write $\lambda=\alpha+i\beta$ where $\alpha,\beta\in \R$ and denote $V_{\lambda}\coloneqq\lambda-V$;
\item For two real-valued functions $a$ and $b$, we write
$a\lesssim b$ (respectively, $a \gtrsim b$)  
if there exists a constant $C>0$, independent of $\lambda$ and $x$ (or any other relevant parameter such as $\alpha$ and $\beta$), such that
$a\leq C b$  (respectively, $a\geq Cb$); and we write $a \approx b$ if $a\lesssim b$ and $a \gtrsim b$.
\end{enumerate}
\subsection{Structure of the paper}
The rest of this article is organized as follows: In Section \ref{Section Statement}, we establish the main conditions for the admissible class of the potentials $V$ and the main theorems related to the central problem \eqref{Eq Intro Main Prob} come shortly after. Many nontrivial and illuminating examples are contained in this section. We reserve Section \ref{Section WKB} to describe the WKB method for the biharmonic operators, which is the main tool to construct the pseudomodes. Section \ref{Section Proof of theorem 1} is devoted to the large real pseudoeigenvalues in the region parallel to the positive semi-axis $\R^{+}$ while the pseudoeigenvalues corresponding to the large imaginary part are dealt with in Section \ref{Section Proof Theorem 23}. The method in Section \ref{Section Proof Theorem 23} is also used to prove the semi-classical result.
\section{Statements, results and applications}\label{Section Statement}
\subsection{Statements and results}
In this article, we consider the maximal forth-order differential operator perturbed by a multiplication operator, where we denote by the same symbol $V$, as follows
\begin{align*}
&\mathrm{Dom}(\mathscr{L}_{V}) = \left\{u \in L^2(\R):  \left(\frac{\dd^4}{\dd x^4}+V(x) \right)u \in  L^2(\R)\right\},\\
&\mathscr{L}_{V}= \frac{\dd^4}{\dd x^4}+V(x),
\end{align*}
where $V:\R \to \C$ is assumed to be locally $L^2$-integrable, \emph{i.e.} $V\in L^{2}_{\mathrm{loc}}(\R)$. This condition ensures that all the action of $\mathscr{L}_{V}$ is well-defined in the sense of distributions. It follows that $\mathscr{L}_{V}$ is a closed operator. However, the closedness of $\mathscr{L}_{V}$ is inessential for our construction of pseudomode. Since the pseudomode constructed in this document has a compact support, our method can be modified to adapt with any closed extension of the operator initially defined on $\mathcal{C}_{0}^{\infty}(\R)$.

Since $\lambda=\alpha+i \beta$, there are two obvious ways to make $\lambda$ become largely, that is increasing $\alpha$ or increasing $\beta$. Therefore, depending on the decisive role of the real part or imaginary part of the pseudoeigenvalue in the decaying estimation of \eqref{Eq Intro Main Prob}, the pseudomodes are accordingly constructed in different ways. 
\subsubsection{Large real pseudoeigenvalues}
Let us state the main assumptions on the admissible class of the potential~$V$. 
\begin{assumpA}\label{Assumption 1}
Let $N\in \N_{0}$, assume that $V\in W^{N+3,\infty}_{\mathrm{loc}}(\R)$ satisfy the following conditions:
\begin{enumerate}[label=\textbf{\arabic*)}]
\item   $\Im V$ has a different asymptotic behaviour at $\pm \infty$:
\begin{equation}\label{Assump G Im V 1}
\limsup_{x\to-\infty} \Im V(x)<0 <\liminf_{x\to +\infty} \Im V(x);
\end{equation} 
\item  There exist continuous functions $\tau_{\pm}:\R_{0}^{\pm} \to \R^{+}$ such that, for all $n \in [[1,N+3]]$,
\begin{equation}\label{Assump G Der}
\begin{aligned}
\left\vert V^{(n)}(x) \right\vert =\mathcal{O}\left( \tau_{\pm}(x)^{n}\left\vert V(x)\right\vert\right) , \qquad x\to \pm \infty;
\end{aligned}
\end{equation}
\item Additional assumptions for $\tau_{\pm}$ and $V$ in each of the following cases:
\begin{enumerate}[label=\textbf{\alph*)}]
\item If $V$ is unbounded at $\pm \infty$, assume that $\tau_{\pm}^{(1)}$ exist for $\vert x \vert\gtrsim 1$ and there exist $\nu_{\pm}\geq -1$ for which
\begin{equation}\label{Tau}
\vert x \vert^{\nu_{\pm}}= \mathcal{O}\left(\tau_{\pm}(x)\right),\qquad\tau_{\pm}^{(1)}(x)=\mathcal{O}\left(\vert x \vert^{\nu_{\pm}} \tau_{\pm}(x)\right),\qquad x\to \pm \infty
\end{equation}
such that
\begin{equation}\label{V1}
\vert \Im V^{(1)}(x) \vert =\mathcal{O}\left(\tau_{\pm}(x)\left\vert  \Im V(x)\right\vert\right) , \qquad x\to \pm \infty,
\end{equation}
and assume further that: $\exists \varepsilon_{1} >0$,
\begin{equation}\label{Assump Bound Re Im}
\tau_{\pm}^{4}(x)\left(\tau_{\pm}^{12+4\varepsilon_{1}}(x)+\left\vert \Re V(x) \right\vert^{3+\varepsilon_{1}}\right)=\mathcal{O}\left( \left\vert \Im V(x) \right\vert^{4}\right), \qquad x\to \pm \infty;
\end{equation}
\item If $V$ is bounded at $\pm \infty$, then assume that: $\exists \varepsilon_{2}\in \left(0,\frac{1}{3}\right)$,
\begin{equation}\label{Assump Bound Above tau}
\tau_{\pm}(x)=\mathcal{O}\left(\vert x \vert^{\frac{1}{3}-\varepsilon_{2}}\right),\qquad x\to \pm \infty.
\end{equation}
\end{enumerate} 
\end{enumerate}
\end{assumpA}
Notice that the constant $\varepsilon_{1}$ in \eqref{Assump Bound Re Im} shall be considered sufficiently small. Since optimization this constant is not interesting, we are not trying to do that in our work. 

Next lines are some comments on Assumption \ref{Assumption 1}. By comparing with the same assumption for the potential in Schr\"{o}dinger operator \cite[Asm. I]{KS19}, the regularity of $V$ in the biharmonic operator requires to be higher, this can be seen clearly in the formula of the remainders $\mathcal{R}_{\lambda,0}$ in \eqref{R0} and in \cite[Eqn. 2.18]{KS19}. These remainders are the amount left over after performing WKB construction (see Section \ref{Section WKB}). As mentioned in the introduction, the assumption \eqref{Assump G Im V 1} makes the operator highly non-self-adjoint. Indeed, if $\left(\mathscr{L}_{V}\right)^{*}$ is the formal adjoint of $H_{V}$, \emph{i.e.} $\left(\mathscr{L}_{V}\right)^{*}=\mathscr{L}_{\overline{V}}$, it is straightforward to verify (at least algebraically) that the normality relation $\mathscr{L}_{V}\left(\mathscr{L}_{V}\right)^{*}=\left(\mathscr{L}_{V}\right)^{*}\mathscr{L}_{V}$ holds iff $\Im V^{(1)}=0$, \emph{i.e.} $\Im V$ is a constant. The sign of $\Im V$ in \eqref{Assump G Im V 1} will determine the sign for the decay of the pseudomode. The larger $\Im V$ is, the faster the pseudomode decreases at infinity, see the estimate \eqref{Exponential P 1}. Furthermore, we can invert the sign of $\Im V$ at infinity and the construction of pseudomode does not change more, see Remark \ref{Remark sign}. The condition \eqref{Assump G Der} is designed exclusively for the very special shapes of transport solutions and the remainder that will be described in the next section. To control the too large $\Re V$ and any wild behaviour of the derivatives of $V$, the conditions \eqref{Assump Bound Re Im} and \eqref{Assump Bound Above tau} are employed, furthermore, the natural  imposition of the assumption  \eqref{Assump Bound Re Im} on $\Re V$ is also discussed in Remark \ref{Remark Re V}, which shows that this condition is optimal in the polynomial case. Finally, two assumptions \eqref{Tau} and \eqref{V1} are technical tools to guarantee that the values of $\tau_{\pm}$ and $\Im V$ on some suitable interval can be comparable up to a constant (see \eqref{Eq v assymptotic}). This technique has been lately used very much, for instance, in \cite{KS19,AS20,KN20,MSV20}, however, in these papers, they fix the functions $\tau_{\pm}(x) \coloneqq \vert x \vert^{\nu}$ for some $\nu\geq -1$.
Here, we provide an improvement of this technique by controlling by general functions $\tau_{\pm}$ satisfying~\eqref{Tau}. For example, $\Im V(x)=e^{e^x}$ can be covered by our assumption with $\tau(x)=e^{x}$, but it is not allowed by the assumption in the papers mentioned above.

In order to state our first theorem, we denote the interval $B:=[-\beta_{-},\beta_{+}]$ where $\beta_{\pm}$ are non-negative constants  satisfying
\begin{equation}\label{Box B}
\limsup_{x\to -\infty} \Im V(x)<-\beta_{-}\qquad \text{ and } \qquad \beta_{+}< \liminf_{x\to +\infty} \Im V(x).
\end{equation}
\begin{theorem}\label{Theorem 1}
Let Assumption \ref{Assumption 1} hold for some $N\in \N_{0}$. Then there exists a $\lambda$-dependent family $\left(\Psi_{\lambda,N}\right)\subset \textup{Dom}(H_{V})$ such that for all $\alpha\gtrsim 1$ and $\beta \in B$, we have
\begin{equation}\label{Eq Main 1}
\frac{\Vert (H_{V}- \lambda)\Psi_{\lambda,N} \Vert}{\Vert \Psi_{\lambda,N} \Vert} \lesssim \alpha^{-\frac{N+1}{4}}+\sigma_{-}^{(N)}(\alpha)+\sigma_{+}^{(N)}(\alpha),
\end{equation}
where
\[\sigma_{-}^{(N)}(\alpha)\coloneqq\alpha^{-\frac{N+1}{4}}\displaystyle \sup_{x\in [-\delta_{\alpha}^{-},0]}\tau_{-}(x)^{N+1}\vert V(x)\vert,\qquad \sigma_{+}^{(N)}(\alpha)\coloneqq\alpha^{-\frac{N+1}{4}}\displaystyle \sup_{x\in [0,\delta_{\alpha}^{+}]}\tau_{+}(x)^{N+1}\vert V(x)\vert,\]
in which  $\delta_{\alpha}^{\pm}$ are defined as follows:
\begin{enumerate}[label=\textbf{\alph*)}]
\item If $V$ is unbounded at $\pm \infty$, $\delta_{\alpha}^{\pm}$ are the smallest positive solutions of the equations 
$$\frac{\left\vert\Im V(\pm x)\right\vert}{\tau_{\pm}(\pm x)}=\alpha^{\frac{3+\varepsilon_{1}}{4+\varepsilon_{1}}}.$$
\item If $V$ is bounded at $\pm \infty$, $\delta_{\alpha}^{\pm}=\alpha^{\frac{3}{4-3\varepsilon_{2}}}$ .
\end{enumerate}
In particular, if $V$ and $\tau_{\pm}$ is  bounded at $\pm \infty$, then
\begin{equation}\label{sigma N}
\sigma_{\pm}^{(N)}(\alpha)\lesssim \alpha^{-\frac{N+1}{4}}.
\end{equation}
\end{theorem}
Although the right hand side of \eqref{Eq Main 1} does not show us the decay obviously, Theorem~\ref{Theorem 1} are very workable in many elementary cases such as logarithmic functions, polynomials, exponential functions and even super-exponential (see its application in Subsection \ref{Application Thm 1}). This theorem shows us that the regularity of the potentials has a direct influence on the decay rates of the problem \eqref{Eq Intro Main Prob}: the more regular the potential is, the stronger the rate of decay in \eqref{Eq Intro Main Prob} is obtained. Furthermore, the shape of $\Omega$ corresponding to these large pseudoeigenvalues can be described generally as follows 
\[ \Omega:=\left\{\alpha+i\beta\in \C: \alpha\gtrsim 1 \text{ and } \beta\in B \right\},\]
in which the wide of the interval $B$ can be any size as long as it is contained in the interval $\displaystyle\left(\limsup_{x\to -\infty} \Im V(x), \liminf_{x \to +\infty} \Im V(x)\right)$.
Furthermore, the method can also be applied for the decaying but not integrable potential 
\[ V(x)=i \frac{\textup{sgn}(x)}{\vert x \vert^{\gamma}}, \qquad \vert x \vert \gtrsim 1,\, \gamma \in (0,1),\]
in which the Assumption \eqref{Assump G Im V 1} is broken (see Example \ref{Example Decaying} and Subsection \ref{Section Decaying potentials}).
\subsubsection{Large imaginary pseudoeigenvalues}
Concerning the pseudoeigenvalues whose imaginary part $\beta$ play the main role in making the right hand side of \eqref{Eq Intro Main Prob} decaying, the pseudomodes will be constructed such that their supports live completely in $\R_{0}^{+}$. Therefore, it will be more convenient to consider the operators on $L^{2}(\R_{0}^{+})$ instead of $L^2(\R)$. Then the application for the class of operators on $L^2(\R)$ is easily obtained by the trivial extension of pseudomodes from $\R_{0}^{+}$ to $\R$. We assume that $\Im V$ is strictly increasing for sufficiently large $x$ and unbounded at $+\infty$ such that we can determine a unique turning point $x_{\beta}>0$ of the equation
\begin{equation*}
\Im V(x_{\beta})= \beta.
\end{equation*}
The WKB analysis will be performed around $x_{\beta}$ and the support of pseudomode will be inside some appropriate neighborhood of this point. Here are our assumptions for this construction:
\begin{assumpA}\label{Assumption 2}
Let $N\in \N_{0}$, and let $V\in W^{N+3,2}_{\mathrm{loc}}(\R_{0}^{+})$ satisfy all conditions of Assumption~\ref{Assumption 1} and further the followings:
\begin{enumerate}[label=\textbf{\arabic*)}]
\item $\Im V$ goes to $+\infty$ as $x\to+\infty$
\begin{equation}
\lim_{x\to +\infty} \Im V(x) = +\infty;
\end{equation}
\item There exists $\left(t_{1},t_{2}\right)\in \left[0,\frac{\varepsilon_{1}}{20(3+\varepsilon_{1})}\right]^2$ satisfying
\begin{equation}\label{xi120}
t_{1}-\frac{1}{4+\varepsilon_{1}}t_{2}<\frac{\varepsilon_{1}}{20(4+\varepsilon_{1})}
\end{equation}
such that, for all $x \gtrsim 1$,
\begin{align}
\Im V(x)^{1-t_{1}}\tau(x)^{1+t_{2}} &\lesssim \Im V^{(1)}(x) ,\label{V'}\\
\left\vert \Im V^{(2)}(x)\right\vert &\lesssim \Im V^{(1)}(x)\tau(x),\label{V2}
\end{align}
where $\tau\coloneqq \tau_{+}$.
\end{enumerate}
\end{assumpA}
Although there are more conditions for the imaginary part $\Im V$ in Assumption \ref{Assumption 2}, the class of admissible potentials is still very large. In \cite[Cond. 5.2]{KS19} for the Schr\"{o}dinger operator, the authors set up the condition
\begin{equation}\label{Cond 5.2}
 \Im V(x)  x^{\nu}\lesssim \Im V^{(1)}(x), \qquad x \gtrsim 1,
\end{equation}
and this is a particular case of \eqref{V'} with $\tau(x)=x^{\nu}$ and $t_{1}=t_{2}=0$. However, \eqref{Cond 5.2} can not treat the potential $V(x)=i \ln(x)$ (here $\nu =-1$), it is because
\[ \Im V(x)\tau(x) = \ln(x) x^{-1}\qquad \text{ and }\qquad \Im V^{(1)}(x)=x^{-1}.\]
By allowing $t_{1}$ and $t_{2}$ to be flexible satisfying \eqref{xi120}, this potential can be covered in our assumption (see Example \ref{Example Log}).  
  Let us define a neighborhood of $x_{\beta}$ in which the pseudomode lives, that is
\[ J_{\beta}= \left(x_{\beta}-2\Delta_{\beta},x_{\beta}+2\Delta_{\beta}\right).\]
Here, $\Delta_{\beta}$ is defined as follows: From the assumption \eqref{Tau}, there exists a constant $\eta$ such that, for sufficiently large~$x>0$,
\begin{equation*}
\frac{\eta}{\tau(x)}\leq \frac{x^{-\nu}}{4},
\end{equation*} 
where $\nu:=\nu_{+}$ is the number defined in \eqref{Tau} and then we define
\begin{equation}\label{Delta beta}
\Delta_{\beta}\coloneqq \frac{\eta}{\tau(x_{\beta})}.
\end{equation}
Since $\nu\geq -1$, then $x_{\beta}-2\Delta_{\beta}\geq x_{\beta}-\frac{x_{\beta}^{\nu}}{2}\geq \frac{x_{\beta}}{2}$, we deduce that $J_{\beta}\subset \R^{+}$ for $x_{\beta}>0$. We see that if $\tau(x)=x^{-1}$ (see Examples \ref{Example Log} and \ref{Example Pol}), the support of the pseudomode is able to be extended on $\R^{+}$, \emph{i.e.} $\vert J_{\beta} \vert=4\Delta_{\beta}=4\eta x_{
\beta} \to +\infty$ as $\beta\to +\infty$. This is one of the strengths of this direct construction which makes it going beyond the semi-classical construction (see \cite[Remark~2.3]{KN20}). Now we can state our second theorem:
\begin{theorem}\label{Theorem 2}
Let Assumption~\ref{Assumption 2} hold for some $N\in \N_{0}$. Assume that there exists a ($\beta$-dependent) $\alpha$ such that the following holds as $\beta\to +\infty$, for all $x\in J_{\beta}$,
\begin{align}
&\alpha-\Re V(x) \approx \vert \alpha \vert,\label{alpha 1}\\
&\left[\beta\tau(x_{\beta})\right]^{\frac{4}{5}}\lesssim \vert \alpha \vert \lesssim \left[ \beta\tau(x_{\beta})^{-1}\right]^{\frac{4+\varepsilon_{1}}{3+\varepsilon_{1}}}.\label{alpha 20}
\end{align}
Then, there exist $c>0$, $\beta_{0}>0$ and a family $\left(\Psi_{\lambda,N}\right)\subset \textup{Dom}(H_{V})$ such that for all $\beta\geq \beta_{0}$, we have
\begin{equation*}
\frac{\Vert (H_{V}- \lambda)\Psi_{\lambda,N} \Vert_{L^2(\R^{+})}}{\Vert \Psi_{\lambda,N} \Vert_{L^2(\R^{+})}} \lesssim \kappa(\beta)+\sigma^{(N)}(\beta),
\end{equation*}
in which
\begin{itemize}
\item $\displaystyle \kappa(\beta)\coloneqq \exp\left(-c \frac{\Im V^{(1)}(x_{\beta})\tau(x_{\beta})^{-2}}{\vert\alpha \vert^{\frac{3}{4}}+\beta^{\frac{3}{4}}}\right)$,
\item $\displaystyle\sigma^{(0)}(\beta):=  \sum_{k=0}^{2}\sum_{j=1}^{k+1}\frac{\tau(x_{\beta})^{k+1}\left(\left\vert \Re V(x_{\beta})\right\vert^{j}+\beta^{j}\right)}{\left\vert \alpha\right\vert^{\frac{k-3}{4}+j}}$,
\item $\displaystyle\sigma^{(N)}(\beta):=  \sum_{k=0}^{3N-1}\sum_{j=1}^{k+N+1}\frac{\tau(x_{\beta})^{k+N+1}\left(\left\vert \Re V(x_{\beta})\right\vert^{j}+\beta^{j}\right)}{\left\vert \alpha\right\vert^{\frac{k+N-3}{4}+j}}$, \qquad  $N\geq 1$.
\end{itemize}
\end{theorem}
The same result for Schr\"{o}dinger operator given in \cite[Cond. (5.5)]{KS19}  is known to be optimal when considering the case $\Re V=0$, we believe that in this case our bound curves for $\alpha$ in  \eqref{alpha 20} is also optimal. The study of optimality of our estimates on these bounds constitutes an interesting open problem.
\subsubsection{Semi-classical setting}
Let us consider the semi-classical biharmonic operator on $\R$
\[ H_{h}= h^4\frac{\dd^4}{\dd x^4}+W(x),\]
where $h$ is the positive semi-classical parameter. By using the same construction as for Theorem \ref{Theorem 2}, we can establish a pseudomode for this operator as $h\to 0$:
\begin{theorem}\label{Theorem 3}
Let $N\in \N_{0}$ and let $W\in W^{N+3,\infty}_{\textup{loc}}(\R)$, $\mu>0$, $x_{0}\in \R$. Assume that there exists a neighborhood $I$ of $x_{0}$ such that the function~$\Im W(x)- \Im W(x_{0})$ changes its sign at the point $x_{0}$ on $I$. By fixing $z=\mu+W(x_{0})$, then there exist $h_{0}>0$ and a family $\Psi_{h,N}\in \textup{Dom}(H_{h})$ such that for all $h\in (0,h_{0})$,
\[ \frac{\left\Vert \left(H_{h}-z\right)\Psi_{h,N}\right\Vert}{\Vert \Psi_{h,N} \Vert}\lesssim h^{N+1}.\]
\end{theorem}
\subsection{Applications}\label{Section Application}
Our goal in this section is to give some examples which are direct or indirect (Example \ref{Example Decaying}) application of Theorem \ref{Theorem 1}, Theorem \ref{Theorem 2} and Theorem \ref{Theorem 3}.
\subsubsection{Application of Theorem \ref{Theorem 1}}\label{Application Thm 1}
\begin{example}
Let us list some smooth potentials $V$ defined on $\R$ such that the Assumption~\ref{Assumption 1} holds.
\begin{enumerate}[label=\textbf{\arabic*)}]
\item $V$ is bounded at both $-\infty$ and $+\infty$: Consider two smooth bounded potentials on $\R$
\begin{equation}\label{Examp arctan}
V_{1}(x)=i \arctan(x) \text{ and } V_{2}(x)=i \frac{x}{\sqrt{x^2+1}}.
\end{equation}
They satisfy Assumption~\ref{Assumption 1} with $\tau_{\pm}(x)=\left(x^2+1\right)^{-\frac{1}{2}}$ and 
\[\lim_{x\to-\infty} \Im V_{j}(x)=-1, \qquad\lim_{x\to +\infty} \Im V_{j}(x)=1,\qquad \text{ for } j=1,2.\]
Since both potentials are smooth, we can achieve the arbitrary fast decay in \eqref{sigma N} by taking any large $N$. More precisely, Theorem~\ref{Theorem 1} states that: For any $N\in \N_{0}$ and for any $\beta_{\pm}\in [0,1)$, there exists a family $(\Psi_{\lambda,N})$ such that
\[ \frac{\Vert (H_{V_{j}}- \lambda)\Psi_{\lambda,N} \Vert}{\Vert \Psi_{\lambda,N} \Vert} \lesssim \alpha^{-\frac{N+1}{4}}, \qquad \text{ for }j=1,2,\]
for all $\lambda$ belonging to $\Omega:=\left\{\alpha+i\beta\in \C: \alpha\gtrsim 1 \text{ and } \beta\in [-\beta_{-},\beta_{+}]\subset (-1,1)\right\}$ whose picture is given in Figure \ref{Fig arctan}.
\begin{figure}[h]
 \centering
\includegraphics[width=0.7\textwidth]{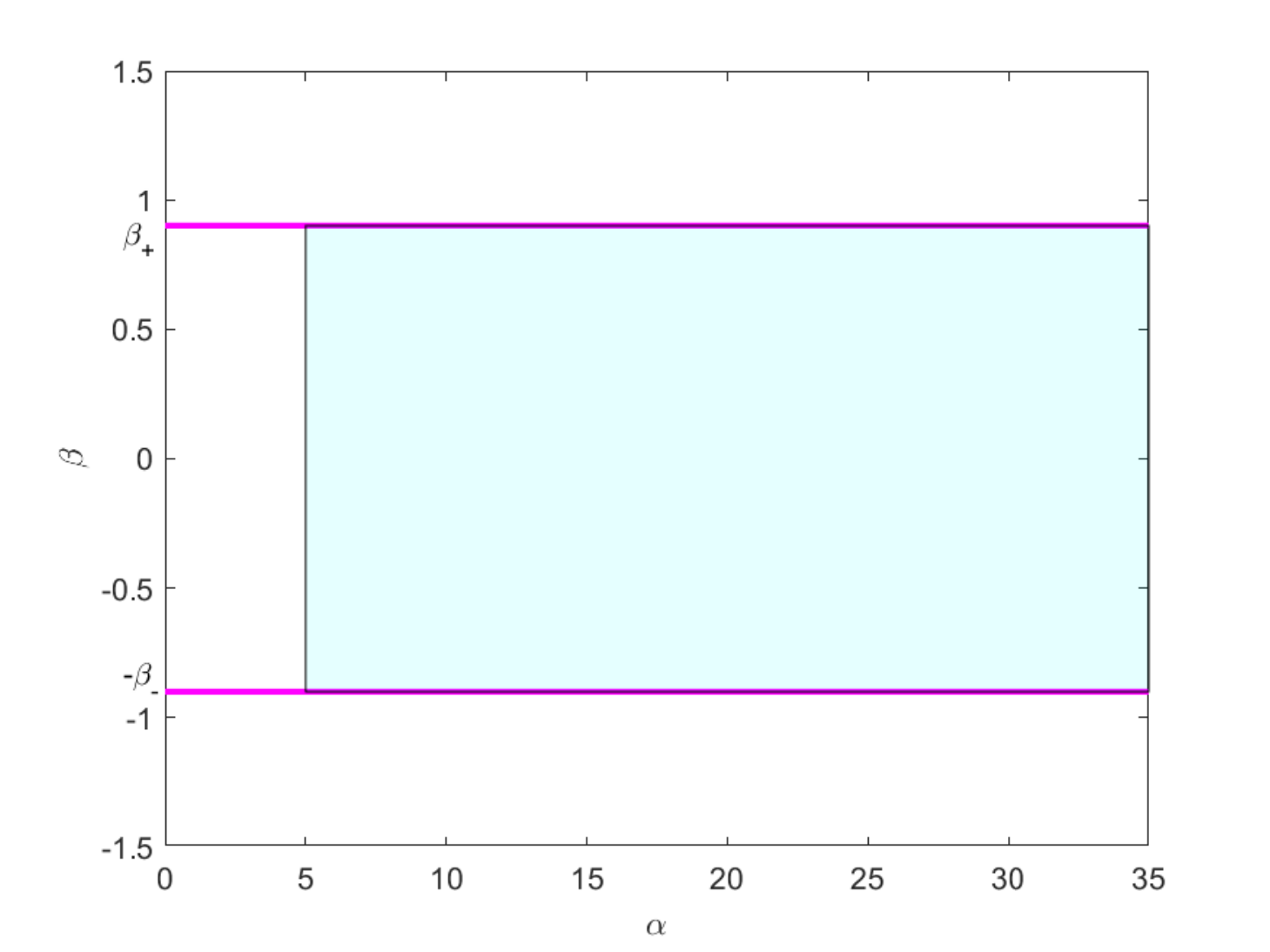}
\caption{Illustration of the shape of $\Omega$ (in cyan color) associated with the potentials $V_{1}(x)=i \arctan(x)$ and $V_{2}(x)=i \frac{x}{\sqrt{x^2+1}}$ given in \eqref{Examp arctan}.}\label{Fig arctan}
\end{figure}
\item $V$ is bounded at $-\infty$ and unbounded at $+\infty$: A simple choice is
\[ V(x)=i \left( e^{x}-1 \right).\]
It meets the condition \eqref{Assump G Im V 1} since 
\[ \lim_{x\to-\infty} \Im V(x) =-1,\qquad \lim_{x\to+\infty} \Im V(x) =+\infty,\]
and it satisfies the other conditions with $\tau_{\pm}=1$ and~$\nu_{+}=0$. Depending on the behaviour of $V$ at $\pm \infty$, $\sigma_{-}^{(N)}(\alpha)$ and $\sigma_{+}^{(N)}(\alpha)$ have different decaying:
\[ \sigma_{-}^{(N)}(\alpha)\lesssim \alpha^{-\frac{N+1}{4}},\qquad \sigma_{+}^{(N)}(\alpha)=\alpha^{\frac{2-N}{4}+\varepsilon}.\]
Here, to estimate $\sigma_{+}^{(N)}$, we just notice that $\delta_{\alpha}^{+}$ is the solution of the equation 
\[ \Im V(x)=\alpha^{\frac{3+\varepsilon_{1}}{4+\varepsilon_{1}}},\]
therefore, by writing $\frac{3+\varepsilon_{1}}{4+\varepsilon_{1}}\eqqcolon\frac{3}{4}+\varepsilon$, we get the above estimate, since
\[ \sigma_{+}^{(N)}(\alpha)=\alpha^{-\frac{N+1}{4}}\displaystyle \sup_{x\in [0,\delta_{\alpha}^{+}]}\vert V(x)\vert=\alpha^{-\frac{N+1}{4}} \Im V\left(\delta_{\alpha}^{+}\right)=\alpha^{\frac{3+\varepsilon_{1}}{4+\varepsilon_{1}}-\frac{N+1}{4}}.\]
Then, for any $N\in \N_{3}$ and for any $\beta_{-}\in [0,1)$ and $\beta_{+}\in \R_{0}^{+}$, there exists a family $(\Psi_{\lambda,N})$ such that
\[ \frac{\Vert (H_{V}- \lambda)\Psi_{\lambda,N} \Vert}{\Vert \Psi_{\lambda,N} \Vert} \lesssim \alpha^{\frac{2-N}{4}+\varepsilon},\]
for all $\displaystyle\lambda\in \Omega:=\left\{\alpha+i\beta\in \C: \alpha\gtrsim 1 \text{ and } \beta\in [-\beta_{-},\beta_{+}]\subset (-1,+\infty)\right\}$.
\end{enumerate}

\end{example}
\begin{example}[\emph{Potentials with logarithmic imaginary}]
Consider $V\in W^{N+3,\infty}_{\textup{loc}}(\R)$, with~$N\geq~0$, satisfying Assumption \ref{Assumption 1} with $\tau_{\pm}(x)=\left(x^2+1\right)^{-\frac{1}{2}}$, which satisfy \eqref{Tau} with $\nu_{\pm}=-1$, that has the form
\[ V(x)= \Re V(x) + i\ln\left(x+\sqrt{x^2+1}\right),\]
where 
$$\left\vert \Re V(x)\right\vert \lesssim \vert x \vert^{\rho}\ln\left(\vert x \vert\right)^4, \qquad \vert x \vert\gtrsim 1, \text{ with some } \rho<\frac{4}{3}.$$
The condition $\rho<\frac{4}{3}$ is sufficient to guarantee \eqref{Assump Bound Re Im}. For instance, $\Re V(x)$ is a polynomial of degree $\rho$. The range such that the constants $\beta_{\pm}$ in Theorem \ref{Theorem 1} can be taken is $\R_{0}^{+}$ since
\[\lim_{x\to-\infty} \Im V(x)=-\infty, \qquad\lim_{x\to +\infty} \Im V(x)=+\infty.\]
Concerning $\sigma_{\pm}^{(N)}(\alpha)$, since the functions $\tau_{\pm}(x)^{N+1}\vert V(x) \vert$ are bounded on $\R$ for any $N\geq 1$ \textup{(}or for any $N\geq 0$ if $\rho<1$\textup{)}, we do not need to compute $\delta_{\alpha}^{\pm}$ in this situation. Accordingly, for any $N\in \N_{1}$ and for any $\beta_{\pm}\in \R_{0}^{+}$, there exists a family $(\Psi_{\lambda,N})$ such that
\[ \frac{\Vert (H_{V}- \lambda)\Psi_{\lambda,N} \Vert}{\Vert \Psi_{\lambda,N} \Vert} \lesssim \alpha^{-\frac{N+1}{4}}, \]
for all $\displaystyle\lambda\in \Omega:=\left\{\alpha+i\beta\in \C: \alpha\gtrsim 1 \text{ and } \beta\in [-\beta_{-},\beta_{+}]\subset \R\right\}$.
\end{example}
\begin{example}[\emph{Polynomial-like potentials}]\label{Example Pol 1}
Let us take a look at the potential $V\in W^{N+3,\infty}_{\textup{loc}}(\R)$, with $N\geq 0$, satisfying Assumption \ref{Assumption 1} with $\tau_{\pm}(x)= (x^2+1)^{-\frac{1}{2}}$ and having the form
\[ \vert \mathrm{Re}\, V(x) \vert\lesssim \vert x \vert^{\rho},\qquad  \vert \mathrm{Im}\, V(x)\vert  \approx  \vert x \vert^{\gamma} ,\qquad \vert x \vert \gtrsim 1,\]
with $\rho\in \R$ and $\gamma\geq 0$. For examples, $\Re V$ and $\Im V$ are, respectively, the polynomials of degree $\rho$ and $\gamma$. It is necessary to assume that $\gamma\geq 0$ in order to meet the condition \eqref{Assump G Im V 1} and we assume further that $\gamma>\frac{3\rho-4}{4}$ such that \eqref{Assump Bound Re Im} is satisfied. Accordingly, the fast growth of $\vert\Re V(x)\vert$ require the fast growth of $\Im V(x)$. In particular if $\rho<\frac{4}{3}$ \textup{(}\emph{i.e.} $\vert \Re V(x) \vert$ grows slower that $\vert x \vert^{\frac{4}{3}}$\textup{)} even a bounded $\Im V$ fits. In order to apply Theorem \ref{Theorem 1}, let us denote the quantity
\[ \omega:=\max\{ \rho, \gamma\}.\]
Clearly, $\omega\geq 0$ and  $V$ is bounded if and only if $\omega=0$. 
When $\vert V\vert$ is unbounded, $\delta_{\alpha}^{\pm}$ is the solution of the equation 
$$\vert x \vert^{\gamma}\left(x^2+1\right)^{\frac{1}{2}}=\alpha^{\frac{3+\varepsilon_{1}}{4+\varepsilon_{1}}}.$$
When $x$ is large enough, since $\vert x \vert^{\gamma}\left(x^2+1\right)^{\frac{1}{2}}\approx \vert x \vert^{\gamma+1}$, we can approximate the solution of the above equation with the notation $\approx$ introduced in Subsection \ref{Handy notation} as follows
\[ \delta^{-}_{\alpha}=\delta^{+}_{\alpha}=\delta\approx \alpha^{\frac{3}{4(\gamma+1)}+\varepsilon}.\]
Here $\varepsilon>0$ can be made arbitrary small by an appropriate choice of small $\varepsilon_{1}>0$. Hence Theorem \ref{Theorem 1} results that
\begin{equation}\label{Polynomial}
\frac{\Vert (H_{V}- \lambda)\Psi_{\lambda,N} \Vert}{\Vert \Psi_{\lambda,N} \Vert} = \left\{
\begin{aligned}
&\mathcal{O}\left(\alpha^{-\frac{N+1}{4}}\right), \qquad & \omega \leq N+1,\\
&\mathcal{O}\left( \alpha^{-\frac{N+1}{4}+\frac{3}{4}\frac{\omega-N-1}{\gamma+1}+\varepsilon(\omega-N-1)}\right),\qquad &\omega>N+1,
\end{aligned}
\right.
\end{equation}
as $\alpha\to+\infty$ and $\beta\in [-\beta_{-},\beta_{+}]$ which is mentioned in \eqref{Box B}. When $V$ is bounded, the decay is also included in the case $\omega\leq N+1$. To improve the decay rate in the second case, we let $\varepsilon$ be small enough and notice that from the condition $\gamma>\frac{3\rho-4}{4}$, we can control the other term by
\begin{equation*}
\frac{3}{4}\frac{\omega-N-1}{\gamma+1}<\left\{ 
\begin{aligned}
&\frac{3}{4}\qquad &&\text{if } \gamma\geq \rho,\\
&1 \qquad &&\text{if } \gamma< \rho.
\end{aligned}
\right.
\end{equation*}
By considering $\varepsilon$ very small in \eqref{Polynomial}, we see that the pseudomode with $N=3$ \textup{(}\emph{i.e.} we require at least $V\in W^{6,\infty}_{\textup{loc}}(\R)$\textup{)} is sufficient to treat all polynomial-like potentials. Comparing this with the same results for Schr\"{o}dinger operators in~\cite[Ex. 3.8]{KS19} and Dirac operators in~\cite[Ex. 2]{KN20}, more terms in the pseudomode expansion are needed in the higher order differential operators.
\end{example}
\begin{example}[\emph{Super-exponential potential}]\label{Example Superexp}
We devote the other application of Theorem \ref{Theorem 1} for the potential that is smooth and grow very fast at infinity, that is
\[ V(x)= \cosh(\sinh(x))+i \sinh(\sinh(x)).\]
The Assumption \ref{Assumption 1} is satisfied with $\tau_{\pm}(x)=\cosh(x)$ and $\nu_{\pm}=0$. We emphasize here that~\cite[Asm. I]{KS19} can not cover this potential, more precisely \cite[Cond. 3.2]{KS19} can not be satisfied. The solution $\delta$ of the equation $\frac{\vert\sinh(\sinh(\pm \delta))\vert}{\cosh(\delta)}=\alpha^{\frac{3+\varepsilon_{1}}{4+\varepsilon_{1}}}$ can be estimated as follows
\[\delta_{\alpha}^{-}=\delta_{\alpha}^{+}=\delta\approx \ln\left(\ln\left(\alpha^{\frac{3}{4}+\varepsilon}\right)\right),\]
where $\varepsilon$ can be made arbitrary small by an appropriate choice of small $\varepsilon_{1}$.
Then, for arbitrary $\beta_{\pm} \in \R_{0}^{+}$ and for all $\beta \in [-\beta_{-}, \beta_{+}]\subset \R$, the pseudomodes with $N\geq 3$ leads to a decay
\[\frac{\Vert (H_{V}- \lambda)\Psi_{\lambda,N} \Vert}{\Vert \Psi_{\lambda,N} \Vert} = \mathcal{O}\left(\alpha^{\frac{2-N}{4}+2\varepsilon}\right), \qquad \text{ as } \alpha \to +\infty. \]
\end{example}
\begin{example}[Decaying potentials]\label{Example Decaying}
Consider a class of potentials with the asymptotic behaviour 
\begin{equation}\label{Examp Decay}
V(x) = i \frac{\textup{sgn}(x)}{\vert x \vert^{\gamma}}, \qquad \vert x \vert\gtrsim 1,\ 0<\gamma<1,
\end{equation}
which spoils the assumption~\eqref{Assump G Im V 1}. However, the analysis in Subsection \ref{Section Decaying potentials} shows us that we can apply the same construction as for Theorem \ref{Theorem 1} to set up pseudomodes for large pseudoeigenvalues $\lambda=\alpha+i\beta$ satisfying
\begin{equation}\label{psi-1 decaying 2}
\vert\beta \vert \alpha^{\frac{3}{4}\frac{\gamma}{1-\gamma}} =o(1),\qquad \text{ as } \alpha\to +\infty,\, \vert \beta \vert \to 0.
\end{equation}
Furthermore, if we make the restriction \eqref{psi-1 decaying 2} stronger by considering
\[ \vert \beta \vert \lesssim \alpha^{-\frac{3}{4}\frac{\gamma}{1-\gamma}-\varepsilon},\qquad \text{ as } \alpha \to +\infty,\]
we will have a decay
\[ \frac{\Vert (H_{V}- \lambda)\Psi_{\lambda,N} \Vert}{\Vert \Psi_{\lambda,N} \Vert} = \mathcal{O}\left(\alpha^{-\frac{N+1}{4}}\right), \qquad\text{ as } \alpha \to +\infty.\]
In other words, the $\Omega$ in the main problem \eqref{Eq Intro Main Prob} can be described by \textup{(see Figure \ref{Fig Decay})}
\[ \Omega=\left\{ \alpha+i\beta\in \C: \alpha\gtrsim 1,\ \vert \beta \vert\lesssim \alpha^{-\frac{3}{4}\frac{\gamma}{1-\gamma}-\varepsilon} \right\}. \]
In \cite[Eqn. 3.24]{KS19}, Krej\v{c}i\v{r}\'{i}k and Siegl used this type of decaying potential to give a natural Laptev-Safronov eigenvalues bounds for the Schr\"{o}dinger operator with $L^{p}$-potentials \textup{($p>1$)} which appears in \cite[Theorem 5]{LS09}. Here, in the same manner, by observing that $V\in L^{p}$ if $\gamma p>1$, it yields that
\[ \frac{3}{4} \frac{\gamma}{1-\gamma}(p-1)=\frac{3}{4}\frac{\gamma p-\gamma}{1-\gamma}>\frac{3}{4}.\]
From this, we obtain a bound for $\Omega$, which is also a bound for the distribution of the eigenvalues of the biharmonic operator
\[ \Omega \subset \left\{\alpha+i\beta \in \C:  \vert \beta \vert^{p-1}=o \left(\alpha^{-\frac{3}{4}} \right) , \alpha \to +\infty\right\}.\]
If the power of $\alpha$ is $-\frac{1}{2}$ in the Schr\"{o}dinger case, this power is replaced by $-\frac{3}{4}$ for the biharmonic one.
\begin{figure}[h]
 \centering
\includegraphics[width=0.7\textwidth]{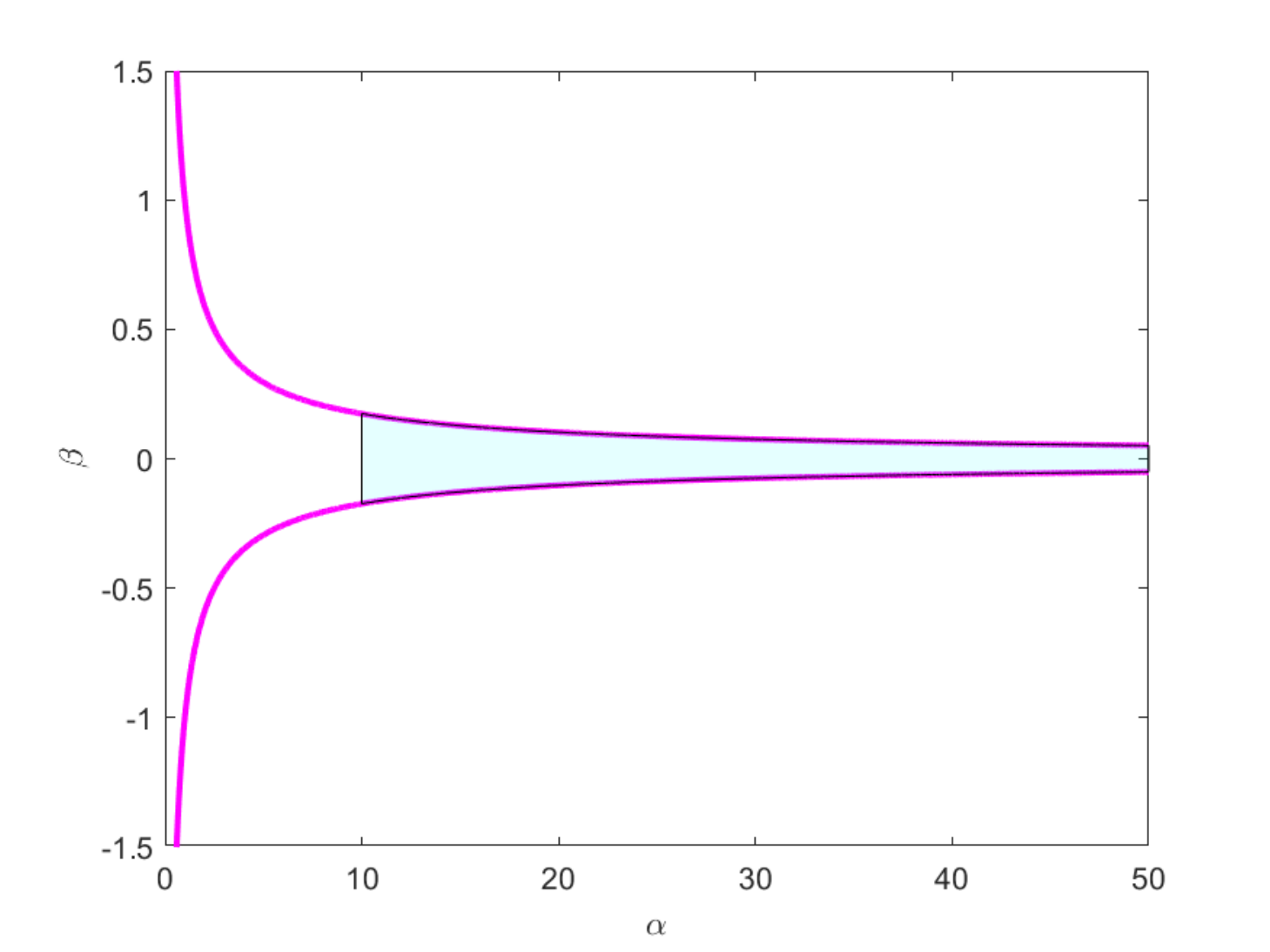}
\caption{Illustrations of the shapes of $\Omega$ (in cyan color) associated with the potential $V(x)= i \frac{\textup{sgn}(x)}{\vert x \vert^{\gamma}}$ given in \eqref{Examp Decay} with $\gamma=\frac{1}{2}$. The curves are the graphs of $\beta=\pm \alpha^{-\frac{3}{4}-\frac{1}{100}}$.}\label{Fig Decay}
\end{figure}
\end{example}
\subsubsection{Application of Theorem \ref{Theorem 2}}\label{Subsection App Theorem 2}
Now, we would like to apply Theorem \ref{Theorem 2} to study the  elementary potentials considered in Subsection \ref{Application Thm 1}. It is worthwhile to mention here that the below Examples \ref{Example Log} and \ref{Example SuperExponential} can not be covered by \cite[Cond. 5.2]{KS19} or \cite[Cond. 4.3]{KN20}. 
\begin{example}\label{Example Log}
Let us consider again the logarithmic potential $V\in W_{\mathrm{loc}}^{N,2}(\R_{0}^{+})$ that has the following behaviour
\begin{equation}\label{Eq Logarithm C}
 V(x) := i \ln(x), \qquad  x\gtrsim 1.
\end{equation}
All conditions of Assumption~\ref{Assumption 2} are satisfied with $\tau(x)=\left(x^2+1\right)^{-\frac{1}{2}}\approx x^{-1}$ \textup{(}$\nu=-1$ in~\eqref{Tau}\textup{)} as $x\gtrsim 1$, and for instance $t_{1}=0$ and some $t_{2} >0$. Given $\beta>0$, then $x_{\beta}>0$ is determined by the relation
$x_{\beta}=e^{\beta}$. Since $\Re V(x)=0$, the conditions \eqref{alpha 1} and \eqref{alpha 20} are assured iff
\begin{equation}\label{Exam Log alpha}
\beta^{\frac{4}{5}} \exp\left(-\frac{4}{5} \beta \right) \lesssim \alpha \lesssim \beta^{\frac{4}{3}-\varepsilon}\exp\left( \left(\frac{4}{3}-\varepsilon\right)\beta \right),
\end{equation}
where $\varepsilon>0$ can be chosen arbitrarily small by an appropriate choice of small $\varepsilon_{1}$. Then, thanks to the second inequality in \eqref{Exam Log alpha}, we have
\begin{align*}
\kappa(\beta)\lesssim \left\{
\begin{aligned}
&\exp \left(- c\beta^{\frac{3\varepsilon}{4}-1}\exp\left( \beta^{\frac{3\varepsilon}{4}}\right)\right) \qquad &&\text{if } \alpha>\beta,\\
&\exp \left(- c\beta^{-\frac{3}{4}}\exp\left( \beta\right)\right) \qquad &&\text{if } \alpha\leq \beta.
\end{aligned}
\right.
\end{align*}
Let us present here the detail of estimating $\sigma^{(N)}$ for $N\geq 1$ (for $N=0$, it is analogous). If $\alpha>\beta$, we estimate straightforwardly as follows
\begin{align*}
\sigma^{(N)}(\beta)&\lesssim \sum_{k=0}^{3N-1}\sum_{j=1}^{k+N+1} \frac{\exp\left(-(k+N+1)\beta\right)\beta^{j} }{\alpha^{\frac{k+N-3}{4}+j}}\lesssim  \sum_{k=0}^{3N-1}\frac{\beta\exp\left(-(k+N+1)\beta\right) }{\beta^{\frac{k+N+1}{4}}}\\
&\lesssim \beta^{\frac{3-N}{4}}\exp\left(-(N+1)\beta\right).
\end{align*}
When $\alpha \leq \beta$, we employ the first inequality in \eqref{Exam Log alpha}, we obtain
\begin{align*}
\sigma^{(N)}(\beta)&\lesssim \sum_{k=0}^{3N-1} \frac{\exp\left(-(k+N+1)\beta\right)\beta^{k+N+1} }{\alpha^{\frac{k+N-3}{4}+k+N+1}}\lesssim  \sum_{k=0}^{3N-1}\frac{\beta^{k+N+1}\exp\left(-(k+N+1)\beta\right) }{\left[ \beta^{\frac{4}{5}}\exp\left(-\frac{4}{5}\beta\right)\right]^{\frac{5}{4}k+\frac{5}{4}N+\frac{1}{4}}}\\
&\lesssim \beta^{\frac{4}{5}} \exp\left(-\frac{4}{5}\beta\right).
\end{align*}
In summary, Theorem \ref{Theorem 2} provides the pseudomodes such that
\[ \frac{\Vert (H_{V}- \lambda)\Psi_{\lambda,N} \Vert}{\Vert \Psi_{\lambda,N} \Vert} = \left\{
\begin{aligned}
&\mathcal{O}\left(\beta^{\frac{3-N}{4}}\exp(-(N+1)\beta)\right) \qquad &&\text{ if } \alpha >\beta,\\
&\mathcal{O}\left(\beta^{\frac{4}{5}} \exp\left(-\frac{4}{5}\beta\right)\right) \qquad &&\text{ if }  \alpha \leq \beta,
\end{aligned}
\right.\]
where $\lambda \to \infty$ in 
\begin{equation}\label{Eq Log Omega C 1}
\Omega:= \left\{\alpha+i\beta \in \C : \beta\gtrsim 1 \text{ and }  \beta^{\frac{4}{5}} \exp\left(-\frac{4}{5} \beta \right) \lesssim \alpha \lesssim \beta^{\frac{4}{3}-\varepsilon}\exp\left( \left(\frac{4}{3}-\varepsilon\right)\beta \right) \right\}.
\end{equation}
From the definition of $\Omega$, we see that the pseudospectral region contains even points which stay very close to the line $\alpha=0$ \textup{(see Figure~\ref{Fig: Log and Pol} \subref{Fig: Log})}.
\begin{figure}[h]
 \centering
 \begin{subfigure}[c]{0.35\textwidth}
 \centering
 \includegraphics[width=\textwidth]{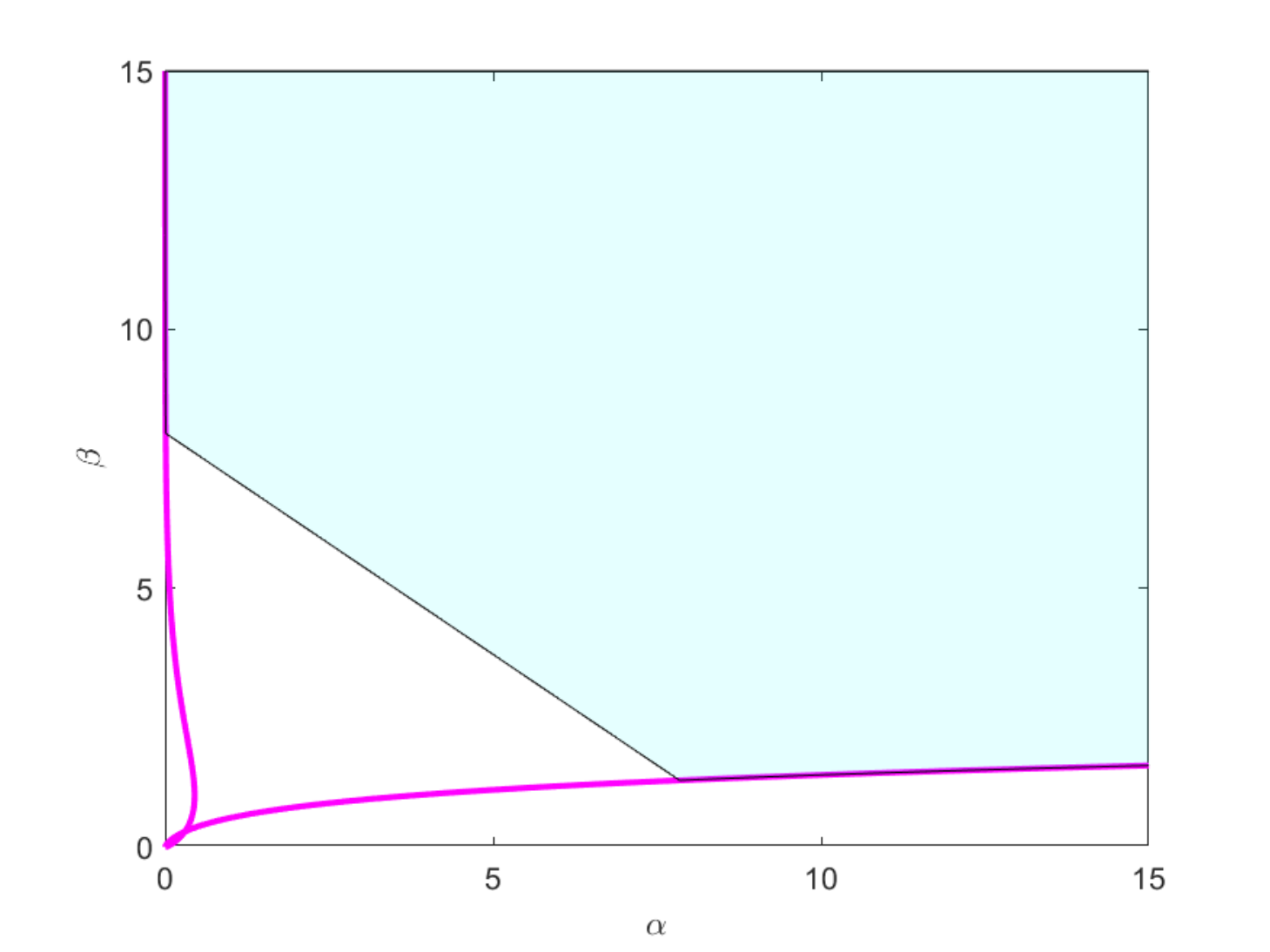}
 \caption{$V(x)=i \ln(x)$.}
 \label{Fig: Log}
\end{subfigure}
~
\begin{subfigure}[c]{0.35\textwidth}
 \centering
\includegraphics[width=1\textwidth]{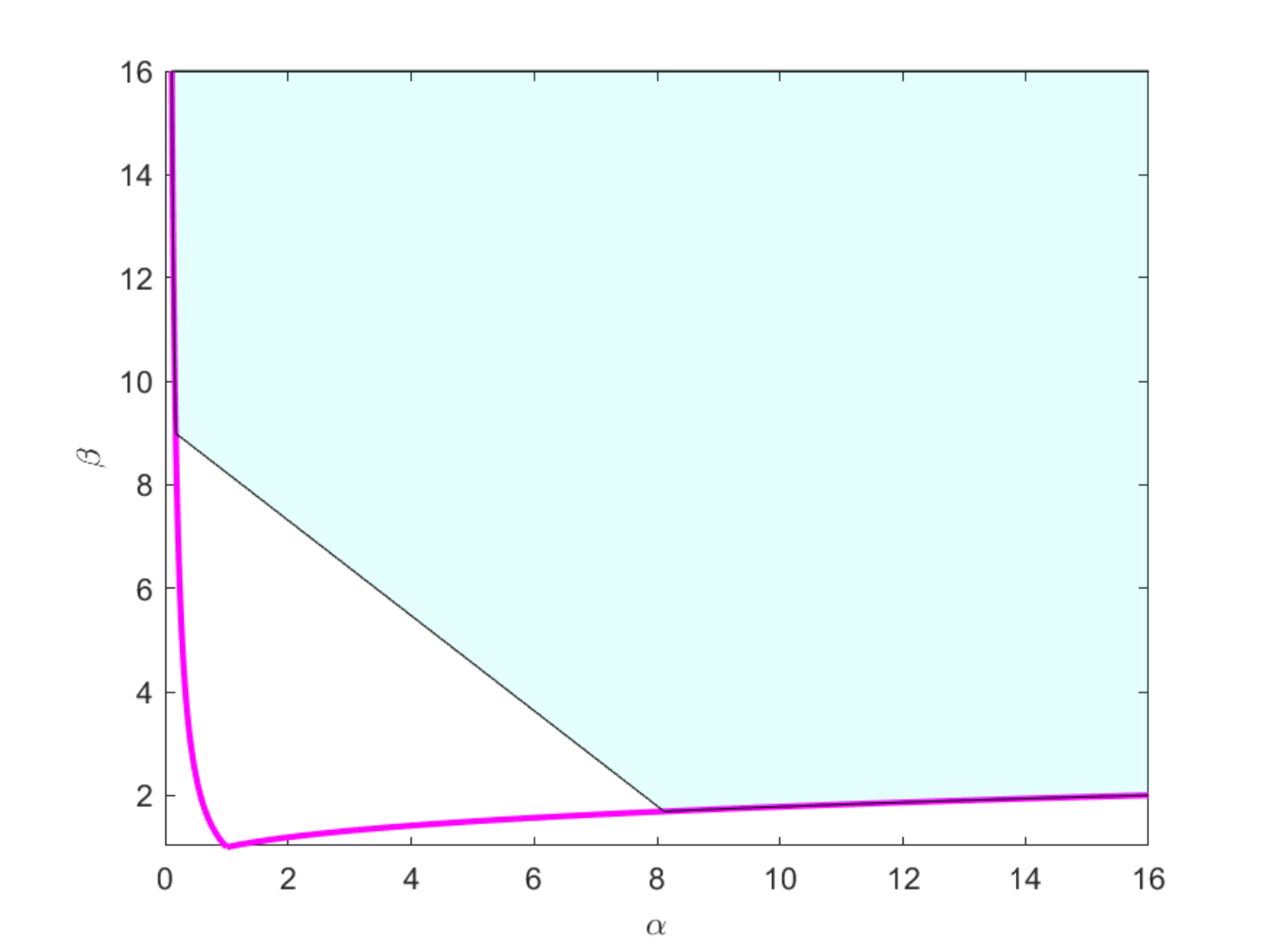}
\caption{$V(x)=ix^{\frac{1}{2}}$.}
\label{Fig: Pol1}
\end{subfigure}
~
\begin{subfigure}[c]{0.35\textwidth}
 \centering
\includegraphics[width=1\textwidth]{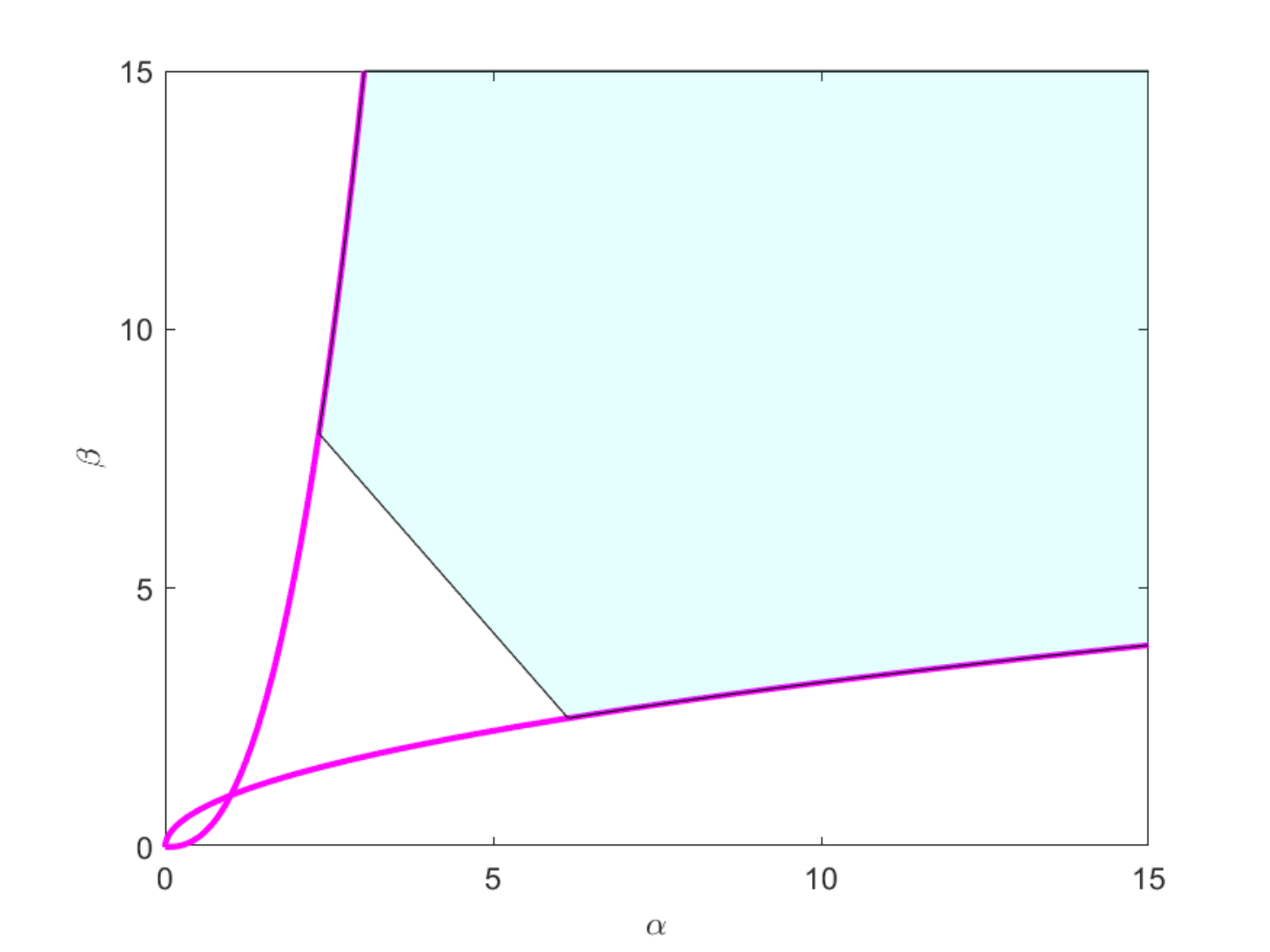}
\caption{$V(x)=i x^{2}$.}
\label{Fig: Pol2}
\end{subfigure}
\captionsetup{singlelinecheck=off}
\caption[foo bar]{Illustration of the shapes of $\Omega$ (in cyan color) with the logarithmic function $V$ given in \eqref{Eq Logarithm C} and the polynomial $V$ given in \eqref{Eq Polynomial}:
\begin{enumerate}[label=(\alph*)]
  \item $V(x)=i \ln(x)$: The curves are the graphs of $\alpha=\beta^{\frac{4}{5}}\exp\left(-\frac{4}{5}\beta\right)$ and $\alpha=\beta^{\frac{4}{3}-\frac{1}{100}}\exp\left(\left(\frac{4}{3}-\frac{1}{100}\right)\beta\right)$,
  \item $V(x)=i x^{\frac{1}{2}}$: The curves are the graphs of $\alpha=\beta^{-\frac{4}{5}}$ and $\alpha=\beta^{4-\frac{1}{100}}$,
   \item $V(x)=i x^{2}$: The curves are the graphs of $\alpha=\beta^{\frac{2}{5}+\frac{1}{100}}$ and $\alpha=\beta^{2-\frac{1}{100}}$.
\end{enumerate}
}
\label{Fig: Log and Pol}
\end{figure}
\end{example}

\begin{example}\label{Example Pol}
Next, we want to study the polynomial potential on $\R_{0}^{+}$:
\begin{equation}\label{Eq Polynomial}
V(x) =  i x^{\gamma}, \qquad x\gtrsim 1
\end{equation}
where $\gamma>0$. All the conditions of Assumption \ref{Assumption 2} are satisfied with $\tau$ chosen as in Example~\ref{Example Log} and we can choose $t_{1}=t_{2}=0$. Given $\beta>0$, then $x_{\beta}>0$ is determined by $x_{\beta}=\beta^{\frac{1}{\gamma}}$. It is straightforward to check that the conditions \eqref{alpha 1} and \eqref{alpha 20} holds if we choose $\alpha=\alpha(\beta)$ in the following way
\begin{equation}\label{alpha Poly}
\beta^{\frac{4}{5}\left(1-\frac{1}{\gamma}\right)}\lesssim \alpha \lesssim \beta^{\frac{4}{3}\left(1+\frac{1}{\gamma}\right)-\varepsilon},
\end{equation}
where $\varepsilon>0$ can be chosen arbitrarily small by an appropriate choice of small $\varepsilon_{1}$. By the second inequality in \eqref{alpha Poly}, the term $\kappa(\beta)$ can be estimated as
\begin{align*}
\kappa(\beta)\lesssim \left\{
\begin{aligned}
&\exp \left(- c\beta^{\frac{3\varepsilon}{4}}\right) \qquad &&\text{if } \alpha>\beta,\\
&\exp \left(- c\beta^{\frac{1}{\gamma}+\frac{1}{4}}\right) \qquad &&\text{if } \alpha\leq \beta.
\end{aligned}
\right.
\end{align*}
By performing the estimate analogously in Example \ref{Example Log} to control $\sigma^{(N)}(\beta)$ and consider $N$ large enough, we obtain the decay in two following cases.
\begin{enumerate}[label=\textbf{\alph*)}]
\item If $0<\gamma<1$, we get
\[ \frac{\Vert (H_{V}- \lambda)\Psi_{\lambda,N} \Vert}{\Vert \Psi_{\lambda,N} \Vert} = \left\{
\begin{aligned}
&\mathcal{O}\left(\beta^{-(N+1)\left(\frac{1}{\gamma}+\frac{1}{4}\right)+1}\right) \qquad &&\text{ if } \alpha >\beta,\\
&\mathcal{O}\left(\beta^{\frac{4}{5}\left(1-\frac{1}{\gamma}\right)}\right) \qquad &&\text{ if }  \alpha \leq \beta,
\end{aligned}
\right.\]
when $\lambda \to \infty$ in the region
\[ \Omega:= \left\{\alpha+i\beta \in \C: \beta \gtrsim 1 \text{ and } \beta^{\frac{4}{5}\left(1-\frac{1}{\gamma}\right)}\lesssim \alpha \lesssim \beta^{\frac{4}{3}\left(1+\frac{1}{\gamma}\right)-\varepsilon} \right\}.\]
The shape of this region as $\gamma=\frac{1}{2}$ is sketched in \textup{Figure~\ref{Fig: Log and Pol} \subref{Fig: Pol1}}.
\item If $\gamma>1$, we need to strengthen the first inequality in \eqref{alpha Poly} to obtain the decay of $\sigma(\beta)$ and thus
\[ \frac{\Vert (H_{V}- \lambda)\Psi_{\lambda,N} \Vert}{\Vert \Psi_{\lambda,N} \Vert} = \left\{
\begin{aligned}
&\mathcal{O}\left(\beta^{-(N+1)\left(\frac{1}{\gamma}+\frac{1}{4}\right)+1}\right) \qquad &&\text{ if } \alpha >\beta,\\
&\mathcal{O}\left(\beta^{\frac{4}{5}\left(1-\frac{1}{\gamma}\right)-\varepsilon\frac{5N+1}{4}}\right) \qquad &&\text{ if }  \alpha \leq \beta,
\end{aligned}
\right.\]
as $\lambda \to \infty$ in
\[ \Omega:= \left\{\alpha+i\beta \in \C: \beta \gtrsim 1 \text{ and } \beta^{\frac{4}{5}\left(1-\frac{1}{\gamma}\right)+\varepsilon}\lesssim \alpha \lesssim \beta^{\frac{4}{3}\left(1+\frac{1}{\gamma}\right)-\varepsilon} \right\}.\]
The shape of this region as $\gamma=2$ is given in \textup{Figure~\ref{Fig: Log and Pol} \subref{Fig: Pol2}}.
\end{enumerate}
\end{example}
\begin{example}\label{Example SuperExponential}
The final example that we want to study is the superexponential potential
\begin{equation}\label{Super-exponential}
V(x)=ie^{e^x},\qquad  x \gtrsim 1.
\end{equation}
All conditions of Assumption \ref{Assumption 2} are satisfied with $\tau(x)=e^{x}$, $\nu=0$ and $t_{1}=t_{2}=0$. Given $\beta>0$, the turning point $x_{\beta}$ is determined by the relation
\[ x_{\beta}=\ln\left( \ln(\beta)\right),\qquad  \beta \gtrsim 1.\]
Then $\tau(x_{\beta})=\ln(\beta)$ and the conditions \eqref{alpha 20} are equivalent to
\begin{equation}\label{alpha super-ex}
 \left[ \beta \ln(\beta) \right]^{\frac{4}{5}}\lesssim \alpha \lesssim \left[ \frac{\beta}{\ln(\beta)}\right]^{\frac{4}{3}-\varepsilon},
\end{equation}
where $\varepsilon>0$ can be chosen arbitrarily small by an appropriate choice of small $\varepsilon_{1}$ as above examples. By the second inequality in \eqref{alpha super-ex}, we obtain the estimate for $\kappa(\beta)$:
\begin{align*}
\kappa(\beta)\lesssim \left\{
\begin{aligned}
&\exp \left(- c\left(\frac{\beta}{\ln(\beta}\right)^{\frac{3\varepsilon}{4}}\right) &&\text{if } \alpha>\beta,\\
&\exp \left(- c\frac{\beta^{\frac{1}{4}}}{\ln (\beta)}\right) &&\text{if } \alpha\leq \beta.
\end{aligned}
\right.
\end{align*}
In order to get the decay of $\sigma^{(N)}(\beta)$, we need to strengthen the first inequality in \eqref{alpha super-ex} as follows
\begin{equation}\label{alpha super-ex 2}
 \beta^{\frac{4}{5}+\varepsilon} \ln(\beta)^{\frac{4}{5}} \lesssim \alpha \lesssim \left[ \frac{\beta}{\ln(\beta)}\right]^{\frac{4}{3}-\varepsilon}.
\end{equation}
 then it yields that
\[ \frac{\Vert (H_{V}- \lambda)\Psi_{\lambda,N} \Vert}{\Vert \Psi_{\lambda,N} \Vert} = \left\{
\begin{aligned}
&\mathcal{O}\left(\ln(\beta)^{N+1}\beta^{\frac{3-N}{4}}\right) \qquad &&\text{ if } \alpha >\beta,\\
&\mathcal{O}\left(\ln(\beta)^{\frac{4}{5}}\beta^{\frac{4}{5}-\varepsilon\frac{5N+1}{4}}\right) \qquad &&\text{ if }  \alpha \leq \beta,
\end{aligned}
\right.\]
as $\lambda \to \infty$ in
\[ \Omega:= \left\{\alpha+i\beta \in \C: \beta \gtrsim 1 \text{ and }  \beta^{\frac{4}{5}+\varepsilon} \ln(\beta)^{\frac{4}{5}} \lesssim \alpha \lesssim \left[ \frac{\beta}{\ln(\beta)}\right]^{\frac{4}{3}-\varepsilon} \right\}.\]
\begin{figure}[h]
 \centering
 \includegraphics[width=0.7\textwidth]{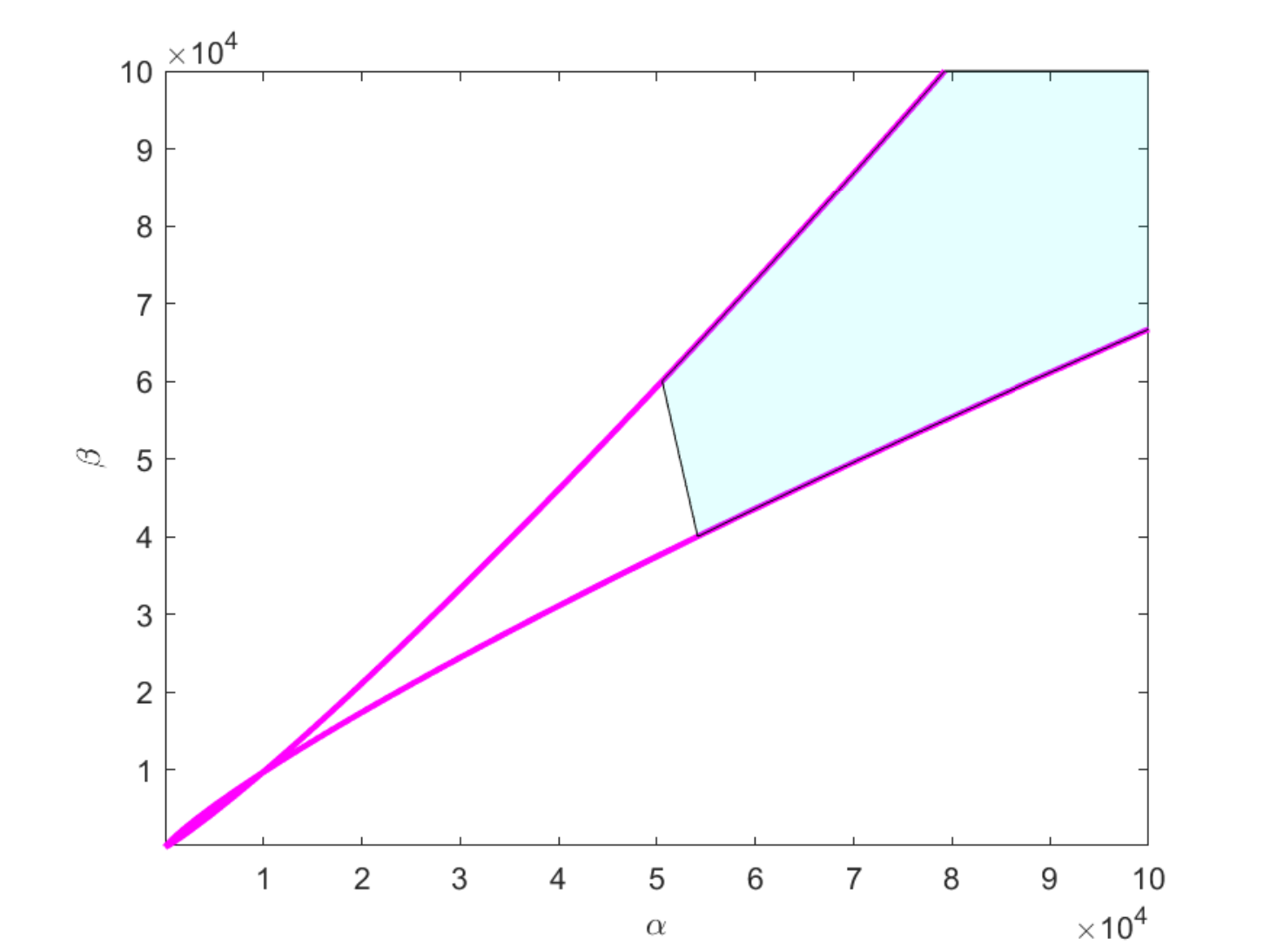}
 \caption{Illustration of the shape of $\Omega$ (in cyan color) with the superexponential $V$ given in \eqref{Super-exponential} in which the curves are the graphs of $\alpha=\beta^{\frac{4}{5}+\frac{1}{100}}\ln(\beta)^{\frac{4}{5}}$ and $\alpha=\left[\frac{\beta}{\ln(\beta} \right]^{\frac{4}{3}-\frac{1}{100}}$.}\label{Fig: Supexp}
 \end{figure}
\end{example}
From Figure~\ref{Fig: Log and Pol} and Figure~\ref{Fig: Supexp}, we see that the pseudospectrum of the operator comes close to the imaginary axis as the potential grows slowly at infinity as logarithmic functions and root-type ones $x^{\gamma}$ with $\gamma \in (0,1)$. When the potential grows faster as polynomial and exponential functions, the pseudospectrum stay away the axes accordingly.
\subsubsection{Application of Theorem \ref{Theorem 3}}
We recall some standard notions that is introduced in~\cite[Sec. IV]{KSTV15}. The symbol associated with $H_{h}$ is
\[ f(x,\xi)=\xi^{4}+W(x), \]
and its semi-classical pseudospectrum of $H_{h}$ is given by the closure of the set
\[ \Lambda=\left\{ \xi^{4}+W(x): \xi^{3} \Im W'(x)<0\right\}.\]
Since the sign of $\xi$ can be chosen freely, we can describe $\Lambda$ as
\[ \Lambda=\left\{ \xi^{4}+W(x): \xi \neq 0,\ \Im W'(x)\neq 0\right\},\]
this condition is also given in Davies' work \cite{Da99} for Schr\"{o}dinger operator.
Then $z\in \Lambda$ if and only if there exists $(\xi_{0},x_{0})\in \R^2$ such that $z=\xi_{0}^4+W(x_{0})$ with $\xi_{0}\in \R\setminus\{0\}$ and  $\Im W'(x_{0})\neq 0$. Taylor's Theorem yields
\[ \Im W(x)-\Im W(x_{0})=\Im W'(x_{0})(x-x_{0})+ \mathcal{O}\left(\vert x- x_{0}\vert^2\right).\]
Then, there exists a neighbourhood of $x_{0}$ in which the sign of the function~$\Im W(x)-~\Im W(x_{0})$ changes at the point $x_{0}$ and we can apply Theorem \ref{Theorem 3} to obtain the pseudomode for the operator $H_{h}$ corresponding to pseudoeigenvalue $z$. Furthermore, we can extend $\Lambda$ to the set
\begin{align*}
\widetilde{\Lambda}=\left\{ \xi^{4}+W(x)\left|\begin{aligned}
&\xi \neq 0, \text{ there exists } p \in \N_{0} \text{ such that }\\
 &\Im W^{(j)}(x)=0 \text{ for all } j\neq 0, j\leq 2p \text{ and } \Im W^{(2p+1)}(x)\neq 0
\end{aligned} \right.\right\}
\end{align*}
which appears in the paper \cite{Karel04} for Schr\"{o}dinger operator, and the above result remains true. In summary, we obtained the result as in the papers  \cite{Da99, Karel04} for the biharmonic operator.
\section{WKB construction}\label{Section WKB}
\subsection{The WKB expansion}
Let $P_{\lambda}: \R \to \C$ be a sufficient regular function depending on the parameter $\lambda$ which will be determined later. We consider the formal conjugated operator of $\mathscr{L}_{V}-\lambda$:
\begin{align*}
e^{P_{\lambda}}\left(\mathscr{L}_{V}-\lambda \right) e^{-P_{\lambda}}=e^{P_{\lambda}}\left(\frac{\dd^4}{\dd x^4}+V(x)-\lambda \right) e^{-P_{\lambda}}= \mathcal{D}_{\lambda}+\mathcal{R}_{\lambda},
\end{align*}
where $\mathcal{D}_{\lambda}$ has a differential expression and $\mathcal{R}_{\lambda}$ has a multiplication expression
\begin{equation}\label{Eq Reminder}
\begin{aligned}
&\mathcal{D}_{\lambda}\coloneqq\frac{\dd^4}{\dd x^4}-4P_{\lambda}^{(1)}\frac{\dd}{\dd x^3}+6\left[ \left(P_{\lambda}^{(1)}\right)^2-P_{\lambda}^{(2)}\right]\frac{\dd^2}{\dd x^2} -4\left[P_{\lambda}^{(3)}-3P_{\lambda}^{(1)}P_{\lambda}^{(2)} +\left(P_{\lambda}^{(1)}\right)^3\right]\frac{\dd}{\dd x},\\
&\mathcal{R}_{\lambda}\coloneqq-P_{\lambda}^{(4)} + 4P_{\lambda}^{(1)}P_{\lambda}^{(3)}+3\left(P_{\lambda}^{(2)}\right)^2-6\left(P_{\lambda}^{(1)}\right)^2P_{\lambda}^{(2)}+\left(P_{\lambda}^{(1)}\right)^4+V(x)-\lambda.
\end{aligned}
\end{equation}
The general WKB strategy is as follows. We look for the pseudomodes in the form
\[ \Psi_{\lambda} = \xi_{\lambda} e^{-P_{\lambda}},\]
where $\xi_{\lambda}$ is a cut-off function whose shape is determined later depending on the behaviour of $V$ at infinity ($\pm \infty$). From the triangular inequality, it yields that
\begin{equation}\label{WKB}
\frac{\left\Vert \left(\mathscr{L}_{V} - \lambda\right)\Psi_{\lambda} \right\Vert}{\Vert \Psi_{\lambda} \Vert}=\frac{\lV e^{-P_{\lambda}}\mathcal{D_{\lambda}}\xi_{\lambda}+ \xi_{\lambda}e^{-P_{\lambda}}\mathcal{R}_{\lambda}\rV}{\lV \Psi_{\lambda} \rV}\leq \frac{\lV e^{-P_{\lambda}}\mathcal{D_{\lambda}}\xi_{\lambda}\rV}{\lV \Psi_{\lambda} \rV}+ \lV \mathcal{R}_{\lambda}\rV_{L^{\infty}(J_{\lambda})}.
\end{equation}
Here $J_{\lambda}$ is the support of the cut-off $\xi_{\lambda}$. The WKB idea is to look for the phase $P_{\lambda}$ in the following form,
\begin{equation}\label{Eq P}
P_{\lambda,n} (x)= \sum_{k=-1}^{n-1} \lambda^{-k} \psi_{k}(x),\qquad n\in \N_{0},
\end{equation}
where functions $\left(\psi_{k}\right)_{k\in[[-1,n-1]]}$ are to be determined by solving some ordinary differential equations (ODEs); the number $n$ is chosen later depending on the maximal possible order derivative of $V$. In principle, by letting $\lambda\to \infty$, the first term in the right hand side of \eqref{WKB} can be shown exponentially decay thanks to consideration on the support of $\xi_{\lambda}$ and the second term often decreases with the rate power of $\lambda^{-1}$. The more regular the potential is, the stronger the rate of decay in \eqref{WKB} is obtained.

 Let us start with $n=0$ and put $P_{\lambda,0}$ into the formula of $\mathcal{R}_{\lambda}$ in \eqref{Eq Reminder}, we obtain
\[ \mathcal{R}_{\lambda,0}:=-\lambda \psi_{-1}^{(4)} + 4\lambda^2 \psi_{-1}^{(1)}\psi_{-1}^{(3)}+3\lambda^2 \left(\psi_{-1}^{(2)}\right)^2-6\lambda^3\left(\psi_{-1}^{(1)}\right)^2\psi_{-1}^{(2)}+\lambda^4\left(\psi_{-1}^{(1)}\right)^4+V(x)-\lambda.\]
By solving the equation $ \lambda^4(\psi_{-1}^{(1)})^4+V(x)-\lambda=0$, we often call it \enquote{eikonal} equation, the forth order of $\lambda$ in $R_{\lambda,0}$ is removed:
\begin{equation}\label{Remainder 0}
\mathcal{R}_{\lambda,0}:=-\lambda \psi_{-1}^{(4)} + 4\lambda^2 \psi_{-1}^{(1)}\psi_{-1}^{(3)}+3\lambda^2 (\psi_{-1}^{(2)})^2-6\lambda^3(\psi_{-1}^{(1)})^2\psi_{-1}^{(2)}.
\end{equation}
From this simple observation, we wish that when we increase $n$ in $P_{\lambda,n}$, the order of $\lambda$ appearing in the remainder $\mathcal{R}_{\lambda,n}$ has been reduced accordingly. For $n\in \N_1$, we replace $P_{\lambda,n}$ into $\mathcal{R}_{\lambda}$ in \eqref{Eq Reminder}, we get
\begin{align*}
\mathcal{R}_{\lambda,n} &\coloneqq - \sum_{k=-1}^{n-1} \lambda^{-k} \psi_{k}^{(4)}+4\sum_{k=-2}^{2n-2} \lambda^{-k} \sum_{\alpha_{1}+\alpha_{2}=k} \psi_{\alpha_{1}}^{(1)}\psi_{\alpha_{2}}^{(3)}+3\sum_{k=-2}^{2n-2} \lambda^{-k} \sum_{\alpha_{1}+\alpha_{2}=k} \psi_{\alpha_{1}}^{(2)}\psi_{\alpha_{2}}^{(2)}\\
&\qquad-6 \sum_{k=-3}^{3n-3} \lambda^{-k} \sum_{\alpha_{1}+\alpha_{2}+\alpha_{3}=k} \psi_{\alpha_{1}}^{(1)}\psi_{\alpha_{2}}^{(1)}\psi_{\alpha_{3}}^{(2)}+\sum_{k=-4}^{4n-4} \lambda^{-k} \sum_{\alpha_{1}+\alpha_{2}+\alpha_{3}+\alpha_{4}=k} \psi_{\alpha_{1}}^{(1)}\psi_{\alpha_{2}}^{(1)}\psi_{\alpha_{3}}^{(1)}\psi_{\alpha_{4}}^{(1)}\\
&\qquad+V(x)-\lambda\\
&\eqqcolon \sum_{k=-4}^{4n-4} \lambda^{-k} \phi_{k+3}.
\end{align*}
Here the function $\phi_{k}$ with $k \in [[-4,4n-4]]$ are naturally defined by grouping together the terms attached with the same order of $\lambda$, with the exception of $V(x)-\lambda$ which we include in the leading order term:
\begin{align*}
&\lambda^{4}: && \left(\psi_{-1}^{(1)}\right)^{4} + \frac{V(x)-\lambda}{\lambda^4} \eqqcolon\phi_{-1},\\
&\lambda^{3}: &&4 \left(\psi_{-1}^{(1)}\right)^{3}\psi_{0}^{(1)}-6 \left(\psi_{-1}^{(1)}\right)^2 \psi_{-1}^{(2)}\eqqcolon\phi_{0},\\
&\lambda^{2}: &&4 \left(\psi_{-1}^{(1)}\right)^{3}\psi_{1}^{(1)}+6 \left(\psi_{-1}^{(1)}\right)^{2}\left(\psi_{0}^{(1)}\right)^2-6 \left(\psi_{-1}^{(1)}\right)^2 \psi_{0}^{(2)}\\
& &&\hspace{2.5 cm}-12 \psi_{-1}^{(1)}\psi_{0}^{(1)} \psi_{-1}^{(2)}+3\left(\psi_{-1}^{(2)}\right)^2+4\psi_{-1}^{(1)}\psi_{-1}^{(3)}\eqqcolon\phi_{1},\\
&\ldots
\end{align*}
For $k \in [[-3,4n-4]]$, the formulae can be written as
\begin{equation}\label{Eq Phi k+3}
\begin{aligned}
&- \psi_{k}^{(4)}+4 \sum_{\alpha_{1}+\alpha_{2}=k} \psi_{\alpha_{1}}^{(1)}\psi_{\alpha_{2}}^{(3)}+3\sum_{\alpha_{1}+\alpha_{2}=k} \psi_{\alpha_{1}}^{(2)}\psi_{\alpha_{2}}^{(2)}\\
&-6  \sum_{\alpha_{1}+\alpha_{2}+\alpha_{3}=k} \psi_{\alpha_{1}}^{(1)}\psi_{\alpha_{2}}^{(1)}\psi_{\alpha_{3}}^{(2)}+ \sum_{\alpha_{1}+\alpha_{2}+\alpha_{3}+\alpha_{4}=k} \psi_{\alpha_{1}}^{(1)}\psi_{\alpha_{2}}^{(1)}\psi_{\alpha_{3}}^{(1)}\psi_{\alpha_{4}}^{(1)}\eqqcolon\phi_{k+3},
\end{aligned}
\end{equation}
with the convention that $\psi_{\alpha}=0$ if $\alpha\leq -2$ or $\alpha\geq n$.

Given $n\in \N_{1}$, requiring $\phi_{k}=0$ for all $k \in [[-1,n-1]]$, we obtain $n+1$ ODEs which can be solved explicitly to find all $\left(\psi_{k}^{(1)}\right)_{k\in[[-1,n-1]]}$ by a recursion formula
\begin{align}
\psi_{-1}^{(1)} &= \pm i \lambda^{-1} \left(\lambda-V\right)^{1/4},\label{Eq Eikonal Solution}\\
\psi_{k+1}^{(1)} &= \frac{1}{4(\psi_{-1}^{(1)})^3}\left(\psi_{k-2}^{(4)}-4 \sum_{\alpha_{1}+\alpha_{2}=k-2} \psi_{\alpha_{1}}^{(1)}\psi_{\alpha_{2}}^{(3)}-3\sum_{\alpha_{1}+\alpha_{2}=k-2} \psi_{\alpha_{1}}^{(2)}\psi_{\alpha_{2}}^{(2)}\right.\nonumber \\
&\left. \hspace{2 cm} +6\sum_{\alpha_{1}+\alpha_{2}+\alpha_{3}=k-2} \psi_{\alpha_{1}}^{(1)}\psi_{\alpha_{2}}^{(1)}\psi_{\alpha_{3}}^{(2)}- \sum_{\substack{\alpha_{1}+\alpha_{2}+\alpha_{3}+\alpha_{4}=k-2\\ \alpha_{1},\alpha_{2},\alpha_{3},\alpha_{4} \neq k+1}} \psi_{\alpha_{1}}^{(1)}\psi_{\alpha_{2}}^{(1)}\psi_{\alpha_{3}}^{(1)}\psi_{\alpha_{4}}^{(1)}\right),\label{Eq Transport Solution}
\end{align}
for $k\in [[-1,n-2]]$, with the convention that $\psi_{\alpha}=0$ if $\alpha\leq -2$ or $\alpha\geq n$. After solving these ODEs, the WKB remainder is
\begin{equation}\label{Eq Reminder R}
\mathcal{R}_{\lambda,n} = \sum_{k=n-3}^{4n-4}\lambda^{-k}\phi_{k+3}, \qquad n\in \N_{1}.
\end{equation}
Since $\lambda-V(x)$ is a complex-valued function, the forth root appearing in \eqref{Eq Eikonal Solution} is considered as the principal branch of the forth root which is defined as
\begin{equation}\label{Eq Forth Root}
z^{\frac{1}{4}}  = \frac{1}{\sqrt{2}}\left( \vert z \vert^{1/2} + \frac{1}{\sqrt{2}}\left(\vert z \vert+ \mathrm{Re}\, z \right)^{1/2}\right)^{1/2}+ \frac{i}{2}\frac{\mathrm{Im}\, z}{\left( \vert z \vert+ \mathrm{Re}\, z \right)^{1/2}\left(\vert z \vert^{1/2}+\frac{1}{\sqrt{2}}\left(\vert z \vert+\mathrm{Re}\, z \right)^{1/2}\right)^{1/2}}.
\end{equation}
Although there are four solutions for the eikonal equation, that is $\pm \lambda^{-1} \left(\lambda-V\right)^{1/4}$ and $\pm i \lambda^{-1} \left(\lambda-V\right)^{1/4}$, but only the latter are suitable for our pseudomodes. The choice of the sign in the definition of $\psi_{-1}^{(1)}$ will be determined by the sign of $\Im V$ at infinity (see Remark \ref{Remark sign}).
\subsection{Structure of solutions of the transport equations and the WKB remainder}
From now on, we assume that we are dealing with the plus sign in the formula of $\psi_{-1}$ in~\eqref{Eq Eikonal Solution}, unless otherwise stated. Let us list some first solutions of the first transport equations, $\psi_{k}^{(1)}$ for $k\geq 0$, to see which structure they are equipped with:
\begin{align*}
\psi_{0}^{(1)}  = - \frac{3}{8} \frac{V^{(1)}}{V_{\lambda}},\qquad\psi_{1}^{(1)}  = \frac{i\lambda}{V_{\lambda}^{(1/4)}}\left(\frac{5}{16} \frac{V^{(2)}}{V_{\lambda}} +\frac{45}{128} \frac{(V^{(1)})^2}{V_{\lambda}^2}\right).
\end{align*}
If we continue, we will see that $\psi_{2}^{(1)}$ and $\psi_{3}^{(1)}$ has the form
\begin{align*}
&\psi_{2}^{(1)}=\frac{\lambda^2}{V_{\lambda}^{\frac{2}{4}}}\left\{\frac{V^{(3)}}{V_{\lambda}},\frac{V^{(1)}V^{(2)}}{V_{\lambda}^2},\frac{\left(V^{(1)}\right)^3}{V_{\lambda}^3} \right\},\\
&\psi_{3}^{(1)}=\frac{\lambda^3}{V_{\lambda}^{\frac{3}{4}}}\left\{\frac{V^{(4)}}{V_{\lambda}},\frac{V^{(1)}V^{(3)}}{V_{\lambda}^2}, \frac{\left(V^{(2)}\right)^2}{V_{\lambda}^2},\frac{\left(V^{(1)}\right)^2V^{(2)}}{V_{\lambda}^3}, \frac{\left(V^{(1)}\right)^4}{V_{\lambda}^4} \right\},
\end{align*}
where the bracket denotes a linear combination of all elements in the bracket with complex coefficients.  To estimate the transport solutions later, the coefficients attached with  these elements are not important, instead the structure they share together is essential. That is: for each $k\in\N_{1}$, each element in the bracket of $\psi_{k}^{(1)}$ has the form
\[ \frac{(V^{(1)})^{\alpha_1} (V^{(2)})^{\alpha_2}...(V^{(s)})^{\alpha_{s}}}{V_{\lambda}^{j}},\]
in which $s=k+1-j$ and all $\left(\alpha_{i}\right)_{i\in [[1,s]]} \in \N_{0}^{s}$ satisfy
\[ \alpha_{1}+\alpha_{2}+\ldots+\alpha_{s}=j,\qquad 1\alpha_{1}+2\alpha_{2}+\ldots+s\alpha_{s}=k+1. \]
This is the content of the following lemma, but first, some notations should be introduced.
\begin{notation}\label{Nota Derivative}
For $j,r\in \N_{0}$ such that $j\leq r$, we employ the following notations
\begin{equation*}
D_{r,j}:=\left\{\displaystyle \sum_{\alpha \in \mathcal{I}_{r,j}} c_{\alpha} (V^{(1)})^{\alpha_1} (V^{(2)})^{\alpha_2}...(V^{(r-j+1)})^{\alpha_{r-j+1}}  : c_{\alpha}\in \C\right\},
\end{equation*}
where
\begin{equation}\label{Notation Irj}
\mathcal{I}_{r,j}:= \left\{ \alpha \in \N_{0}^{r-j+1}:\sum_{p=1}^{r-j+1} p \alpha_{p} = r \text{ and }  \sum_{p=1}^{r-j+1} \alpha_{p} =j \right\}\,.
\end{equation}
When $\mathcal{I}_{r,j} =\emptyset$, we make a convention that $D_{r,j}=\{0\}$. Thus, $D_{r,0}=\{0\}$ if $r\geq 1$.
\end{notation}
\begin{lemma}\label{Lem Transport Solution}
Let $n\in \N_{0}$, $V\in W_{\mathrm{loc}}^{n+3,2}(\R)$ and functions $\left(\psi_{k}^{(1)}\right)_{k\in [[-1,n-1]]}$ be determined by~\eqref{Eq Eikonal Solution} and \eqref{Eq Transport Solution}. When $x\in \R$ such that $V_{\lambda}(x) \in \C \setminus \R_{0}^{-}$, then we have
\begin{equation}
\psi_{k}^{(m)}(x) = \frac{\lambda^{k}}{V_{\lambda}(x)^{\frac{k}{4}}} \sum_{j=0}^{k+m} \frac{d_{k+m,j}(x)}{V_{\lambda}(x)^{j}},\qquad m\in [[1,n+3-k]], \text{ with } d_{r,j} \in D_{r,j}.
\end{equation}
Moreover, $d_{r,0}=0$ if $r\geq 1$.
\end{lemma}
The condition that the range of $V_{\lambda}$ needs to stay away from $(-\infty,0]$ is added to ensure that $V_{\lambda}^{\frac{1}{4}}$ is well-defined (\emph{i.e.} non-multi-valued) and differentiable inherited from the differentiability of $V$ (since the principal branch of forth root is holomorphic on $\C\setminus (-\infty,0]$).
From \eqref{Remainder 0} and  \eqref{Eq Eikonal Solution}, the remainder when we solve up to $\psi_{-1}$ can be calculated explicitly:
\begin{equation}\label{R0}
\begin{aligned}
& \mathcal{R}_{\lambda,0}(x)= i V_{\lambda}(x)^{1/4} \left( \frac{1}{4}\frac{V^{(3)}(x)}{V_{\lambda}(x)}+\frac{9}{16}\frac{V^{(1)}(x)V^{(2)}(x)}{V_{\lambda}^2(x)}+\frac{21}{64}\frac{(V^{(1)}(x))^3}{V_{\lambda}^3(x)}\right)\\
&\hspace{1.6 cm}+V_{\lambda}(x)^{2/4}\left(\frac{V^{(2)}(x)}{V_{\lambda}(x)}+\frac{9}{16}\frac{(V^{(1)}(x))^2}{V_{\lambda}(x)^2}\right)+i V_{\lambda}(x)^{3/4} \left(-\frac{3}{2}\frac{V^{(1)}(x)}{V_{\lambda}(x)}\right).
\end{aligned}
\end{equation}
We see that the remainder contains the elements which still share the structure mentioned above. Thanks to the recursion formula \eqref{Eq Transport Solution}, we can show by induction that the shape of the remainder $\mathcal{R}_{\lambda,n}$  can be written by means of Notation \ref{Nota Derivative} as follows:
\begin{lemma}\label{Lem Reminder Estimate}
Let $n\in\N_{0}$, $V\in W_{\mathrm{loc}}^{n+3,2}(\R)$ and functions $\left(\psi_{k}^{(1)}\right)_{k\in [[-1,n-1]]}$ be determined by~\eqref{Eq Eikonal Solution} and \eqref{Eq Transport Solution}, $\left(\phi_{k}\right)_{k\in [[-1,4n-1]]}$ be as in \eqref{Eq Phi k+3} and $R_{\lambda,n}$ as in \eqref{Eq Reminder R}.  When $x\in \R$ such that $V_{\lambda}(x) \in \C \setminus \R_{0}^{-}$, then the maximal order derivative of $V$ in $\mathcal{R}_{\lambda,n}$ is $n+3$ and 
\begin{equation*}
\begin{aligned}
&\mathcal{R}_{\lambda,0}(x)=\sum_{k=0}^{2}\frac{1}{V_{\lambda}(x)^{\frac{k-3}{4}}} \sum_{j=1}^{k+1}\frac{d_{k+1,j}(x)}{V_{\lambda}^{j}(x)},\\
&\mathcal{R}_{\lambda, n}(x)=\sum_{k=0}^{3n-1}\frac{1}{V_{\lambda}(x)^{\frac{k+n-3}{4}}} \sum_{j=1}^{k+n+1}\frac{d_{k+n+1,j}(x)}{V_{\lambda}^{j}(x)},\qquad n\geq 1, \text{ with } d_{r,j} \in D_{r,j}.
\end{aligned}
\end{equation*}
\end{lemma}
The proofs of these lemmata is postponed to Appendix \ref{Appendix}.
\section{Pseudomodes for large real pseudoeigenvalues}\label{Section Proof of theorem 1}
We reserve this section for proving Theorem \ref{Theorem 1}. In other words, we are going to construct the pseudomode for the perturbed biharmonic operator when the real part of the spectral parameter $\lambda$ is considered largely.

From the assumptions \eqref{Assump G Der} and \eqref{Box B}, there exist constants $a_{\pm}>0$ such that, for all $\beta\in B=[-\beta_{-},\beta_{+}]$,
\begin{equation}\label{Assump G Im V 2}
\begin{aligned}
&\Im V(x)-\beta\lesssim  \Im V(x)\lesssim -1,\qquad && x\in I^{-}\coloneqq \left\{ x\in \R_{0}^{-}: x\leq -a_{-}\right\},\\
&\Im V(x)-\beta \gtrsim \Im V(x)\gtrsim 1,\qquad && x\in I^{+}\coloneqq \left\{ x\in \R_{0}^{+}: x\geq a_{+}\right\}.
\end{aligned}
\end{equation} 
Notice that the constants in the above notations $\lesssim$ and $\gtrsim$ are uniformly on $\beta$. Possibly considering larger $a_{\pm}$, we assume that all the remain assumptions of Assumption~\ref{Assumption 1} also happen on $I^{\pm}$. Furthermore, two constants $a_{\pm}$ will be used as universal constants in this section, i.e. the size of $a_{\pm}$ can be changed a finite number of times but we still keep denoting them as $a_{\pm}$.
\subsection{Shapes of the cut-off functions}
The cut-off function is employed to complete two tasks in our construction of pseudomodes: first, by attaching with the functions created by WKB method, the cut-off make them belong to the domain of the operator; second, the cut-off makes $(\lambda-V)^{\frac{1}{4}}$ well-defined (\emph{i.e.} non-multi-valued) and differentiable in its support. In the latter, the differentiability of $(\lambda-V)^{\frac{1}{4}}$ comes from the analyticity of the forth root on $\C\setminus (-\infty,0]$ and the regularity of $V$ with the requirement that $\lambda-V(x)\in \C\setminus (-\infty,0]$ on the support of the cut-off. Let us denote by $\xi: \R \to [0,1]$ the cut-off function satisfying the following properties
\begin{equation}\label{Cutoff 1}
\left\{
\begin{aligned}
&\xi \in \mathcal{C}_{0}^{\infty}(\R),\\
&\xi(x)=1 \qquad \text{on } \left(-\delta_{\alpha}^{-}+\Delta_{\alpha}^{-}, \delta_{\alpha}^{+}-\Delta_{\alpha}^{+}\right),\\
&\xi(x)=0 \qquad \text{on } \R\setminus \left(-\delta_{\alpha}^{-},\delta_{\alpha}^{+}\right),
\end{aligned}\right.
\end{equation}
where $\delta_{\alpha}^{\pm}$ and $\Delta_{\alpha}^{\pm} <\delta_{\alpha}^{\pm}$ are $\alpha$-dependent positive numbers which will be determined later. Notice that the cut-off $\xi$ can be selected in such a way that
\begin{equation}\label{Eq cutoff Delta}
\Vert \xi^{(j)} \Vert_{L^{\infty}(\R_{\pm})} \lesssim \left(\Delta_{\alpha}^{\pm}\right)^{-j},\qquad j\in [[1,4]].
\end{equation}
To simplify the notation, we introduce the following sets
$$
\begin{aligned}
&J_{\alpha}^{-}:=\left[-\delta_{\alpha}^{-},0\right],\qquad &&J_{\alpha}^{+}:=\left[0,\delta_{\alpha}^{+}\right], \qquad &&&J_{\alpha}:= J_{\alpha}^{-}\cup J_{\alpha}^{+}.
\end{aligned}
$$
We use the following lemma (see the proof in \cite[Lemma 3.1]{KN20}) to define the boundary $\delta_{\alpha}^{\pm}$ of the cut-off $\xi$.
\begin{lemma}\label{Lemma alpha}
Let $g:\R_{0}^{+} \to \R_{0}^{+} $ be a continuous function and let $\alpha$ be a positive number, we define 
\begin{equation*}
\delta(\alpha) : = \inf\left\{ x\geq 0: g(x)=\alpha\right\} .
\end{equation*}
Then $\delta(\alpha)$ can be infinite $(\inf \emptyset = +\infty)$, however, when $g$ is unbounded at $+\infty$ and for all sufficiently large $\alpha>0$, the number $\delta(\alpha)$ is finite and
\begin{equation*}
\lim_{\alpha\to +\infty} \delta(\alpha)=+\infty.
\end{equation*}
Furthermore, if $\alpha>g(0)$ then
\begin{equation*}
 g(x)\leq \alpha,
 \qquad x\in [0,\delta(\alpha)].
\end{equation*}
\end{lemma}
We consider the following cases:
\begin{enumerate}[label=\textbf{\alph*)}]
\item When $V$ is unbounded at $\pm \infty$: From the assumption \eqref{Assump Bound Re Im}, we have
\begin{equation*}
 \tau_{\pm}^{4+\varepsilon_{1}}(x)\lesssim \left\vert \Im V(x) \right\vert,\qquad \left\vert \Re V(x) \right\vert^{3+\varepsilon_{1}} \tau_{\pm}^{4}(x)\lesssim \left\vert \Im V(x) \right\vert^4, \qquad  x\in I^{\pm}.
\end{equation*}
By choosing $\eta_{1}\in (0,1)$ such that $4+\varepsilon_{1}=\frac{4}{1-\eta_{1}}$, we will obtain, for all $ x\in I^{\pm}$,
\begin{equation}\label{Re Im 1}
\left\{
\begin{aligned}
&\tau_{\pm}(x)^{\frac{4}{1-\eta_{1}}}\lesssim \left\vert\Im V(x) \right\vert, \\
&\left\vert\Im V(x) \right\vert\lesssim \left\vert \frac{\Im V(x)}{\tau_{\pm}(x)}\right\vert^{\frac{4}{3+\eta_{1}}}, \\
&\left\vert \Re V(x) \right\vert^{\frac{1}{1-\eta_{1}}}\lesssim \left\vert \frac{\Im V(x)}{\tau_{\pm}(x)}\right\vert^{\frac{4}{3+\eta_{1}}}.
\end{aligned}
\right.
\end{equation}
\item When $V$ is bounded at $\pm \infty$, from the assumption \eqref{Assump Bound Above tau}, by choosing $\eta_{2}\in (0,1)$ such that $\frac{1}{3}-\varepsilon_{2} = \frac{1-\eta_{2}}{3+\eta_{2}}$, we have
\begin{equation}\label{tau bound}
\tau_{\pm}(x)\lesssim \vert x \vert^{\frac{1-\eta_{2}}{3+\eta_{2}}},\qquad  x \in I^{\pm}.
\end{equation}
\end{enumerate}
By using Lemma \ref{Lemma alpha}, we can define the boundary of the cut-off $\xi$
\begin{equation}\label{Eq delta Def}
\delta_{\alpha}^{\pm} : = \inf\left\{ x\geq 0: g_{\pm}(x)=\alpha\right\} 
\end{equation}
through defining functions $g_{\pm}:[0,+\infty) \to [0,+\infty)$ as follows
\begin{equation}\label{Function g}
\begin{aligned}
g_{\pm}(x) := \left\{ \begin{aligned}
&\left\vert \frac{\Im V(\pm x)}{\tau_{\pm}(\pm x)}\right\vert^{\frac{4}{3+\eta_{1}}} \qquad &&\text{ if } V \text{ is unbounded at }\pm \infty, \\
&\vert x \vert^{\frac{4}{3+\eta_{2}}}\qquad &&\text{ if } V \text{ is bounded at }\pm \infty.
\end{aligned}
\right.
\end{aligned}
\end{equation}
In the latter case, since $g_{\pm}$ is strictly increasing, we can work out precisely $\delta_{\alpha}^{\pm}$ whose formula is given at Theorem~\ref{Theorem 1}. In the first case, $V$ is continuous since $V\in W^{n+3,\infty}_{\textup{loc}}(\R)$, so are the functions $g_{\pm}$. Furthermore,  thanks to $\tau_{\pm}(0)>0$, the last two conditions in \eqref{Re Im 1}, $g_{\pm}$ is unbounded at $\pm \infty$ and bounded at $0$. In practice, it is not easy to get the exact solution $\delta_{\alpha}^{\pm}$ of the equation $g_{\pm}(x)=\alpha$, instead we may approximate this solution by means of the symbol \enquote{$\approx$} introduced in the Notation \ref{Handy notation} (see Examples~\ref{Example Pol 1} and~\ref{Example Superexp}).

Next, we define the remain ingredient of the cut-off $\xi$ in \eqref{Cutoff 1}, that is $\Delta_{\alpha}^{\pm}$. When $V$ is unbounded at $\pm \infty$, with the aid of \eqref{Tau}, there exist a constant $\kappa>0$ such that, for sufficiently large $\vert x \vert$,
\begin{equation}\label{Eq c}
\frac{\eta}{\tau_{\pm}(x)}\leq \frac{\vert x \vert^{-\nu_{\pm}}}{4}.
\end{equation}
Let us define
\begin{equation}\label{Eq Delta Def}
\Delta_{\alpha}^{\pm}=\left\{
\begin{aligned}
&\frac{\eta}{\tau_{\pm}(\delta_{\alpha}^{\pm})} \qquad &&\text{if } V  \text{ is unbounded at  }\pm \infty,\\
&\frac{1}{4}\delta_{\alpha}^{\pm} \qquad  &&\text{if } V  \text{ is bounded at  }\pm \infty.
\end{aligned}
\right.
\end{equation}
Notice that since $\nu_{\pm}\geq -1$, from \eqref{Eq c}, it implies that $\Delta_{\alpha}^{\pm}\leq \delta_{\alpha}^{\pm}/4$.
\begin{proposition}
For all sufficiently large $\alpha>0$, $\delta_{\alpha}^{\pm}$ are finite and $\displaystyle\lim_{\alpha \to +\infty} \delta_{\alpha}^{\pm}=+\infty$. Furthermore,  the following hold for $\alpha \gtrsim 1$ and for all $\beta\in B$,
\begin{enumerate}[label=\textbf{\arabic*)}]
\item On $J_{\alpha}$, 
\begin{align}
&\left\vert \Re V(x) \right\vert =o(\alpha),\qquad \left\vert \Im V(x) \right\vert\lesssim \alpha, \qquad\vert \lambda-V(x) \vert \approx \alpha, \label{V lambda}\\
&\tau_{\pm}(x) \lesssim \left\{\begin{aligned}
& \alpha^{\frac{1-\eta_{1}}{4}} \qquad &&\text{if } V \text{ is unbounded at }\pm \infty,\\
& \alpha^{\frac{1-\eta_{2}}{4}} \qquad &&\text{if } V \text{ is bounded at }\pm \infty.
\end{aligned}
\right. \label{Eq tau bound above}
\end{align}

\item When $V$ is unbounded at $\pm \infty$, we have
\begin{equation}\label{Eq v assymptotic}
\tau_{\pm}(x) \approx \tau(x),\qquad \Im V(x) \approx \Im V(\pm\delta_{\alpha}^{\pm}) ,\qquad  x\in \R:  \vert x-\delta_{\alpha}^{\pm} \vert\leq  2\Delta_{\alpha}^{\pm}.
\end{equation}
\end{enumerate}
\end{proposition}

\begin{proof}
The statements related to the boundary $\delta_{\alpha}^{\pm}$ of the cut-off function $\xi$ is obtained from Lemma~\ref{Lemma alpha}. We will give a proof for $x\geq 0$, the case $x\leq 0$ is analogous.
\begin{enumerate}[label=\textbf{\arabic*)}]
\item All the estimates in \eqref{V lambda} and \eqref{Eq tau bound above} will be claimed on $I^{+}\cap J_{\alpha}^{+}$. In order to have them on $J_{\alpha}^{+}$, the continuity of $V$ and $\tau_{+}$ on $[0,a_{+}]$ is employed.
\begin{enumerate}[label=\textbf{\alph*)}]
\item When $V$ is bounded at $+\infty$, the claim on the real and imaginary part of $V$ in \eqref{V lambda} is obvious. The last one in \eqref{V lambda} is deduced from the triangle inequality
\[ \vert \alpha \vert - \vert V(x)-\beta\vert\leq \left\vert \lambda-V(x)\right\vert\leq  \vert \alpha \vert + \vert V(x)-\beta\vert, \]
and boundedness of $B$ on $\R$ when $\alpha>0$ is considered largely enough. In order to estimate $\tau_{+}(x)$, we apply Lemma \ref{Lemma alpha} with the definition of $g_{+}$ in \eqref{Function g} and the aid of \eqref{tau bound}, for all $x\in I^{+}\cap J_{\alpha}^{+}$,
\[ \tau_{+}(x)^{\frac{4}{1-\eta_{2}}}\lesssim \vert x \vert^{\frac{4}{3+\eta_{2}}}\leq \alpha.\]
\item When $V$ is unbounded at $+\infty$, we apply again Lemma \ref{Lemma alpha} with the definition of $g_{+}$ in \eqref{Function g} and the help of \eqref{Re Im 1}, we have, for all $x\in I^{+}\cap J_{\alpha}^{+}$,
\begin{align*}
\left\vert \Im V(x) \right\vert \lesssim g_{+}(x)\leq  \alpha,\quad \left\vert \Re V(x) \right\vert\lesssim g_{+}(x)^{1-\eta_{1}}\leq  \alpha^{1-\eta_{1}}.
\end{align*}
Then the estimate for $\lambda-V$ in \eqref{V lambda} is followed from
\[\left\vert \alpha-\Re V(x) \right\vert \leq\left\vert \lambda-V(x) \right\vert \leq \left\vert \alpha-\Re V(x)\right\vert +\left\vert \beta-\Im V(x)\right\vert.\] 

The above bound of $\tau_{+}$ in \eqref{Eq tau bound above} is inferred from  \eqref{Re Im 1} and Lemma \ref{Lemma alpha} that, for all $x\in I^{+}\cap J_{\alpha}^{+}$,
\[ \tau_{+}(x)^{\frac{4}{1-\eta_{1}}}\lesssim \left\vert \Im V(x) \right\vert\lesssim g_{+}(x) \leq \alpha. \]
\end{enumerate}
\item First of all, let us show that, for sufficient large $x>0$ and every $\vert h \vert\leq \frac{x^{-\nu_{+}}}{2}$, we have
\[ \tau(x+h)\approx\tau(x).\]
Indeed, from the assumption \eqref{Tau}, 
\begin{align*}
\left\vert \ln \frac{\vert \tau(x+h)\vert}{\vert \tau(x)\vert} \right\vert = \left\vert \int_{x}^{x+h} \frac{\tau^{(1)}(t)}{\tau(t)}\, \dd t\right\vert&\lesssim \left\{ 
\begin{aligned}
& \int_{x}^{x+h} \vert t \vert^{\nu} \, \dd t \qquad &&h\geq 0,\\
& \int_{x+h}^{x} \vert t \vert^{\nu} \, \dd t \qquad &&h\leq 0,
\end{aligned}
\right.\\
&\leq\left\{ 
\begin{aligned}
& h (x+h)^{\nu} \qquad &&h\geq 0,\, \nu \geq 0,\\
& h x^{\nu} \qquad &&h\geq 0,\, \nu < 0,\\
& (-h) x^{\nu} \qquad &&h\leq 0,\, \nu \geq 0,\\
& (-h) (x+h)^{\nu} \qquad &&h\leq 0, \,\nu < 0,\\
\end{aligned}
\right.\\
&\lesssim 1.
\end{align*}
In the last inequality, since $\nu_{+}\geq -1$, we used the observation that, for all $\vert h \vert \leq \frac{x^{-\nu_{+}}}{2}$ and for $x>0$,
\begin{equation*}
\frac{x}{2}\leq x+h \leq \frac{3x}{2}.
\end{equation*}
Then, the first estimate for $\tau_{+}(x)$ is deduced by replacing the above $x$ by $\delta_{\alpha}^{+}$ with the notice that $\displaystyle 2\Delta_{\alpha}^{+}\leq \frac{\left(\delta_{\alpha}^{+}\right)^{-\nu_{+}}}{2}$, thanks to \eqref{Eq c} and the definition of $\Delta_{\alpha}^{+}$. We employ this idea for the function $\Im V$ as following: for large $\delta_{\alpha}^{+}$ and for all $\vert h \vert\leq 2\Delta_{\alpha}^{+}$, we have
\begin{align*}
\left\vert \ln \frac{\vert \Im V(\delta_{\alpha}^{+}+h)\vert}{\vert \Im V(\delta_{\alpha}^{+})\vert} \right\vert = \left\vert \int_{\delta_{\alpha}^{+}}^{\delta_{\alpha}^{+}+h} \frac{\Im V^{(1)}(t)}{\Im V(t)}\, \dd t\right\vert &\leq \left\{
\begin{aligned}
 &\int_{\delta_{\alpha}^{+}}^{\delta_{\alpha}^{+}+h}\tau_{+}(t)\, \dd t &&\text{if } h\geq 0, \\
 &\int_{\delta_{\alpha}^{+}+h}^{\delta_{\alpha}^{+}}\tau_{+}(t)\, \dd t &&\text{if } h\leq 0.
\end{aligned}\right.\\
&\lesssim  \tau_{+}(\delta_{\alpha}^{+}) \Delta_{\alpha}^{+}= \eta.
\end{align*}
Here, in the second inequality, we have used the estimate for $\tau_{+}$ in \eqref{Eq v assymptotic}. Thus, the estimate for $\Im V$ in \eqref{Eq v assymptotic} is followed.
\end{enumerate}
\end{proof}
\subsection{Pseudomode estimate}\label{Section Pseudomode estimate}
Let Assumption \ref{Assumption 1} hold for some $N\in \N_{0}$. Let $\psi_{-1}^{(1)}$ be determined by \eqref{Eq Eikonal Solution} with the plus sign and $\left(\psi_{k}^{(1)}\right)_{k\in[[0,N-1]]}$ be determined by \eqref{Eq Transport Solution}, in order to obtain the primitive functions $\left(\psi_{k}\right)_{k\in [[-1,N-1]]}$, we fix the initial data for them
\[ \psi_{k}(0):=0,\qquad \forall k \in [[-1,N-1]].\]
Let us define the pseudomode for Theorem \ref{Theorem 1} as follows
\[\Psi_{\lambda,N}:=\xi_{\lambda} \exp(-P_{\lambda,N}),\]
where
\begin{itemize}
\item $\xi_{\lambda}$ is the cut-off defined in \eqref{Cutoff 1} with $\delta_{\alpha}^{\pm}$ and $\Delta_{\alpha}^{\pm}$ as in \eqref{Eq delta Def} and \eqref{Eq Delta Def},
\item $\displaystyle P_{\lambda,N}=\sum_{k=-1}^{N-1} \lambda^{-k} \psi_{k}(t)$  defined as in \eqref{Eq P}.
\end{itemize}
With the intention of estimating the pseudomode $\Psi_{\lambda,N}$ later, we firstly provide some estimates for the functions $\left(\psi_{k}^{(1)}\right)_{k\in[[-1,N-1]]}$ in the following lemma.

\begin{lemma}\label{Lemma psi k}
For $\alpha\gtrsim 1$ and for all $\beta\in B$, we have
\begin{equation}\label{Eq Estimate Re -1}
\begin{aligned}
& \mathrm{Re} \,\left(\lambda \psi_{-1}^{(1)}(x)\right)\approx  \frac{\Im V(x)-\beta}{\alpha^{\frac{3}{4}}}\qquad &&\text{on } I^{\pm}\cap J_{\alpha}^{\pm},\\
& \left\vert\mathrm{Re} \,\left(\lambda \psi_{-1}^{(1)}(x) \right)\right\vert\lesssim  \frac{1}{\alpha^{\frac{3}{4}}}\qquad &&\text{on }  [-a_{-},a_{+}],
\end{aligned}
\end{equation}
and for all $k\in [[-1,N-1]]$, for all $m\in [[1,3]]$ such that $k+m\geq 1$, we have
\begin{equation}\label{Eq Estimate Re k}
\begin{aligned}
&\left\vert \lambda^{-k} \psi_{k}^{(m)}(x) \right\vert \lesssim \frac{\vert V(x) \vert \tau_{\pm}(x)^{k+m}}{\alpha^{\frac{k}{4}+1}}\qquad &&\text{ on }  I^{\pm}\cap J_{\alpha}^{\pm},\\
&\left\vert \lambda^{-k} \psi_{k}^{(m)}(x) \right\vert \lesssim \frac{1}{\alpha^{\frac{k}{4}+1}}\qquad &&\text{ on } [-a_{-},a_{+}].
\end{aligned}
\end{equation}
\end{lemma}
\begin{proof}
From the formula of the eikonal solution \eqref{Eq Eikonal Solution} and the principal forth root given by~\eqref{Eq Forth Root}, we have
\begin{equation}\label{Re psi -1}
\mathrm{Re}\, \left(\lambda \psi_{-1}^{(1)}\right) = \frac{1}{2}\frac{\Im V-\beta}{\left( \vert V_{\lambda}\vert +\alpha-\Re V \right)^{1/2}\left(\vert V_{\lambda}\vert^{1/2}+\frac{1}{\sqrt{2}}\left(\vert V_{\lambda}\vert+\alpha-\Re V\right)^{1/2}\right)^{1/2}}.
\end{equation}
Thanks to the application of \eqref{V lambda} for the denominator of $\Re\left(\lambda \psi_{-1}^{(1)}\right)$, the first estimate in \eqref{Eq Estimate Re -1} is obtained directly by the fixed sign of $\Im V-\beta$ on each $I^{-}$ and $I^{+}$ (see \eqref{Assump G Im V 2}), while the second one in \eqref{Eq Estimate Re -1} is attained from the continuity of $\Im V$ on $[-a_{-},a_{+}]$ and the boundedness of $B$.

For each $k\in [[-1,N-1]]$ and for each $m\in [[1,3]]$ such that $k+m\geq 1$, Lemma \ref{Lem Transport Solution} and the assumption \eqref{Assump G Der} combining with the characteristic  \eqref{Notation Irj} of the set $I_{r,j}$ yield that, on $I^{+}\cap J_{\alpha}^{+}$,
\begin{align*}
\left\vert \lambda^{-k} \psi_{k}^{(m)}\right\vert &\leq  \frac{1}{\left\vert V_{\lambda}\right\vert^{\frac{k}{4}}}\sum_{j=1}^{k+m} \frac{\vert d_{k+m,j}\vert}{\vert V_{\lambda}\vert^{j}}\\
&\lesssim  \frac{1}{\vert V_{\lambda}\vert^{\frac{k}{4}}}\sum_{j=1}^{k+m} \frac{\displaystyle \sum_{\alpha\in I_{k+m,j}} \vert V^{(1)}\vert^{\alpha_{1}}\vert V^{(2)}\vert^{\alpha_{2}}\ldots \vert V^{(k+m+1-j)}\vert^{\alpha_{k+m+1-j}}}{\vert V_{\lambda}\vert^{j}}\\
&\lesssim  \frac{\tau_{+}^{k+m}}{\vert V_{\lambda}\vert^{\frac{k}{4}}}\sum_{j=1}^{k+m}  \frac{\vert V \vert^{j}}{\vert V_{\lambda}\vert^{j}}.
\end{align*}
From \eqref{V lambda}, it implies that $\vert V \vert \lesssim \alpha \lesssim \vert V_{\lambda} \vert$ on $J_{\alpha}$. Then the first estimate in \eqref{Eq Estimate Re k} for $x>0$ is obtained:
\[ \left\vert \lambda^{-k} \psi_{k}^{(m)}\right\vert \lesssim \frac{\tau_{+}^{k+m}\vert V\vert}{\vert V_{\lambda}\vert^{\frac{k}{4}+1}}\sum_{j=0}^{k+m-1}  \frac{\vert V \vert^{j}}{\vert V_{\lambda}\vert^{j}}\lesssim \frac{\tau_{+}^{k+m}\vert V\vert}{\alpha^{\frac{k}{4}+1}}. \]
On $I^{-}\cap J_{\alpha}^{-}$, all the above estimates are analogous. For $k\in [[-1,N-1]]$ and $m\in [[1,3]]$, we observe that the maximal derivatives of $V$ appearing in the expression of $\lambda^{-k}\psi_{k}^{(m)}$ is $ k+m+1-1 =k+m$ that is at most $N+2$. Since $V\in W_{\textup{loc}}^{N+3,\infty}(\R)$, all the derivatives of $V$ in $\lambda^{-k}\psi_{k}^{(m)}$ are continuous. The second one in \eqref{Eq Estimate Re k} follows immediately from the boundedness of the derivatives of $V$ on $[-a_{-},a_{+}]$.
\end{proof}
\begin{remark}[\emph{The assumption on $\Re V$}]\label{Remark Re V}
Let us explain why we set up the condition \eqref{Assump Bound Re Im} on $\Re V$. We consider the polynomial potential
\[ V(x)=-x^{\rho}+i\ \textup{sign}(x) \vert x \vert^{\gamma} \qquad \text{ with }  \rho >0,\,\gamma> 0.\]
It is obvious that the assumption \eqref{Assump G Im V 1} is fulfilled. In this case, we can choose $\tau_{\pm}(x)=(x^2+1)^{-\frac{1}{2}}$ to satisfy assumptions \eqref{Assump G Der} and \eqref{V1} and we have $\tau_{\pm}(x) \approx \vert x \vert^{-1}$ for $\vert x \vert \gtrsim 1$ then the condition \eqref{Tau} is satisfied with $\nu_{\pm}=-1$. The assumption \eqref{Assump Bound Re Im} reads: there exists $\varepsilon_{1}>0$ such that
\[ 3\rho-4 +\rho \varepsilon_{1}\leq 4\gamma. \]
We assume on the contrary that
\[ 3\rho-4>4 \gamma.\]
By change of variable $t=\alpha^{\frac{1}{\rho}}s$, it yields that, for all $x>0$,
\[\int_{0}^{x} \left\vert \Re \left( \lambda \psi_{-1}^{(1)}(t) \right)\right\vert\, \dd t \lesssim  \int_{0}^{x} \frac{\vert t^{\gamma} \vert+ \vert\beta \vert}{\left(\alpha+t^{\rho}\right)^{\frac{3}{4}}} \, \dd t\lesssim \alpha^{\frac{4\gamma-3\rho+4}{4\rho}}\int_{0}^{+\infty} \frac{s^{\gamma}+\frac{\vert \beta \vert}{\alpha^{\gamma}}}{\left(1+s^{\rho}\right)^{\frac{3}{4}}} \, \dd t=o(1), \qquad \alpha \to +\infty.\]
It means that the dominant part $\int_{0}^{x} \lambda \Re \left( \lambda \psi_{-1}^{(1)}(t)\right) \, \dd t$ in the expansion of $P_{\lambda,N}$ (as we see later in the next proposition) is bounded (uniformly in $\lambda$) on $R_{+}$ and this will spoil completely the decay of our pseudomode.
\end{remark}
\begin{proposition}\label{Prop Exponential Decay}
There exists $c>0$ such that, for $\alpha\gtrsim 1$ and for all $\beta\in B$,
\begin{equation}\label{Eq Derivative Part Estimate}
\frac{\Vert (\mathcal{D}_{\lambda,N} \xi) \exp(-P_{\lambda,N}) \Vert }{\Vert \xi \exp(-P_{\lambda,N})\Vert }\lesssim \exp\left(-c\alpha^{\frac{\eta}{4}}\right),
\end{equation}
where
\begin{itemize}
\item $\mathcal{D}_{\lambda,N}$ is the differential representation $D_{\lambda}$ after replacing $P$ as $P_{\lambda,N}$ in \eqref{Eq Reminder},
\item  $\eta:=\left\{ \begin{aligned}
& \eta_{1}\qquad &&\text{if } V \text{ is unbounded at } \pm \infty,\\
& \eta_{2}\qquad &&\text{if } V \text{ is bounded at } \pm \infty.
\end{aligned}
\right.$
\end{itemize}

\end{proposition}
\begin{proof}
The idea of the proof is as follows.
\begin{enumerate}[label=\textbf{(\arabic*)}]
\item \label{Step 1 1} We will deal with the denominator of \eqref{Eq Derivative Part Estimate} first by showing that it is bounded below by a constant which is independent of $\alpha$.
\item \label{Step 2 1} In order to handle the numerator of \eqref{Eq Derivative Part Estimate}, we will show that the eikonal term $\lambda\psi_{-1}^{(1)}$ dominates over the other terms $\lambda^{-k}\psi_{k}^{(1)}$ with $k\geq 0$ and thus there exists constants $C_{\pm}>0$ such that, for all $x\in I^{\pm}\cap J_{\alpha}^{\pm}$,
\begin{equation}\label{Exponential P 1}
\left\vert \exp(-P_{\lambda,N}(x))\right\vert\lesssim \exp\left( -\frac{C_{\pm}}{\alpha^{3/4}}\int_{\pm a_{\pm}}^{x}  \Im V(t)\, \dd t   \right).
\end{equation}
\item \label{Step 3 1} From \eqref{Exponential P 1}, we show that there exists a constant $c_{1}>0$ such that, for all $x\in \left[-\delta_{\alpha}^{-},-\delta_{\alpha}^{-}+\Delta_{\alpha}^{-}\right]$  and for all $x\in \left[\delta_{\alpha}^{+}-\Delta_{\alpha}^{+},\delta_{\alpha}^{+}\right]$,
\begin{equation}\label{Exponential P 2}
\left\vert \exp\left(-P_{\lambda,N}(x)\right)\right\vert\lesssim \exp\left( -c_{1}\alpha^{\frac{\eta}{4}}   \right).
\end{equation}
\item \label{Step 4 1} Finally, we use \eqref{Exponential P 2} to control $D_{\lambda,N}\xi$ in \eqref{Eq Derivative Part Estimate}.
\end{enumerate}
The details are as follows.
\begin{enumerate}[label=\textbf{(\arabic*)}]
\item Let us recall that
\[ P_{\lambda,N}(x) =  \sum_{k=-1}^{N-1}  \int_{0}^{x} \lambda^{-k}\psi_{k}^{(1)}(t) \, \dd t.\]
Thanks to the estimates \eqref{Eq Estimate Re -1} and \eqref{Eq Estimate Re k}, we have the bound of \eqref{Eq Derivative Part Estimate},
\begin{equation}\label{Re P bounded}
\left\vert\Re P_{\lambda,N}(x)\right\vert \lesssim \frac{1}{\alpha^{\frac{3}{4}}},\qquad  x\in [0,a_{+}].
\end{equation}
Then, there is a constant $C_{1}>0$ such that
\begin{align*}
\int_{\R} \vert \xi \exp(-P_{\lambda,n}(x)) \vert^2 \, \dd x \geq \int_{0}^{a^{+}} \exp\left(-2 \Re P_{\lambda,N}(x) \right) \, \dd x\geq a_{+} \exp\left(-\frac{C_1}{\alpha^{3/4}}  \right).
\end{align*}
Therefore, by considering large $\alpha>0$, we get
\[ \left\Vert \xi \exp\left(-P_{\lambda,N}\right) \right\Vert \gtrsim 1.\]
\item Next, we will prove \eqref{Exponential P 1}. On $I^{+}\cap J^{+}_{\alpha}$, thanks to \eqref{Eq Estimate Re k} and \eqref{Eq tau bound above}, we have 
\begin{align*}
\left\vert \sum_{k=0}^{N-1}  \mathrm{Re}\, \left(\lambda^{-k}\psi_{k}^{(1)}(t)\right) \right\vert &\lesssim \sum_{k=0}^{N-1} \frac{\vert V(t)\vert \tau_{+}(t)^{k+1}}{\alpha^{\frac{k}{4}+1}} =\frac{\vert V(t)\vert \tau_{+}(t) }{\alpha} \sum_{k=0}^{N-1} \left(\frac{\tau_{+}(t)}{\alpha^{\frac{1}{4}}}\right)^{k}\\
&\lesssim \frac{\vert V(t)\vert \tau_{+}(t) }{\alpha}.
\end{align*}
Combining this with \eqref{Eq Estimate Re -1} and \eqref{Assump G Im V 2}, it implies that
\begin{align*}
\frac{\displaystyle\left\vert \sum_{k=0}^{N-1}  \mathrm{Re}\, \left(\lambda^{-k}\psi_{k}^{(1)}(t)\right)\right\vert}{\Re \left( \lambda \psi_{-1}^{(1)}(t)\right)}\lesssim \frac{\vert V(t)\vert \tau_{+}(t) }{\Im V(t)\alpha^{\frac{1}{4}}}\lesssim \left\{
\begin{aligned}
&\alpha^{-\frac{\eta_{1}}{4}}\qquad \text{if } V \text{ is unbounded at }+\infty,\\
&\alpha^{-\frac{\eta_{2}}{4}}\qquad \text{if } V \text{ is bounded at }+\infty.
\end{aligned}
\right.
\end{align*}
Indeed, we consider that cases:
\begin{enumerate}[i)]
\item If $V$ is bounded at $+\infty$, the above estimate comes from \eqref{Assump G Im V 2} and \eqref{Eq tau bound above}.
\item If $V$ is unbounded at $+\infty$, the triangular inequality leads to
\begin{align*}
\frac{\vert V(t)\vert \tau_{+}(t) }{\Im V(t)\alpha^{\frac{1}{4}}}\leq \frac{ \vert\Re V(t)\vert \tau_{+}(t) }{\Im V(t)\alpha^{\frac{1}{4}}}+ \frac{\tau_{+}(t)}{\alpha^{\frac{1}{4}}},
\end{align*}
then \eqref{Re Im 1}, the definition of $g_{+}$ in \eqref{Function g} and Lemma \ref{Lemma alpha} allow us to control the term with $\Re V$ as follows
\begin{align*}
\frac{ \vert\Re V(t)\vert \tau_{+}(t) }{\Im V(t)\alpha^{\frac{1}{4}}}\lesssim \frac{g_{+}(t)^{1-\eta_{1}}\tau_{+}(t)}{\Im V(t) \alpha^{\frac{1}{4}}}=\frac{g_{+}(t)^{\frac{1-5\eta_{1}}{4}}}{\alpha^{\frac{1}{4}}}\leq \alpha^{-\frac{5\eta_{1}}{4}}.
\end{align*}
In the last inequality, we assumed that $\varepsilon_{1}$ small enough in Assumption \eqref{Assump Bound Re Im}. Hence using \eqref{Eq tau bound above} to control the term $\frac{\tau_{+}(t)}{\alpha^{\frac{1}{4}}}$, we have a conclusion.
\end{enumerate}
Therefore, \eqref{Exponential P 1} is obtained by employing \eqref{Re P bounded} , \eqref{Eq Estimate Re -1} and \eqref{Assump G Im V 2}. In detail, for all $x\in I^{+}\cap J_{\alpha}^{+}$, we have (with some constant $C_{+}$)
\begin{align*}
 \left\vert\exp\left(-P_{\lambda,N}(x) \right)\right\vert &=\exp\left( -\Re P_{\lambda,N}(a_{+})-\int_{a_{+}}^{x}\sum_{k=-1}^{N-1}\Re\left(\lambda^{-k}\psi_{k}^{(1)}(t)\right) \, \dd t\right)\\
 &\lesssim\exp\left( -\int_{a_{+}}^{x}(1-o(1))\Re \left(\lambda\psi_{-1}^{(1)}\right)\, \dd t\right)\\
 &\lesssim \exp\left( -\frac{C_{+}}{\alpha^{3/4}}\int_{a_{+}}^{x}  \Im V(t)\, \dd t   \right).
\end{align*}
The proof for $x\in I^{-}\cap J_{\alpha}^{-}$ is the same.
\item In order to prove \eqref{Exponential P 2}, we consider two cases:
\begin{enumerate}[i)]
\item If $V$ is unbounded at $+\infty$, from the property \eqref{Eq v assymptotic} and the definitions of $\delta_{\alpha}^{+}$ in \eqref{Eq delta Def} and $g_{+}(\delta_{\alpha}^{+})$ in \eqref{Function g}, we have that, for $x\in \left[\delta_{\alpha}^{+}-\Delta_{\alpha}^{+},\delta_{\alpha}^{+}\right]$,
\begin{align*}
\int_{a_{+}}^{x} \Im V(t)\, \dd t &\geq \int_{\delta_{\alpha}^{+}-2\Delta_{\alpha}^{+}}^{x} \Im V(t)\, \dd t\gtrsim \left(\Delta_{\alpha}^{+}\right) \Im V(\delta_{\alpha}^{+})\\
&\gtrsim \frac{\Im V(\delta_{\alpha}^{+})}{\tau(\delta_{\alpha}^{+})}=\alpha^{\frac{3+\eta_{1}}{4}}.
\end{align*}
\item If $V$ is bounded at $+\infty$, we use the assumption \eqref{Assump G Im V 2} to obtain,
\begin{align*}
\int_{a_{+}}^{x} \Im V(t)\, \dd t \gtrsim \int_{\delta_{\alpha}^{+}-2\Delta_{\alpha}^{+}}^{x}  \Im V(t)\, \dd t\gtrsim \Delta_{\alpha}^{+}=\alpha^{\frac{3+\eta_{2}}{4}}.
\end{align*}
\end{enumerate}
Then, \eqref{Exponential P 2} is followed directly from \eqref{Exponential P 1}.
\item In order to control the terms attached with $\xi^{(\ell)}$ for $\ell\in[[1,4]]$, we notice that,  for $m\in [[1,3]]$
\begin{align*}
\left\vert P_{\lambda,N}^{(m)}(x) \right\vert \leq  \sum_{k=-1}^{N-1} \left\vert  \lambda^{-k} \psi_{k}^{(m)}(x)\right\vert \lesssim  \sum_{k=-1}^{N-1} \frac{\tau_{\pm}(x)^{k+m}}{\alpha^{\frac{k}{4}}},\qquad x\in I^{\pm}\cap J_{\alpha}^{\pm}.
\end{align*}
Here we employed \eqref{Eq Estimate Re k} and \eqref{Eq Eikonal Solution}. Thanks to the upper bound of $\tau_{\pm}$ by some power of $\alpha$ in \eqref{Eq tau bound above}, we can bound $P_{\lambda,n}^{(m)}$ by a polynomial of $\alpha$ and thus they are also rapidly decaying when they are attached with $\exp(-c_{1}\alpha^{\frac{\eta}{4}})$. For example, we give a detail on how to deal with the terms attached with $\xi^{(4)}$ and $\xi^{(3)}$, the other terms are estimated similarly:
\begin{enumerate}[label=\textbf{\alph*)}]
\item The attached with $\xi^{(4)}$, by employing \eqref{Eq cutoff Delta} and \eqref{Exponential P 2}, we have
\begin{equation*}
\begin{aligned}
&\int_{-\delta_{\alpha}^{-}}^{-\delta_{\alpha}^{-}+\Delta_{\alpha}^{-}} \vert\xi^{(4)}(x)\vert^2 \left\vert \exp\left(-P_{\lambda,N}(x)\right) \right\vert^2\, \dd x+\int_{\delta_{\alpha}^{+}-\Delta_{\alpha}^{+}}^{\delta_{\alpha}^{+}}  \vert\xi^{(4)}(x)\vert^2 \left\vert \exp\left(-P_{\lambda,N}(x)\right) \right\vert^2\, \dd x\\
&\lesssim  \left(\Delta_{\alpha}^{-}\right)^{-7} \exp\left(-2c_{1} \alpha^{\frac{\eta}{4}}\right)+\left(\Delta_{\alpha}^{+}\right)^{-7} \exp\left(-2c_{1} \alpha^{\frac{\eta}{4}}\right).
\end{aligned}
\end{equation*}

By using \eqref{Eq tau bound above} for $\Delta_{\alpha}^{\pm}$ defined in \eqref{Eq Delta Def} when $V$ is unbounded at $\pm \infty$ and $\delta_{\alpha}^{\pm}=\frac{\alpha^{\frac{3+\eta_{2}}{4}}}{2}$ when $V$ is bounded at $\pm \infty$, it implies that
\begin{equation*}
\left(\Delta_{\alpha}^{\pm}\right)^{-1}\lesssim \left\{
\begin{aligned}
&\alpha^{\frac{1-\eta_{1}}{4}} &&\text{if } V \text{ is unbounded at } \pm \infty,\\
&\alpha^{-\frac{3+\eta_{2}}{4}} &&\text{if } V \text{ is bounded at } \pm \infty.
\end{aligned}
\right.
\end{equation*}
Therefore, there exists $c_{2}>0$ such that
\[\left\Vert \xi^{(4)} \exp\left(-P_{\lambda,N}(x)\right)\right\Vert\lesssim \exp\left(-c_{2} \alpha^{\frac{\eta}{4}}\right). \]
\item The attached with $\xi^{(3)}$, we have
\begin{align*}
\left\vert \xi^{(3)}(x)P_{\lambda,N}^{(1)} \exp\left(-P_{\lambda,N}(x)\right) \right\vert &\lesssim \left(\Delta_{\alpha}^{\pm}\right)^{-3}\left(\sum_{k=-1}^{N-1} \frac{\tau_{+}(x)^{k+1}}{\alpha^{\frac{k}{4}}} \right)\exp\left(-c_{1} \alpha^{\frac{\eta}{4}}\right)\\
&\lesssim \left(\Delta_{\alpha}^{\pm}\right)^{-3}\left(\sum_{k=-1}^{N-1} \alpha^{\frac{(k+1)(1-\eta)-k}{4}}\right)\exp(-c_{1} \alpha^{\frac{\eta}{4}})\\
&\lesssim \exp\left(-c_{3}\alpha^{\frac{\eta}{4}}\right).
\end{align*}
\end{enumerate}
Consequently, we put everything together, we obtain \eqref{Eq Derivative Part Estimate}.
\end{enumerate}
\end{proof}
\subsection{Remainder estimate (Proof of Theorem \ref{Theorem 1})}
Obviously, $\Psi_{\lambda,N}$ belongs to the domain of $\mathscr{L}_{V}$ because of its support. By the estimate \eqref{WKB} and Proposition \ref{Prop Exponential Decay}, we have
\begin{align*}
\frac{\Vert (\mathscr{L}_{V}- \lambda)\Psi_{\lambda,N} \Vert}{\Vert \Psi_{\lambda,N} \Vert} \lesssim \exp(-c \alpha^{\frac{\eta}{4}})+ \Vert \mathcal{R}_{\lambda,N} \Vert_{L^{\infty}(J_{\alpha})}.
\end{align*}
Let $n=N$ in Lemma \ref{Lem Reminder Estimate}, estimate as in the proof of Lemma \ref{Lemma psi k} and employ \eqref{Eq tau bound above}, it yields that, for $N\geq 1$,
\begin{align*}
\left\vert \mathcal{R}_{\lambda,N}(x)\right\vert &\lesssim \sum_{k=0}^{3N-1} \frac{1}{\vert V_{\lambda}\vert^{\frac{k+N-3}{4}}} \sum_{j=1}^{k+N+1}\frac{\vert d_{k+N+1,j}\vert}{\vert V_{\lambda}\vert^{j}}\lesssim \sum_{k=0}^{3N-1} \frac{\tau_{\pm}(x)^{k+N+1}}{\vert V_{\lambda}\vert^{\frac{k+N-3}{4}}} \sum_{j=1}^{k+N+1}\frac{\vert V(x)\vert^{j}}{\vert V_{\lambda}\vert^{j}}\\
&\lesssim \sum_{k=0}^{3N-1} \frac{\tau_{\pm}(x)^{k+N+1}\vert V(x) \vert}{\vert V_{\lambda}\vert^{\frac{k+N+1}{4}}}\lesssim \frac{\vert V(x) \vert \tau_{\pm}(x)^{N+1}}{\alpha^{\frac{N+1}{4}}},\qquad  x\in I^{\pm}\cap J_{\lambda}^{\pm}.
\end{align*}
Similarly, since $V\in W^{N+3,\infty}_{\textup{loc}}(\R)$, we have the following estimate on the compact set $[-a_{-},a_{+}]$:
\begin{align*}
\left\vert \mathcal{R}_{\lambda,N}(x) \right\vert \lesssim \alpha^{-\frac{N+1}{4}},\qquad  x\in [-a_{-},a_{+}].
\end{align*}
Likewise, with the same reason, we have the same estimate for $N=0$. Thus, the estimate \eqref{Eq Main 1} is followed for all $N\geq 0$. 

\begin{remark}\label{Remark sign}
From the above construction,  we see that if we change the sign of $\Im V$ in condition \eqref{Assump G Im V 1} as follows
\begin{equation}\label{Assump G Diagonal 1'}
\limsup_{x\to +\infty} \Im V(x)<0 <\liminf_{x\to -\infty} \Im V(x);
\end{equation}
then the previous analysis still works. Indeed, what we need to do is just to choose the minus sign in the formula of $\psi_{-1}^{(1)}$ in \eqref{Eq Eikonal Solution}. Then, we have
\begin{equation*}
\mathrm{Re}\, \left(\lambda \psi_{-1}^{(1)}\right)=\frac{1}{2}\frac{\beta-\Im V}{\left( \vert V_{\lambda}\vert +\alpha-\Re V \right)^{1/2}\left(\vert V_{\lambda}\vert^{1/2}+\frac{1}{\sqrt{2}}\left(\vert V_{\lambda}\vert+\alpha-\Re V\right)^{1/2}\right)^{1/2}}.
\end{equation*}
By fixing $\beta_{\pm} \in \R_{0}^{+}$ such that
\[ \limsup_{x\to +\infty} \Im V(x)<-\beta_{-} \text{ and } \beta_{+} <\liminf_{x\to -\infty} \Im V(x),\]
then there exist constants $a_{\pm}>0$ such that, for all $\beta\in B=[-\beta_{-},\beta_{+}]$,
\begin{equation*}
\begin{aligned}
&\Im V(x)-\beta\lesssim  \Im V(x)\lesssim -1,\qquad && x\in I^{+}\coloneqq \left\{ x\in \R_{0}^{+}: x\geq a_{+}\right\},\\
&\Im V(x)-\beta \gtrsim \Im V(x)\gtrsim 1,\qquad && x\in I^{-}\coloneqq \left\{ x\in \R_{0}^{-}: x\leq -a_{-}\right\}.
\end{aligned}
\end{equation*} 
By repeating the procedure when proving \eqref{Eq Estimate Re -1}, we have
\begin{equation}\label{Eq lambda Re psi-1'}
\begin{aligned}
&\mathrm{Re} \,\left(\lambda\psi_{-1}^{(1)}(x)\right)\approx \frac{\beta-\Im V(x)}{\alpha^{\frac{3}{4}}}\gtrsim  - \frac{\Im V(x)}{\alpha^{\frac{3}{4}}}\qquad && \text{ on } I^{+}\cap J_{\alpha}^{+},\\
&\mathrm{Re} \,\left(\lambda\psi_{-1}^{(1)}(x)\right)\approx \frac{\beta-\Im V(x)}{\alpha^{\frac{3}{4}}}\lesssim - \frac{\Im V(x)}{\alpha^{\frac{3}{4}}}\qquad &&\text{ on } I^{-}\cap J_{\alpha}^{-}.
\end{aligned}
\end{equation}
Therefore, the pseudomode now possesses the right sign for the decay. Although all the other terms $\left(\psi_{k}^{(1)}\right)_{0\leq k \leq N-1}$ also change their signs accordingly, but it does not matter because they are all estimated with the absolute value. 
\end{remark} 
\subsection{Decaying potentials}\label{Section Decaying potentials}

We reserve this section for constructing the pseudomodes for the potentials in Example~\ref{Example Decaying}. Since the condition \ref{Assump G Im V 1} is not met by the decaying of the potentials, we can not apply directly the previous constructions. However, the shape of the pseudomodes is the same as in the beginning of Subsection \ref{Section Pseudomode estimate}, just the definition of $\delta_{\lambda}^{\pm}$ (replace for $\delta_{\alpha}^{\pm}$) should be defined differently. In the coming paragraphs, when we say $\lambda=\alpha+i\beta\to \infty$, we mean $\alpha\to+\infty$ and $\vert \beta \vert \to 0$ ($\beta$ can be zero). Let $a_{+}>0$ such that
\begin{align*}
\Im V(x)=\vert x \vert^{-\gamma},\qquad x\geq a_{+}.
\end{align*}
We seek for the boundary $\delta_{\lambda}^{+}$ of the cut-off such that the first term in the expansion very large when $\lambda \to \infty$. Since $V$ is bounded, we still have $\vert \lambda-V(x)\vert\approx \alpha$, thus, for $x>a_{+}$,
\[ \int_{a_{+}}^{x} \Re \left( \lambda \psi_{-1}^{(1)}(t) \right)\, \dd t\gtrsim \frac{1}{\alpha^{\frac{3}{4}}} \int_{a_{+}}^{x}\left[ \Im V(t)-\beta\right]\, \dd t=\frac{x^{1-\gamma}\left[\frac{1}{1-\gamma}-\beta x^{\gamma} \right]-\frac{a_{+}^{1-\gamma}}{1-\gamma}+\beta a_{+}}{\alpha^{\frac{3}{4}}}.\]
Here, in order that the inequality happens, we fixed the sign $\Im V-\beta>0$ on $[a_{+},\delta_{\lambda}^{+}]$ by assuming that $\vert \beta \vert \left(\delta_{\lambda}^{+}\right)^{\gamma}=o(1)$ as $\lambda\to \infty$. Combining this assumption for $\delta_{\lambda}^{+}$ with the expect that the right hand side of the above estimate very large, $\delta_{\lambda}^{\pm}$ should read
\begin{equation}\label{psi-1 decaying}
\alpha^{\frac{3}{4}} \left(\delta_{\lambda}^{+}\right)^{\gamma-1}+\vert \beta \vert \left(\delta_{\lambda}^{+}\right)^{\gamma}=o(1),\qquad \lambda \to \infty.
\end{equation}
The existence of finite positive number $\delta_{\lambda}^{+}$ satisfying $\displaystyle\lim_{\lambda \to \infty} \delta_{\lambda}^{+}=+\infty$ and \eqref{psi-1 decaying} is equivalent to the constraint \eqref{psi-1 decaying 2} on $\alpha$ and $\beta$. Indeed, this is due to the following inequality for all $s>0$,
$$\alpha^{\frac{3}{4}} s^{\gamma-1}+\vert \beta \vert s^{\gamma}\geq c_{\gamma} \alpha^{\frac{3}{4}\gamma}\vert \beta \vert^{1-\gamma}, \qquad c_{\gamma}=\frac{1}{\gamma^{\gamma}(1-\gamma)^{1-\gamma}},$$
and the choice of $\delta_{\lambda}^{+}$, for example, as follows
\begin{equation}\label{delta decaying}
\delta_{\lambda}^{+}=\left\{\begin{aligned}
&\alpha^{\frac{3}{4}}\vert \beta \vert^{-1} && \text{ if } \beta \neq 0,\\
&\alpha^{\frac{1}{1-\gamma}} && \text{ if } \beta = 0.
\end{aligned} \right.
\end{equation}
Step \ref{Step 1 1} of Proposition~\ref{Prop Exponential Decay} is easily to be obtained since all the estimates in \eqref{Eq Estimate Re -1} and \eqref{Eq Estimate Re k} on $[0,a_{+}]$ still hold. Since the potential $V$ still satisfy the assumption~\eqref{Assump G Der} with $\tau(x)=\left(x^2+1\right)^{-\frac{1}{2}}$, the estimate~\eqref{Eq Estimate Re k} keep being true for all $t \in \left[a_{+}, \delta_{\lambda}^{+}\right]$ and thus
\[\left\vert \sum_{k=0}^{N-1}  \mathrm{Re}\, \left(\lambda^{-k}\psi_{k}^{(1)}(t)\right) \right\vert \lesssim \sum_{k=0}^{N-1} \frac{\vert V(t)\vert \tau_{+}(t)^{k+1}}{\alpha^{\frac{k}{4}+1}} \lesssim \frac{\vert \Im V(t)\vert  }{\alpha}.\]
With the choice of $\delta_{\lambda}^{+}$ satisfying \eqref{psi-1 decaying}, we have, for all $t \in \left[a_{+}, \delta_{\lambda}^{+}\right]$,
\[\Im V(t)-\beta =\frac{1-\beta t^{\gamma}}{t^{\gamma}}\gtrsim \Im V(t). \]
Therefore, we estimate as in Step \ref{Step 2 1} of Proposition~\ref{Prop Exponential Decay}, that is the terms $\lambda^{-k} \psi_{k}^{(1)}$ for $k\geq 0$ can be neglected in the expansion of $P_{\lambda,N}$ and we also obtain \eqref{Exponential P 1} for all $x\in [a_{+}, \delta_{\lambda}^{+}]$. By choosing $\Delta_{\alpha}=\delta_{\lambda}^{+}/4$, we have for all $x\in [\delta_{\lambda}^{+}-\Delta_{\lambda}^{+},\delta_{\lambda}^{+}]$,
\begin{align*}
\int_{a_{+}}^{x} \Im V(t)\, \dd t \gtrsim \int_{\delta_{\lambda}^{+}-2\Delta_{\lambda}^{+}}^{x}  t^{-\gamma}\, \dd t\gtrsim \left(\delta_{\lambda}^{+}\right)^{1-\gamma},
\end{align*}
and thus, there exists $c_{1}>0$ such that, for all $x\in [\delta_{\lambda}^{+}-\Delta_{\lambda}^{+},\delta_{\lambda}^{+}]$,
\begin{equation}\label{Exponential Decaying}
\left\vert \exp\left(-P_{\lambda,N}(x)\right)\right\vert \lesssim \exp\left(-\frac{c_{1}}{\alpha^{\frac{3}{4}} \left(\delta_{\lambda}^{+}\right)^{\gamma-1}} \right).
\end{equation}
Thanks to \eqref{psi-1 decaying}, we know that the right hand side of \eqref{Exponential Decaying} has a decay as $\lambda\to \infty$. If we strengthen \eqref{psi-1 decaying 2} to (with some $\varepsilon>0$)
\[\vert\beta \vert \alpha^{\frac{3}{4}\frac{\gamma}{1-\gamma}+\varepsilon} =\mathcal{O}(1),\qquad \lambda \to \infty, \]
and choose $\delta_{\lambda}^{+}$ as in \eqref{delta decaying}, the decay of the right hand side of \eqref{Exponential Decaying} is given by (with some $c>0$ and $\eta>0$)
\[ \left\vert \exp\left(-P_{\lambda,N}(x)\right)\right\vert =\mathcal{O}\left( \exp\left( -c \alpha^{\eta}\right)\right).\]
The claim for $\delta_{\lambda}^{-}$, $\Delta_{\lambda}^{-}$ and the estimates on the negative axis are as same as the above. Therefore, by the same manner as Step \ref{Step 4 1}, we obtain 
\[ \frac{\Vert (\mathcal{D}_{\lambda,N} \xi) \exp(-P_{\lambda,N}) \Vert }{\Vert \xi \exp(-P_{\lambda,N})\Vert }=o(1), \qquad \lambda \to \infty.\]
Concerning the remainder $\mathcal{R}_{\lambda,N}$, since $V$ and $\tau$ are bounded on $\R$, we indeed have
\[ \Vert \mathcal{R}_{\lambda,N} \Vert_{L^{\infty}(J_{\lambda})} \lesssim \alpha^{-\frac{N+1}{4}}.\]
\section{Pseudomodes for large imagine pseudoeigenvalues}\label{Section Proof Theorem 23}
\subsection{Pseudomode construction}
Let Assumption \ref{Assumption 2} hold for some $N\in \N_{0}$ and let us define the pseudomode for Theorem~\ref{Theorem 2} as follows
\[ \Psi_{\lambda,N}:= \xi_{\lambda} \exp\left(-P_{\lambda,N} \right),\]
where 
\begin{itemize}
\item $\xi_{\lambda}$ is the cut-off function chosen such that, with $\Delta_{\beta}$ and $J_{\beta}$ defined as in \eqref{Delta beta},
\begin{equation}\label{Eq cutoff 2}
\begin{aligned}
&\xi_{\lambda} \in \mathcal{C}_{0}^{\infty}(\R_{+}), \quad 0\leq \xi_{\lambda} \leq 1,\\
&\xi_{\lambda}(x)=1, \qquad \text{ for all } x \in \left(x_{\beta}-\Delta_{\beta},x_{\beta}+\Delta_{\beta}\right)=:J_{\beta}',\\
&\xi_{\lambda}(x)=0,\qquad \text{ for all } x \notin \left(x_{\beta}-2\Delta_{\beta},x_{\beta}+2\Delta_{\beta}\right)=J_{\beta}.
\end{aligned}
\end{equation}

\item $\displaystyle P_{\lambda,N}(x) = \sum_{k=-1}^{N-1}  \int_{x_{\beta}}^{x} \lambda^{-k}\psi_{k}^{(1)}(t) \, \dd t$  in which $\psi_{-1}^{(1)}$ determined by \eqref{Eq Eikonal Solution} with the plus sign and $\left(\psi_{k}^{(1)}\right)_{k\in[[0,N-1]]}$ determined by \eqref{Eq Transport Solution}.
\end{itemize}
From \eqref{V1} and \eqref{V2}, we can deduce that
\begin{equation}\label{Comparable}
\tau(x) \approx \tau(x_{\beta}),\qquad \Im V(x) \approx \Im V(x_{\beta}), \qquad \Im V^{(1)}(x) \approx \Im V^{(1)}(x_{\beta}) ,\qquad  x\in J_{\beta}.
\end{equation}

\begin{proposition}
There exists $c>0$ such that, for all $\beta \gtrsim 1$, we have
\begin{equation}\label{quotient}
\frac{\left\Vert (\mathcal{D}_{\lambda,N} \xi) \exp(-P_{\lambda,N}) \right\Vert_{L^2\left(\R_{0}^{+}\right)} }{\left\Vert \xi \exp(-P_{\lambda,N})\right\Vert_{L^2\left(\R_{0}^{+}\right)} } \lesssim \exp\left(-c \frac{\Im V^{(1)}(x_{\beta})\tau(x_{\beta})^{-2}}{\alpha^{\frac{3}{4}}+\beta^{\frac{3}{4}}}\right).
\end{equation}
\end{proposition}
\begin{proof}
Following are the lines of steps to prove the Theorem \ref{Theorem 2}.
\begin{enumerate}[label=\textbf{(\arabic*)}]
\item  We begin the proof by showing that for sufficiently large $\beta>0$, the set of admissible $\alpha$ in \eqref{alpha 1} and \eqref{alpha 2} is non-empty.
\item Under these conditions for $\alpha$, we can show that, there exists a constant $C>0$ such that
\begin{equation}\label{numerator}
\left\Vert (\mathcal{D}_{\lambda,N} \xi) \exp(-P_{\lambda,N}) \right\Vert_{L^2\left(\R_{0}^{+}\right)}\lesssim \exp\left(-C \frac{\Im V^{(1)}(x_{\beta})\tau(x_{\beta})^{-2}}{\alpha^{\frac{3}{4}}+\beta^{\frac{3}{4}}}\right).
\end{equation}
\item Since the support $J_{\beta}$ of pseudomode makes a move as $\beta\to +\infty$, we can not bound below the denominator $\left\Vert \xi \exp\left( -P_{\lambda,N}\right) \right\Vert_{L^2\left(\R_{0}^{+}\right)}^{2}$ by a constant as in [Prop.~\ref{Prop Exponential Decay},Step~\ref{Step 1 1}]. However, we can bound it below by a small term whose inverse can be controlled by the right hand side of \eqref{numerator}. Then it yields \eqref{quotient}.
\end{enumerate}
Before entering the details of the proof, for simplifying the later computation, let us write 
$$4+\varepsilon=\frac{4}{1-\eta}$$
as in \eqref{Re Im 1} where $\eta_{1}=\eta \in (0,1)$. Then, the condition \eqref{xi120} is rewritten in terms of $\eta$ as follows
\begin{equation}\label{xi12}
\left(t_{1},t_{2}\right)\in \left[0,\frac{\eta}{5(3+\eta)}\right]^2, \qquad t_{1}-\frac{1-\eta}{4}t_{2}<\frac{\eta}{20}.
\end{equation}
Similarly, the condition \eqref{alpha 20} for $\alpha$ is revised to  
\begin{equation}\label{alpha 2}
\left[\beta\tau(x_{\beta})\right]^{\frac{4}{5}}\lesssim \vert \alpha \vert \lesssim\left[ \beta\tau(x_{\beta})^{-1}\right]^{\frac{4}{3+\eta}}.
\end{equation}
Following are the details of the proof:
\begin{enumerate}[label=\textbf{(\arabic*)}]
\item In order to show that the set of $\alpha$ satisfying \eqref{alpha 1} and \eqref{alpha 2} is non-empty, we can choose, for instance, 
\[ \alpha=\alpha(\beta):=\left[ \beta\tau(x_{\beta})^{-1}\right]^{\frac{4}{3+\eta}}.\]
Then, from \eqref{Re Im 1} which is a direct consequence of the assumption \eqref{Assump Bound Re Im}, we have
\begin{equation}\label{beta 1}
\tau(x_{\beta})^{4}\lesssim \beta^{1-\eta}, \qquad \beta \lesssim \left[ \beta\tau(x_{\beta})^{-1}\right]^{\frac{4}{3+\eta}},\qquad \left\vert \Re V(x) \right\vert^{\frac{1}{1-\eta}} \lesssim \left[ \beta\tau(x_{\beta})^{-1}\right]^{\frac{4}{3+\eta}},
\end{equation}
and thus, we deduce \eqref{alpha 2} from
\[ \left[\beta\tau(x_{\beta})\right]^{\frac{4}{5}}\lesssim \beta \lesssim \left[ \beta\tau(x_{\beta})^{-1}\right]^{\frac{4}{3+\eta}}. \]
Furthermore, \eqref{alpha 1} is also followed from \eqref{beta 1}:
\[ \vert \Re V(x)\vert\lesssim \alpha^{1-\eta},\qquad  x\in J_{\beta}.\]
Therefore, this choice of $\alpha$ satisfies \eqref{alpha 1} and \eqref{alpha 2}.
\item Following the work of Lemma~\ref{Lemma psi k} and Proposition~\ref{Prop Exponential Decay}, we firstly perform the estimate for the real part of the eikonal term $\displaystyle \int_{x_{\beta}}^{x}  \lambda \psi_{-1}^{(1)}(t)\, \dd t$ and then prove that the other transport terms $\displaystyle \int_{x_{\beta}}^{x} \lambda^{k} \psi_{k}^{(1)}(t)\, \dd t$, for $k\in [[0,N-1]]$ play less important roles than the eikonal one. Let us recall that
\begin{align*}
\mathrm{Re}\, \left(\lambda \psi_{-1}^{(1)}\right) = \frac{1}{2}\frac{\Im V-\beta}{\left( \vert V_{\lambda}\vert +\alpha-\Re V \right)^{1/2}\left(\vert V_{\lambda}\vert^{1/2}+\frac{1}{\sqrt{2}}\left(\vert V_{\lambda}\vert+\alpha-\Re V\right)^{1/2}\right)^{1/2}}.
\end{align*}
Thanks to \eqref{alpha 1} and \eqref{Comparable}, we have
\begin{align*}
\left( \vert V_{\lambda}\vert +\alpha-\Re V \right)^{1/2}\left(\vert V_{\lambda} \vert^{1/2}+\frac{1}{\sqrt{2}}\left(\vert V_{\lambda} \vert+\alpha-\Re V\right)^{1/2}\right)^{1/2}\approx \vert\alpha \vert^{\frac{3}{4}}+\beta^{\frac{3}{4}}.
\end{align*}
By observing the sign of the term $\Im V(x)-\beta$ on the left and on the right of $x_{\beta}$ on~$J_{\beta}$, it implies that, for all $x\in J_{\beta}$,
\begin{equation}\label{Psi -1 2}
\int_{x_{\beta}}^{x} \mathrm{Re}\, \left(\lambda \psi_{-1}^{(1)}(t)\right)\, \dd t \approx \frac{\displaystyle \int_{x_{\beta}}^{x} \left[\Im V(t)-\beta\right]\, \dd t}{\vert \alpha \vert^{\frac{3}{4}}+\beta^{\frac{4}{4}}}\approx \frac{\Im V^{(1)}(x_{\beta})(x-x_{\beta})^2}{\vert \alpha \vert^{\frac{3}{4}}+\beta^{\frac{4}{4}}}.
\end{equation}
Here in the last estimate, we changed variable twice in integrals and employing \eqref{Comparable}, in detail, that is, for all $x\in J_{\beta}$,
\begin{align*}
 \int_{x_{\beta}}^{x} \Im V(t)-\beta\, \dd t &=(x-x_{\beta})^2 \int_{0}^1 \int_{0}^1 \xi \Im V^{(1)}\left(x_{\beta}+ \tau \xi(x-x_{\beta}) \right)\, \dd \tau \dd \xi\\
 &\approx \Im V^{(1)}(x_{\beta})(x-x_{\beta})^2.
\end{align*}
On $J_{\beta}\setminus J_{\beta}'$, all $x$ stays away from the turning point $x_{\beta}$ a distance~$\Delta_{\beta}=\frac{\eta}{\tau(x_{\beta})}$, hence, for every $x\in J_{\beta}\setminus J_{\beta}'$,
\begin{align}
\int_{x_{\beta}}^{x} \mathrm{Re}\, \left(\lambda \psi_{-1}^{(1)}(t)\right)\, \dd t &\gtrsim \frac{\Im V^{(1)}(x_{\beta})\tau(x_{\beta})^{-2}}{\vert\alpha \vert^{\frac{3}{4}}+\beta^{\frac{3}{4}}} \nonumber\\
&\gtrsim \left\{
\begin{aligned}
&\vert \alpha \vert^{\frac{\eta}{5}}\beta^{\frac{\eta}{5(3+\eta)}-t_{1}}\tau(x_{\beta})^{t_{2}-\frac{\eta}{5(3+\eta)}} &&\text{ if } \vert\alpha \vert>\beta,\\
&\beta^{\frac{1}{4}-t_{1}}\tau(x_{\beta})^{t_{2}-1} &&\text{ if } \vert \alpha \vert \leq \beta.\\
\end{aligned}
\right.\label{Psi -1 0}
\end{align}
In the second inequality, we used \eqref{V1} and \eqref{alpha 2}. Notice that, from the assumption \eqref{xi12}, the powers of $\beta$ and $\tau$ are related by the following inequalities
\begin{align*}
\frac{\eta}{5(3+\eta)}-t_{1} >\frac{1-\eta}{4}\left(\frac{\eta}{5(3+\eta)}-t_{2}\right),\qquad \frac{1}{4}-t_{1} &> \frac{1-\eta}{4}\left(1-t_{2}\right)+\frac{\eta}{5}.
\end{align*}
Combine this with the first inequality in \eqref{beta 1} and the fact~$0\leq \frac{\eta}{5(3+\eta)}-t_{2}<1-t_{2}$, we obtain, for all $x\in J_{\beta}\setminus J_{\beta}'$,
\begin{equation}\label{Psi -1 1}
\int_{x_{\beta}}^{x} \mathrm{Re}\, \left(\lambda \psi_{-1}^{(1)}(t)\right)\, \dd t\gtrsim \frac{\Im V^{(1)}(x_{\beta})\tau(x_{\beta})^{-2}}{\vert\alpha \vert^{\frac{3}{4}}+\beta^{\frac{3}{4}}}\gtrsim \left\{ \begin{aligned}
&\vert \alpha \vert^{\frac{\eta}{5}}&&\text{ if } \vert\alpha \vert>\beta,\\
&\beta^{\frac{\eta}{5}} &&\text{ if } \vert \alpha \vert \leq \beta.
\end{aligned}
\right.
\end{equation}
Next, in the same manner of proving~\eqref{Eq Estimate Re k}, by the choice of $\alpha$ in \eqref{alpha 1} and \eqref{alpha 2} collaborating with the first inequality in \eqref{beta 1}, we can show that, for $k\in [[-1,N-1]]$ and $m\in [[1,3]]$ such that $k+m\geq 1$, and for all $t\in J_{\beta}$,
\begin{align}\label{Psi k}
\left\vert \lambda^{-k}\psi_{k}^{(m)}(t) \right\vert  &\lesssim  \sum_{j=1}^{k+m} \frac{\tau^{k+m}(t)\vert V(t) \vert^{j}}{\vert V_{\lambda}\vert^{j+\frac{k}{4}}}\lesssim \frac{\tau(x_{\beta})^{k+m}}{\vert \alpha \vert^{\frac{k}{4}}}  \sum_{j=1}^{k+m} \left(1+\frac{\beta}{\vert\alpha \vert}\right)^{j}\nonumber\\
&\lesssim \frac{\tau(x_{\beta})^{k+m}}{\vert \alpha \vert^{\frac{k}{4}}}\left(1+\frac{\beta}{\vert \alpha \vert}\right)^{k+m}\nonumber\\
&\lesssim \left\{
\begin{aligned}
&\left[\frac{\tau(x_{\beta})}{\beta^{\frac{1}{4}}} \right]^{k}\tau(x_{\beta})^{m} &&\text{ if } \vert \alpha \vert> \beta,\\
& \left[\frac{\beta \tau(x_{\beta})}{\vert \alpha \vert^{\frac{5}{4}}}\right]^{k+\frac{4m}{5}}  [\beta \tau(x_{\beta})]^{\frac{m}{5}} &&\text{ if } \vert \alpha \vert \leq \beta,
\end{aligned}
\right.\nonumber\\
&\lesssim \left\{
\begin{aligned}
&\beta^{-\frac{k\eta}{4}}\tau(x_{\beta})^{m} &&\text{ if } \vert \alpha \vert> \beta,\\
&  [\beta \tau(x_{\beta})]^{\frac{m}{5}} &&\text{ if } \vert \alpha \vert \leq \beta.
\end{aligned}
\right.
\end{align}
In particular, for $k\geq 0$ and at $m=1$, by employing \eqref{Psi -1 1} for the case $\vert \alpha\vert>\beta$ and \eqref{Psi -1 0} for the other case, we obtain, for all $x\in J_{\beta}\setminus J_{\beta}'$,
\begin{equation*}
\begin{aligned}
\frac{ \displaystyle \left\vert\int_{x_{\beta}}^{x}  \lambda^{-k}\psi_{k}^{(1)}(t)\, \dd t \right\vert }{\displaystyle \frac{\Im V^{(1)}(x_{\beta})\tau(x_{\beta})^{-2}}{\alpha^{\frac{3}{4}}+\beta^{\frac{3}{4}}}}  &\lesssim \left\{
\begin{aligned}
& \vert \alpha\vert^{-\frac{\eta}{5}}\qquad &&\text{ if } \vert \alpha \vert> \beta,\\
& \left[\beta^{\frac{1}{20}-t_{1}}\tau(x_{\beta})^{t_{2}-\frac{1}{5}}\right]^{-1}\qquad &&\text{ if } \vert \alpha \vert \leq \beta.
\end{aligned}
\right.\\
\end{aligned}
\end{equation*}
Furthermore, in the case $\vert \alpha\vert\leq \beta$, we notice that
\begin{equation}\label{xi122}
\frac{1}{20}-t_{1}=\frac{1-\eta}{4}\left(\frac{1}{5}-t_{2}\right)+\frac{\eta}{20}-\left(t_{1}-\frac{1-\eta}{4}t_{2}\right).
\end{equation}
Accordingly, for all $k\in [[0,N-1]]$ and for all $x\in J_{\beta}\setminus J_{\beta}'$, we get
\begin{align}
\frac{ \displaystyle \left\vert\int_{x_{\beta}}^{x}  \lambda^{-k}\psi_{k}^{(1)}(t)\, \dd t \right\vert }{\displaystyle \int_{x_{\beta}}^{x} \mathrm{Re}\,(\lambda\psi_{-1}^{(1)}(t)) \, \dd t} &\lesssim \frac{ \displaystyle \left\vert\int_{x_{\beta}}^{x}  \lambda^{-k}\psi_{k}^{(1)}(t)\, \dd t \right\vert }{\displaystyle \frac{\Im V^{(1)}(x_{\beta})\tau(x_{\beta})^{-2}}{\alpha^{\frac{3}{4}}+\beta^{\frac{3}{4}}}} \nonumber\\
&\lesssim \left\{
\begin{aligned}
& \vert \alpha\vert^{-\frac{\eta}{5}} &&\text{ if } \vert \alpha \vert> \beta,\\
& \beta^{-\left[ \frac{\eta}{20}-\left(t_{1}-\frac{1-\eta}{4}t_{2}\right)\right]} &&\text{ if } \vert \alpha \vert \leq \beta.
\end{aligned}
\right. \label{Psi k 2}
\end{align}
Therefore, for all $x\in J_{\beta}\setminus J_{\beta}'$, we obtain (with some constant $C_{1}>0$)
\[\left\vert \exp\left(-P_{\lambda,N}(x)\right)\right\vert\lesssim \exp\left(-C_{1} \frac{\Im V^{(1)}(x_{\beta})\tau(x_{\beta})^{-2}}{\alpha^{\frac{3}{4}}+\beta^{\frac{3}{4}}}\right).\]
By using \eqref{Psi k} and the fact $\left\vert \lambda \psi_{-1}^{(1)} \right\vert = \left\vert V_{\lambda} \right\vert^{\frac{1}{4}}\lesssim \vert \alpha \vert^{\frac{1}{4}}+\beta^{\frac{1}{4}}$, we have, for all $x\in J_{\beta}$,
\[\left\vert P_{\lambda,N}^{(m)}(x)\right\vert \lesssim \left\{\begin{aligned}
&\vert \alpha\vert^{\frac{m}{4}} &&\text{ if } \vert \alpha \vert>\beta,\\
&\beta^{\frac{m}{4}} &&\text{ if } \vert \alpha \vert\leq \beta.
\end{aligned}
\right.  \]
With the help of \eqref{Psi -1 1}, we can control all appearing polynomial terms in $D_{\lambda,N}\xi$ to get the estimate \eqref{numerator}.
\item In this step, we will check that $\left\Vert \xi \exp\left( -P_{\lambda,N}\right) \right\Vert_{L^2\left(\R_{0}^{+}\right)}$ is not too small. To do that, we set $\widetilde{\Delta}_{\beta}=\tau(x_{\beta})^{-\frac{4}{5}}\beta^{-\frac{1}{20}}$. Then, by \eqref{beta 1}, we have $\widetilde{\Delta}_{\beta}<\Delta_{\beta}$ and thus
\begin{align*}
\left\Vert \xi \exp\left( -P_{\lambda,N}\right) \right\Vert_{L^2\left(\R_{0}^{+}\right)}^{2}&\geq \int_{x_{\beta}}^{x_{\beta}+\widetilde{\Delta}_{\beta}}\exp\left(-2\sum_{k=-1}^{N-1}  \left\vert\int_{x_{\beta}}^{x} \Re \left(\lambda^{-k}\psi_{k}^{(1)}(t)\right)\, \dd t \right\vert\right)\, \dd x\\
&\geq \widetilde{\Delta}_{\beta} \exp\left(-2\sum_{k=-1}^{N-1}  \int_{x_{\beta}}^{x_{\beta}+\widetilde{\Delta}_{\beta}} \left\vert\Re \left(\lambda^{-k}\psi_{k}^{(1)}(t)\right)\right\vert\, \dd t \right).
\end{align*}
For the integral of $\Re \left( \lambda \psi_{-1}\right)$, we make use of \eqref{Psi -1 2} for $\vert \alpha\vert>\beta$ and the fact $ \left\vert \Re \left(\lambda \psi_{-1}^{(1)}\right)\right\vert\leq  \left\vert V_{\lambda}\right\vert^{\frac{3}{4}}\lesssim \beta^{\frac{3}{4}}$ for $\vert \alpha\vert\leq \beta$, then we have
\begin{equation*}
\int_{x_{\beta}}^{x_{\beta}+\widetilde{\Delta}_{\beta}} \left\vert \Re \left(\lambda \psi_{-1}^{(1)}(t)\right)\right\vert\, \dd t \lesssim \left\{
\begin{aligned}
&\vert \alpha \vert^{-\frac{3}{4}}\Im V^{(1)}(x_{\beta})\left(\widetilde{\Delta}_{\beta}\right)^2 &&\text{ if } \vert \alpha \vert>\beta,\\
&\beta^{\frac{1}{4}}\widetilde{\Delta}_{\beta} &&\text{ if } \vert \alpha \vert\leq\beta.
\end{aligned}
\right.
\end{equation*}
Thus, by the definition of $\widetilde{\Delta}_{\beta}$ and \eqref{xi122} combining with the first inequality in \eqref{beta 1}, we obtain
\begin{align*}
\frac{\displaystyle\int_{x_{\beta}}^{x_{\beta}+\widetilde{\Delta}_{\beta}} \left\vert \Re \left(\lambda \psi_{-1}^{(1)}(t)\right)\right\vert\, \dd t }{\displaystyle\frac{\Im V^{(1)}(x_{\beta})\tau(x_{\beta})^{-2}}{\alpha^{\frac{3}{4}}+\beta^{\frac{3}{4}}}}&\lesssim \left\{
\begin{aligned}
&\left[\tau(x_{\beta})^4 \beta^{-1}\right]^{\frac{1}{10}} &&\text{ if } \vert \alpha \vert> \beta,\\
&\left[\beta^{\frac{1}{20}-t_{1}}\tau(x_{\beta})^{t_{2}-\frac{1}{5}}\right]^{-1} &&\text{ if } \vert \alpha \vert \leq \beta,
\end{aligned}
\right.\\
&\lesssim \left\{
\begin{aligned}
&\beta^{-\frac{\eta}{10}} &&\text{ if } \vert \alpha \vert> \beta,\\
&\beta^{-\left[ \frac{\eta}{20}-\left(t_{1}-\frac{1-\eta}{4}t_{2}\right)\right]} &&\text{ if } \vert \alpha \vert \leq \beta.
\end{aligned}
\right.
\end{align*}
For $k\in [[0,N-1]]$, we use \eqref{Psi k 2} to get
\begin{align*}
\frac{\displaystyle \int_{x_{\beta}}^{x_{\beta}+\widetilde{\Delta}_{\beta}}\left\vert \Re \left(\lambda^{-k}\psi_{k}^{(1)}(t)\right)\right\vert\, \dd t }{\displaystyle\frac{\Im V^{(1)}(x_{\beta})\tau(x_{\beta})^{-2}}{\alpha^{\frac{3}{4}}+\beta^{\frac{3}{4}}}}\lesssim \left\{
\begin{aligned}
& \vert \alpha\vert^{-\frac{\eta}{5}} &&\text{ if } \vert \alpha \vert> \beta,\\
& \beta^{-\left[ \frac{\eta}{20}-\left(t_{1}-\frac{1-\eta}{4}t_{2}\right)\right]} &&\text{ if } \vert \alpha \vert \leq \beta.
\end{aligned}
\right.
\end{align*}
From the above estimates, we have shown that
\[ \sum_{k=-1}^{N-1}  \int_{x_{\beta}}^{x_{\beta}+\widetilde{\Delta}_{\beta}} \left\vert\Re \left(\lambda^{-k}\psi_{k}^{(1)}(t)\right)\right\vert\, \dd t =o\left( \frac{\Im V^{(1)}(x_{\beta})\tau(x_{\beta})^{-2}}{\alpha^{\frac{3}{4}}+\beta^{\frac{3}{4}}}\right),\qquad \beta \to +\infty.\]
Therefore, we have
\begin{align*}
\frac{\left\Vert (\mathcal{D}_{\lambda,N} \xi) \exp(-P_{\lambda,N}) \right\Vert_{L^2\left(\R_{0}^{+}\right)} }{\left\Vert \xi \exp(-P_{\lambda,N})\right\Vert_{L^2\left(\R_{0}^{+}\right)} } &\lesssim \widetilde{\Delta}_{\beta}^{-\frac{1}{2}}\exp\left(-\left(C-o(1)\right) \frac{\Im V^{(1)}(x_{\beta})\tau(x_{\beta})^{-2}}{\alpha^{\frac{3}{4}}+\beta^{\frac{3}{4}}}\right).
\end{align*}
Thus, \eqref{quotient} is followed directly by using \eqref{beta 1} and \eqref{Psi -1 1} to control the term $\widetilde{\Delta}_{\beta}^{-\frac{1}{2}}$.
\end{enumerate}
\end{proof}
\subsection{Remainder estimate (Proof of Theorem \ref{Theorem 2})}
By using the same trick as proving \eqref{Eq v assymptotic}, we can show that
\begin{equation}\label{V comparible}
\vert V(x)\vert \approx \vert V(x_{\beta}) \vert, \qquad  x\in J_{\beta}.
\end{equation}
Indeed, since $V$ is a complex-valued function, we should be careful with some steps when we perform estimation, more precisely, for all $h\in \R$ such that $\vert h \vert\leq 2\Delta_{\beta}$, we have
\begin{align*}
\left\vert \ln \frac{\vert V(x_{\beta}+h) \vert}{\vert V(x_{\beta}) \vert}\right\vert \leq \left\vert \textup{Log } V(x_{\beta}+h)-  \textup{Log } V(x_{\beta}) \right\vert= \left\vert \int_{V(x_{\beta})}^{V(x_{\beta}+h)} \frac{\dd z}{z} \right\vert=\left\vert\int_{x_{\beta}}^{x_{\beta}+h} \frac{V'(t)}{V(t)}\, \dd t\right\vert\lesssim 1,
\end{align*}
in which $\textup{Log } z$ is the principal branch of the logarithmic function   with its antiderivative $\frac{1}{z}$ and keep in mind that the range $V(J_{\beta})$ stays away the negative semi-axis $\R_{0}^{-}$, and the last inequality is estimated as same as proving \eqref{Eq v assymptotic} by using \eqref{Assump G Der}.

We finish the proof by estimating the remainder whose shape is given in Lemma \ref{Lem Reminder Estimate}, for $N\geq 1$ and for all $x\in J_{\beta}$,
\begin{align*}
\left\vert \mathcal{R}_{\lambda, N}(x)\right\vert &\lesssim\sum_{k=0}^{3N-1}\frac{1}{\left\vert V_{\lambda}(x)\right\vert^{\frac{k+N-3}{4}}} \sum_{j=1}^{k+N+1}\frac{\left\vert d_{k+N+1,j}(x)\right\vert}{\left\vert V_{\lambda}(x)\right\vert^{j}}\\
&\lesssim\sum_{k=0}^{3N-1}\sum_{j=1}^{k+N+1}\frac{\tau(x)^{k+N+1} \left\vert V(x)\right\vert^{j}}{\left\vert V_{\lambda}(x)\right\vert^{\frac{k+N-3}{4}+j}}\\
&\lesssim\sum_{k=0}^{3N-1}\sum_{j=1}^{k+N+1}\frac{\tau(x_{\beta})^{k+N+1}\left(\left\vert \Re V(x_{\beta})\right\vert^{j}+\beta^{j}\right)}{\left\vert \alpha\right\vert^{\frac{k+N-3}{4}+j}}.
\end{align*}
Here, we employed \eqref{Assump G Der} to control $d_{k+N+1,j}$, \eqref{alpha 1} for $V_{\lambda}$ in the denominator, \eqref{beta 1} for function $\tau$ and \eqref{V comparible} for $V$. The remainder $\mathcal{R}_{\lambda,0}$ is estimated analogously.

\subsection{Pseudomode for semi-classical biharmonic operator (Proof of Theorem \ref{Theorem 3})}
Since $\Re z-\Re W(x_{0})=\mu>0$, there exists an interval centered at $x_{0}$ included in $I$, denoted as $J\coloneqq (x_{0}-2\Delta,x_{0}+2\Delta)\subset I$, such that the function $\Re z-\Re W(t)$ is positive in $J$.  Without loss of generality, we assume further that $\Im W(x)-\Im W(x_{0})<0$ for all $x\in (x_{0}-2\Delta,x_{0})$ and $\Im W(x)-\Im W(x_{0})>0$ for all $x\in (x_{0},x_{0}+2\Delta)$. Let us write our semi-classical problem in our previous setting by factoring the parameter $h^{4}$ out
\[ H_{h}-z = h^{4}\left(\frac{\dd^4}{\dd x^4}+ V_{h}(x)-\lambda_{h}\right),\]
where $V_{h}(x):=h^{-4}W(x)$ and $\lambda_{h}:=h^{-4}z$. Being inspired by the above analysis, the pseudomode is achieved around the point $x_{0}$ satisfying the equation $\Im V_{h}(x_{0})=\Im \lambda_{h}$, \emph{i.e.}~$\Im W(x_{0})= \Im z$. We should keep in mind that the point $x_{0}$ here is fixed. We set up the pseudomode as follows 
\[ \Psi_{h,N} = \xi \exp\left(-P_{h,N} \right),\]
in which
\begin{itemize}
\item $\xi\in  \mathcal{C}_{0}^{\infty}(\R)$ is a cut-off function which is equal $1$ on  $J'\coloneqq \left(x_{0}-\Delta,x_{0}+\Delta\right)$ and equal $0$ in the complement of $J$ in $\R$,
\item $\displaystyle P_{h,N}(x) \coloneqq  \sum_{k=-1}^{N-1}  \int_{x_{0}}^{x} \lambda_{h}^{-k}\psi_{k}^{(1)}(t) \, \dd t$ where $\psi_{-1}^{(1)}$ is determined by \eqref{Eq Eikonal Solution} with the plus sign and $\left(\psi_{k}^{(1)}\right)_{k\in[[0,N-1]]}$ is determined by \eqref{Eq Transport Solution}, in which $V$ is replaced by $V_{h}$ and $\lambda$ is replaced by $\lambda_{h}$. If the above sign of $\Im W(x)-\Im W(x_{0})$ changes inversely on the left and on the right of $x_{0}$, we merely choose the minus sign in \eqref{Eq Eikonal Solution}.
\end{itemize}

Notice that
\[ \lambda_{h}- V_{h}(t)=h^{-4}\left[\left( \Re z- \Re W(t)\right)+i\left(\Im W(x_{0})-\Im W(t)\right)\right],\]
the function $\left(\lambda_{h}-V_{h}\right)^{\frac{1}{4}}$ is well-defined and so are all functions  $\psi_{k}^{(1)}$, for $k\in [[-1,N-1]]$ on~$J$. By the formula of $\mathrm{Re}\, \left(\lambda_{h} \psi_{-1}^{(1)}(t)\right)$ in \eqref{Re psi -1}, we obtain, for all $t\geq x_{0}$,
\begin{align*}
\mathrm{Re}\, \left(\lambda_{h} \psi_{-1}^{(1)}(t)\right) \gtrsim h^{-1}\frac{\Im W(t)-\Im W(x_{0})}{\left(\Re z- \Re W(t) \right)^{\frac{3}{4}}+\vert \Im W(t)-\Im W(x_{0})\vert^{\frac{3}{4}}}, \qquad  t \in (x_{0},x_{0}+2\Delta)
\end{align*}
and
\begin{align*}
\mathrm{Re}\, \left(\lambda_{h} \psi_{-1}^{(1)}(t)\right) \lesssim h^{-1}\frac{\Im W(t)-\Im W(x_{0})}{\left(\Re z- \Re W(t) \right)^{\frac{3}{4}}+\vert \Im W(t)-\Im W(x_{0})\vert^{\frac{3}{4}}}, \qquad  t \in (x_{0}-2\Delta,x_{0}).
\end{align*}
On $J\setminus J'$, the function $\int_{x_{0}}^{x} \left(\Im W(t)-\Im W(x_{0})\right)\, \dd t$ is bounded below by a positive constant, we have
\[ \int_{x_{0}}^{x} \mathrm{Re}\, \left(\lambda_{h} \psi_{-1}^{(1)}(t)\right) \, \dd t\gtrsim h^{-1}.\]
Thanks to the shape of the WKB solutions in Lemma \ref{Lem Transport Solution}, it can be seen that, for all $k\geq -1$ and for all $t\in J$,
\[ \left\vert \lambda_{h}^{-k} \psi_{k}^{(1)}(t)\right\vert \lesssim h^{k}.\]
It yields that the transport terms $\left(\psi_{k}^{(1)}\right)_{k\geq 0}$ are harmless in the expansion of pseudomode and thus there exists $c_{1}>0$ such that, for all $x\in J\setminus J'$,
\[ \left\vert \exp\left( -P_{h,N}(x)\right) \right\vert \lesssim \exp(-c_{1} h^{-1}). \]
By considering on the fixed support $J$, we also obtain (with some $c_{2}>0$)
$$\left\Vert (\mathcal{D}_{\lambda_{h},N} \xi) \exp(-P_{h,N}) \right\Vert_{L^2\left(\R\right)}\lesssim \exp(-c_{2} h^{-1}).$$
Let $\widetilde{\Delta}=h \Delta$ with $h<1$, we have, for all $x\in [x_{0},x_{0}+\widetilde{\Delta}]\subset J'$,
\[ \left\vert P_{h,N}(x) \right\vert \leq \sum_{k=-1}^{N-1} \widetilde{\Delta} h^{k} \lesssim 1.\]
Thus, as in [Prop.~\ref{Prop Exponential Decay}, Step~\ref{Step 1 1}], it implies that
\[ \Vert \Psi_{h}\Vert \gtrsim 1.\]
Concerning the remainder, we have, for all $N\geq 0$ and $x\in J$,
\begin{align*}
h^{4}\left\vert \mathcal{R}_{\lambda, N}(x)\right\vert \lesssim h^{N+1}.
\end{align*}
The conclusion of Theorem \ref{Theorem 3} is followed.

\appendix
\section{The solutions of the transport equations and the WKB~remainder}\label{Appendix}
This appendix is devoted to the proofs of Lemma \ref{Lem Transport Solution} and Lemma \ref{Lem Reminder Estimate}. These lemmata describe the structure of the transport solutions $\left(\psi_{k}^{(1)}\right)_{k\in [[-1,n-1]]}$ and the WKB remainder $\mathcal{R}_{\lambda,n}$ which are made from the elements of $D_{r,j}$ defined in Notation \ref{Nota Derivative}. First, let us recall here some simple observations about $D_{r,j}$ which are mentioned in~\cite[\textbf{Appendix A}]{KS19}.
\begin{enumerate}[label=\textbf{(\arabic*)}]
\item \label{Appendix R1} If $r\geq 1$, then $D_{r,0}=\{0\}$. If $r<j$, then $D_{r,j}=\{0\}$.
\item \label{Appendix R2} $D_{0,0}=\C$.
\item \label{Appendix R3} $D_{r,j}+D_{r,j} = D_{r,j}$.
\item \label{Appendix R4} $cD_{r,j}= D_{j}^{r,s}$ for any constant $c\in \C$.
\item \label{Appendix R5} $D_{r_1,j_{1}} D_{r_2,j_{2}} \subset D_{r_1+r_2,j_1+j_2}$.
\end{enumerate}

\begin{proof}[Proof of Lemma \ref{Lem Transport Solution}]
The induction method is employed to prove the following statement with respect to the index $k\in [[-1,n-1]]$:
\begin{equation}\label{Eq Inductive State}
\psi_{k}^{(m)} \in \frac{\lambda^{k}}{V_{\lambda}^{\frac{k}{4}}} \sum_{j=0}^{k+m} \frac{D_{k+m,j}}{V_{\lambda}^{j}},\qquad m\in [[1,n+3-k]].
\end{equation}
\textbf{Base step}: We check that the statement is true for $k=-1$. It is a direct application of Fa\`{a}~di~Bruno's formula for the high derivative of the composition of the functions $f(x):=x^{\frac{1}{4}}$ and $g(x):=V_{\lambda}(x)$, for $\ell\in \N_{1}$,
\[ \frac{\dd^{\ell}}{\dd x^{\ell}} f(g(x)) = \sum_{j=1}^{\ell} f^{(j)}(g(x))B_{\ell,j}\left(g^{(1)}(x),g^{(2)}(x),\ldots,g^{(\ell-j+1)}(x)\right).\]
Here $B_{\ell,j}$ are Bell polynomials which have formulae
\[ B_{\ell,j}\left(x_1,x_2,\ldots,x_{\ell-j+1}\right)=\sum_{\alpha\in I_{\ell.j}} \frac{\ell!}{\alpha_{1}!\alpha_{2}!\ldots\alpha_{\ell-j+1}!}\left( \frac{x_1}{1!}\right)^{\alpha_{1}}\left( \frac{x_2}{2!}\right)^{\alpha_{2}}\ldots\left( \frac{x_{\ell-j+1}}{(\ell-j+1)!}\right)^{\alpha_{\ell-j+1}},\]
where $\alpha=(\alpha_{1},\alpha_{2},\ldots,\alpha_{\ell-j+1})$ and $I_{\ell,j}$ is the set defined in \eqref{Notation Irj}. It is not difficult to see that $f^{(j)}(x) = c_{j} x^{\frac{1}{4}-j}$ for some $c_{j}\in \R$ and 
$$B_{\ell,j}\left(g^{(1)}(x),g^{(2)}(x),\ldots,g^{(\ell-j+1)}(x)\right)\in D_{\ell,j}.$$
By plugging these objects in the formula of $\frac{\dd^{\ell} }{\dd x^{\ell}}\psi_{-1}^{(1)}=i\lambda^{-1}\frac{\dd^{\ell} }{\dd x^{\ell}} V_{\lambda}^{\frac{1}{4}}$, regarding the rule \ref{Appendix R1}, it implies that, for all $\ell\in \N_{1}$,
\[ \psi_{-1}^{(\ell+1)} \in \frac{\lambda^{-1}}{V_{\lambda}^{-\frac{1}{4}}} \sum_{j=0}^{\ell} \frac{D_{\ell,j}}{V_{\lambda}^{j}}.\]
Thanks to the rule \ref{Appendix R2}, the claim in the case $k=-1$ is confirmed, for all $m\in \N_{1}$,
\[ \psi_{-1}^{(m)} \in \frac{\lambda^{-1}}{V_{\lambda}^{-\frac{1}{4}}} \sum_{j=0}^{m-1} \frac{D_{m-1,j}}{V_{\lambda}^{j}}.\]
Since $V\in W^{n+3,2}_{\mathrm{loc}}(\R)$, the maximal order of the derivative that $\psi_{-1}$ can be taken is $n+4$, in which $V^{(n+3)}$ appears in $D_{n+3,1}$ (see Notation \ref{Nota Derivative}).\\
\textbf{Inductive step}: Let $q\in [[-1,n-2]]$, we assume that \eqref{Eq Inductive State} holds for all $k\in [[-1,q]]$, we need to show that
\begin{equation}\label{Eq Inductive k+1}
\psi_{q+1}^{(m)} \in \frac{\lambda^{q+1}}{V_{\lambda}^{\frac{q+1}{4}}} \sum_{j=0}^{q+m+1} \frac{D_{q+1+m,j}}{V_{\lambda}^{j}},\qquad m\in [[1,n+2-q]].
\end{equation}
From the formula \eqref{Eq Transport Solution}, by using the Leibniz product rule, we obtain
\begin{align*}
\psi_{q+1}^{(m)} = &\left( \psi_{q+1}^{(1)} \right)^{(m-1)}=\sum_{i=0}^{m-1}\begin{pmatrix}
 m-1\\
 i
\end{pmatrix} \left[\frac{1}{4\left(\psi_{-1}^{(1)}\right)^3}\right]^{(m-1-i)}\times\\
  &\hspace{2.7 cm}\left[\psi_{q-2}^{(4+i)}-4 \sum_{\alpha_{1}+\alpha_{2}=q-2} \left(\psi_{\alpha_{1}}^{(1)}\psi_{\alpha_{2}}^{(3)}\right)^{(i)}-3\sum_{\alpha_{1}+\alpha_{2}=q-2}\left(\psi_{\alpha_{1}}^{(2)}\psi_{\alpha_{2}}^{(2)}\right)^{(i)}\right.\nonumber \\
&\left. \hspace{2cm} +6\sum_{\alpha_{1}+\alpha_{2}+\alpha_{3}=q-2} \left(\psi_{\alpha_{1}}^{(1)}\psi_{\alpha_{2}}^{(1)}\psi_{\alpha_{3}}^{(2)}\right)^{(i)}- \sum_{\substack{\alpha_{1}+\alpha_{2}+\alpha_{3}+\alpha_{4}=q-2\\ \alpha_{1},\alpha_{2},\alpha_{3},\alpha_{4} \neq q+1}} \left(\psi_{\alpha_{1}}^{(1)}\psi_{\alpha_{2}}^{(1)}\psi_{\alpha_{3}}^{(1)}\psi_{\alpha_{4}}^{(1)}\right)^{(i)}\right].
\end{align*}
We compute each term appearing in the above formula:
\begin{enumerate}[label=\textbf{\alph*)}]
\item For the derivatives of $\frac{1}{\left(\psi_{-1}^{(1)}\right)^3}$: we do the same manner as computing the derivatives of $\psi_{-1}^{(1)}$ in the base step by considering the functions $f(x)=x^{-\frac{3}{4}}$ and $g(x)=V_{\lambda}$. Since $f^{(j)}(x)= c_{j} x^{-\frac{3}{4}-j}$ for some $c_j\in \R$, we obtain
\begin{equation}\label{Eq Inverse psi-1}
\left( \frac{1}{\left(\psi_{-1}^{(1)}\right)^3}\right)^{(\ell)} \in \frac{\lambda^3}{V_{\lambda}^{\frac{3}{4}}} \sum_{j=0}^{\ell} \frac{D_{\ell,j}}{V_{\lambda}^{j}},\qquad \ell \in [[0, n+3]].
\end{equation}
\item For the derivatives of $\psi_{q-2}$: for each $m\in [[1,n+2-q]]$, $i\in[[0,m-1]]$, we have $4+i \in[[1, n+3-(q-2)]]$. It means that we can use the induction assumption for the derivatives of $\psi_{q-2}$:
\begin{equation}\label{Eq One alpha}
\psi_{q-2}^{(4+i)} \in \frac{\lambda^{q-2}}{V_{\lambda}^{\frac{q-2}{4}}} \sum_{j=0}^{q-2+4+i} \frac{D_{q-2+4+i,j}}{V_{\lambda}^{j}}=\frac{\lambda^{q-2}}{V_{\lambda}^{\frac{q-2}{4}}} \sum_{j=0}^{q+2+i} \frac{D_{q+2+i,j}}{V_{\lambda}^{j}}.
\end{equation}
\item The derivatives of $\psi_{\alpha_{1}}^{(1)}\psi_{\alpha_{2}}^{(3)}$ and $\psi_{\alpha_{1}}^{(2)}\psi_{\alpha_{2}}^{(2)}$ where $\alpha_1+ \alpha_2=q-2$. By the Leibniz product rule again, we have, for each $i\in [[0,m-1]]$,
\begin{align*}
\left(\psi_{\alpha_{1}}^{(1)}\psi_{\alpha_{2}}^{(3)}\right)^{(i)} = \sum_{\ell=0}^{i} \begin{pmatrix}
i\\
\ell
\end{pmatrix} \psi_{\alpha_{1}}^{(\ell+1)}\psi_{\alpha_{2}}^{(i-\ell+3)}.
\end{align*}
For each  $m\in [[1,n+2-q]]$, $i\in [[0,m-1]]$ and $\ell\in [[0,i]]$, we notice that for $\alpha_{1}\geq -1$, $\alpha_{2}\geq-1$ and $\alpha_{1}+\alpha_{2}=q-2$, we have $\ell+1 \in [[1, n+3-\alpha_{1}]]$ and $i-\ell+3\in [[1,n+3-\alpha_{2}]]$. This enables us to use the induction assumption to rewrite both $\psi_{\alpha_{1}}^{(\ell+1)} $ and $\psi_{\alpha_{2}}^{(i-\ell+3)}$. Then, we arrive at
\begin{align*}
\psi_{\alpha_{1}}^{(\ell+1)}\psi_{\alpha_{2}}^{(i-\ell+3)} &\in \left(\frac{\lambda^{\alpha_{1}}}{V_{\lambda}^{\frac{\alpha_{1}}{4}}} \sum_{j_{1}=0}^{\alpha_{1}+\ell+1} \frac{D_{\alpha_{1}+\ell+1,j_{1}}}{V_{\lambda}^{j_{1}}}\right) \left( \frac{\lambda^{\alpha_{2}}}{V_{\lambda}^{\frac{\alpha_{2}}{4}}} \sum_{j_{2}=0}^{\alpha_{2}+i-\ell+3} \frac{D_{\alpha_{2}+i-\ell+3,j_2}}{V_{\lambda}^{j_2}}\right)\\
&= \frac{\lambda^{\alpha_{1}+\alpha_{2}}}{V_{\lambda}^{\frac{\alpha_{1}+\alpha_{2}}{4}}} \sum_{j=0}^{\alpha_{1}+\alpha_{2}+i+4}\sum_{q_1+q_2=j} \frac{D_{\alpha_{1}+\ell+1,q_1}}{V_{\lambda}^{q_1}}\frac{D_{\alpha_{2}+i-\ell+3,q_2}}{V_{\lambda}^{q_2}}\\
&\subset\frac{\lambda^{q-2}}{V_{\lambda}^{\frac{q-2}{4}}} \sum_{j=0}^{q+2+i}\sum_{q_1+q_2=j} \frac{D_{q+2+i,j}}{V_{\lambda}^{j}}\\
&=\frac{\lambda^{q-2}}{V_{\lambda}^{\frac{q-2}{4}}} \sum_{j=0}^{q+2+i}\frac{D_{q+2+i,j}}{V_{\lambda}^{j}},
\end{align*}
in which we used
\begin{itemize}
\item[-] the rule \ref{Appendix R1} for the first equality, that is $D_{\alpha_{1}+\ell+1,q_1}=\{0\}$ and $D_{\alpha_{2}+i-\ell+3,q_2}=\{0\}$ when $q_1>\alpha_{1}+\ell+1$ and if $q_2> \alpha_{2}+i-\ell+3$,
\item[-] the rule \ref{Appendix R5} and $\alpha_{1}+\alpha_{2}=q-2$ for the inclusion,
\item[-] the rule \ref{Appendix R3} for the second equality.
\end{itemize}
Since the resulting terms do not depend on $\ell$ any more, by the rule \ref{Appendix R3} and \ref{Appendix R4}, it implies that
\begin{equation}\label{Eq Two alpha 1}
\left(\psi_{\alpha_{1}}^{(1)}\psi_{\alpha_{2}}^{(3)}\right)^{(i)} \in  \sum_{\ell=0}^{i} \begin{pmatrix}
i\\
\ell
\end{pmatrix} \frac{\lambda^{q-2}}{V_{\lambda}^{\frac{q-2}{4}}} \sum_{j=0}^{q+i+2}\sum_{q_1+q_2=j} \frac{D_{q+2+i,j}}{V_{\lambda}^{j}} =\frac{\lambda^{q-2}}{V_{\lambda}^{\frac{q-2}{4}}} \sum_{j=0}^{q+2+i}\frac{D_{q+2+i,j}}{V_{\lambda}^{j}}.
\end{equation}
\item In the same manner as above, we have
\begin{equation}\label{Eq other alpha}
\begin{aligned}
\left(\psi_{\alpha_{1}}^{(2)}\psi_{\alpha_{2}}^{(2)}\right)^{(i)} &\in \frac{\lambda^{k-2}}{V_{\lambda}^{\frac{k-2}{4}}} \sum_{j=0}^{q+2+i}\frac{D_{q+2+i,j}}{V_{\lambda}^{j}},\\
\left(\psi_{\alpha_{1}}^{(1)}\psi_{\alpha_{2}}^{(1)}\psi_{\alpha_{3}}^{(2)}\right)^{(i)}&\in \frac{\lambda^{q-2}}{V_{\lambda}^{\frac{q-2}{4}}} \sum_{j=0}^{q+2+i}\frac{D_{q+2+i,j}}{V_{\lambda}^{j}},\\
\left(\psi_{\alpha_{1}}^{(1)}\psi_{\alpha_{2}}^{(1)}\psi_{\alpha_{3}}^{(1)}\psi_{\alpha_{4}}^{(1)}\right)^{(i)} &\in \frac{\lambda^{q-2}}{V_{\lambda}^{\frac{q-2}{4}}} \sum_{j=0}^{q+2+i}\frac{D_{q+2+i,j}}{V_{\lambda}^{j}}.
\end{aligned}
\end{equation}
\end{enumerate}
To finish the proof, we put \eqref{Eq Inverse psi-1}, \eqref{Eq One alpha}, \eqref{Eq Two alpha 1} and \eqref{Eq other alpha} together into the formula of $\psi_{q+1}^{(m)}$, we obtain
\begin{align*}
\psi_{q+1}^{(m)}\in &\sum_{i=0}^{m-1} \left(\frac{\lambda^{3}}{V_{\lambda}^{\frac{3}{4}}} \sum_{j_{1}=0}^{m-1-i}\frac{D_{m-1-i,j_1}}{V_{\lambda}^{j_1}}\right)\left(\frac{\lambda^{q-2}}{V_{\lambda}^{\frac{q-2}{4}}} \sum_{j_2=0}^{q+2+i} \frac{D_{q+2+i,j_2}}{V_{\lambda}^{j_2}} \right)\\
= & \sum_{i=0}^{m-1} \frac{\lambda^{q+1}}{V_{\lambda}^{\frac{q+1}{4}}} \sum_{j=0}^{q+m+1} \sum_{q_{1}+q_{2}=j} \frac{D_{m-1-i,q_1}}{V_{\lambda}^{q_1}} \frac{D_{q+2+i,q_2}}{V_{\lambda}^{q_2}}\\
\subset & \sum_{i=0}^{m-1} \frac{\lambda^{q+1}}{V_{\lambda}^{\frac{q+1}{4}}} \sum_{j=0}^{q+m+1} \sum_{q_{1}+q_{2}=j} \frac{D_{q+m+1,q_1+q_2}}{V_{\lambda}^{q_1+q_2}} \\
=&\frac{\lambda^{q+1}}{V_{\lambda}^{\frac{q+1}{4}}} \sum_{j=0}^{q+m+1}\frac{D_{q+m+1,j}}{V_{\lambda}^{j}},
\end{align*}
where we employed the rule \ref{Appendix R3} for the first membership, the rule \ref{Appendix R1} for the second equality as proving in \eqref{Eq Two alpha 1}, the rule \ref{Appendix R5} for the inclusion and the rule \ref{Appendix R3} for the final equality.
Therefore, the inductive claim in \eqref{Eq Inductive k+1} is proved.
\end{proof}

\begin{proof}[Proof of Lemma \eqref{Lem Reminder Estimate}]
When $n=0$, the conclusion of the Lemma follows from \eqref{Remainder 0} and the fact that $\psi_{-1}^{(1)}=\lambda^{-1}i V_{\lambda}^{\frac{1}{4}}$. When $n>0$, thanks to \eqref{Eq Phi k+3} and \eqref{Eq Reminder R}, the reminder $\mathcal{R}_{\lambda,n}$ can be written in the following way 
\begin{align*}
\mathcal{R}_{\lambda,n} =& \sum_{k=n-3}^{4n-4}\lambda^{-k}\phi_{k+3}\\
=&\sum_{k=n-3}^{4n-4}\lambda^{-k} \left( - \psi_{k}^{(4)}+4 \sum_{\substack{\alpha_{1}+\alpha_{2}=k\\-1\leq \alpha_1,\alpha_2\leq n-1}} \psi_{\alpha_{1}}^{(1)}\psi_{\alpha_{2}}^{(3)}+3\sum_{\substack{\alpha_{1}+\alpha_{2}=k\\-1\leq \alpha_1,\alpha_2\leq n-1}} \psi_{\alpha_{1}}^{(2)}\psi_{\alpha_{2}}^{(2)}\right.\\
&\left.\hspace{2 cm} -6  \sum_{\substack{\alpha_{1}+\alpha_{2}+\alpha_{3}=k\\-1\leq \alpha_1,\alpha_2,\alpha_3\leq n-1}} \psi_{\alpha_{1}}^{(1)}\psi_{\alpha_{2}}^{(1)}\psi_{\alpha_{3}}^{(2)}+ \sum_{\substack{\alpha_{1}+\alpha_{2}+\alpha_3+\alpha_4=k\\-1\leq \alpha_1,\alpha_2,\alpha_3,\alpha_4\leq n-1}} \psi_{\alpha_{1}}^{(1)}\psi_{\alpha_{2}}^{(1)}\psi_{\alpha_{3}}^{(1)}\psi_{\alpha_{4}}^{(1)}\right).
\end{align*}
By using the fact that $\psi_{k}=0$ for all $k>n-1$, we can reduce the indices of the sums as follows
\begin{align*}
\mathcal{R}_{\lambda,n}=&-\underbrace{\sum_{k=n-3}^{n-1}\lambda^{-k} \psi_{k}^{(4)}}_{S_{1}} +4\underbrace{\sum_{k=n-3}^{2n-2}\lambda^{-k}\sum_{\substack{\alpha_{1}+\alpha_{2}=k\\-1\leq \alpha_1,\alpha_2\leq n-1}} \psi_{\alpha_{1}}^{(1)}\psi_{\alpha_{2}}^{(3)}}_{S_{2}}+3\underbrace{\sum_{k=n-3}^{2n-2}\lambda^{-k} \sum_{\substack{\alpha_{1}+\alpha_{2}=k\\-1\leq \alpha_1,\alpha_2\leq n-1}} \psi_{\alpha_{1}}^{(2)}\psi_{\alpha_{2}}^{(2)}}_{S_3}\\
&-6 \underbrace{\sum_{k=n-3}^{3n-3} \lambda^{-k}\sum_{\substack{\alpha_{1}+\alpha_{2}+\alpha_{3}=k\\-1\leq \alpha_1,\alpha_2,\alpha_3\leq n-1}} \psi_{\alpha_{1}}^{(1)}\psi_{\alpha_{2}}^{(1)}\psi_{\alpha_{3}}^{(2)}}_{S_{4}}+ \underbrace{\sum_{k=n-3}^{4n-4}\lambda^{-k}\sum_{\substack{\alpha_{1}+\alpha_{2}+\alpha_3+\alpha_4=k\\-1\leq \alpha_1,\alpha_2,\alpha_3,\alpha_4\leq n-1}} \psi_{\alpha_{1}}^{(1)}\psi_{\alpha_{2}}^{(1)}\psi_{\alpha_{3}}^{(1)}\psi_{\alpha_{4}}^{(1)}}_{S_{5}}.
\end{align*}
For each $k\in [[-1,n-1]]$, Lemma \ref{Lem Transport Solution} shows us that the maximal possible derivative of $V$ in $\psi_{k}^{(1)}$ is $k+1$. Therefore, the maximal possible derivative of $V$ in $\mathcal{R}_{\lambda,n}$ that may appear in~$\psi_{n-1}^{(4)}$ is $n+3$. 

Using Lemma \ref{Lem Transport Solution}, we obtain
\begin{align*}
S_{1} \in  \sum_{k=n-3}^{n-1} \frac{1}{V_{\lambda}^{\frac{k}{4}}}\sum_{j=0}^{k+4}  \frac{D_{k+4,j}}{V_{\lambda}^{j}}=\sum_{k=0}^{2} \frac{1}{V_{\lambda}^{\frac{k+n-3}{4}}} \sum_{j=0}^{k+n+1} \frac{D_{k+n+1,j}}{V_{\lambda}^{j}}.
\end{align*} 
In order to estimate $S_{m}$ for $m\in [[2,5]]$, we do the same trick as proving \eqref{Eq other alpha}. For example, to estimate $S_{2}$, we do as follows
\begin{align*}
S_{2}\in &\sum_{k=n-3}^{2n-2}\lambda^{-k}\sum_{\substack{\alpha_{1}+\alpha_{2}=k\\-1\leq \alpha_1,\alpha_2\leq n-1}} \left(\frac{\lambda^{\alpha_{1}}}{V_{\lambda}^{\frac{\alpha_{1}}{4}}} \sum_{j_{1}=0}^{\alpha_{1}+1}\frac{D_{\alpha_{1}+1,j_{1}}}{V_{\lambda}^{j_{1}}} \right)\left(\frac{\lambda^{\alpha_{2}}}{V_{\lambda}^{\frac{\alpha_{2}}{4}}} \sum_{j_{2}=0}^{\alpha_{2}+3}\frac{D_{\alpha_{2}+3,j_{2}}}{V_{\lambda}^{j_{2}}} \right)\\
=&\sum_{k=n-3}^{2n-2} \frac{1}{V_{\lambda}^{\frac{k}{4}}} \sum_{\substack{\alpha_{1}+\alpha_{2}=k\\-1\leq \alpha_1,\alpha_2\leq n-1}}\sum_{j=0}^{\alpha_{1}+\alpha_{2}+4}\sum_{q_1+q_2=j}\frac{D_{\alpha_{1}+1,q_{1}}}{V_{\lambda}^{q_{1}}} \frac{D_{\alpha_{2}+3,q_{2}}}{V_{\lambda}^{q_{2}}}\\
\subset &  \sum_{k=n-3}^{2n-2} \frac{1}{V_{\lambda}^{\frac{k}{4}}} \sum_{j=0}^{k+4}\frac{D_{k+4,j}}{V_{\lambda}^{j}}=\sum_{k=0}^{n+1}\frac{1}{V_{\lambda}^{\frac{k+n-3}{4}}} \sum_{j=0}^{k+n+1}\frac{D_{k+n+1,j}}{V_{\lambda}^{j}}.
\end{align*}  
Similarly, we obtain
\begin{align*}
&S_{3}\in \sum_{k=0}^{n+1}\frac{1}{V_{\lambda}^{\frac{k+n-3}{4}}} \sum_{j=0}^{k+n+1}\frac{D_{k+n+1,j}}{V_{\lambda}^{j}},\\
&S_{4}\in \sum_{k=0}^{2n}\frac{1}{V_{\lambda}^{\frac{k+n-3}{4}}} \sum_{j=0}^{k+n+1}\frac{D_{k+n+1,j}}{V_{\lambda}^{j}},\\
&S_{5}\in \sum_{k=0}^{3n-1}\frac{1}{V_{\lambda}^{\frac{k+n-3}{4}}} \sum_{j=0}^{k+n+1}\frac{D_{k+n+1,j}}{V_{\lambda}^{j}}.
\end{align*} 
Therefore, thanks to the rules \ref{Appendix R1} (for $j=0$ and $r\geq 1$, $D_{r,j}=\{0\}$) and \ref{Appendix R3}, the conclusion of Lemma \ref{Lem Reminder Estimate} is obtained.
\end{proof}
\section*{Acknowledgement}
I would like to thank Professor David Krej\v{c}i\v{r}\'{i}k for introducing and encouraging me to study this great subject. Especially, I am very grateful to him for giving me many precious opportunities to continue my research career. This project was supported by the EXPRO grant number 20-17749X of the Czech Science Foundation (GA\v{C}R).
\bibliographystyle{abbrv}
\bibliography{Ref21}
\end{document}